\documentclass[a4paper,10pt]{article}

\title{Two-round Ramsey games on random graphs}

\author{
  Yahav Alon\thanks{
    School of Mathematical Sciences,
    Tel Aviv University, Tel Aviv 6997801, Israel.
    Email: \texttt{yahavalo@tauex.tau.ac.il}.
  }
  \and
  Patrick Morris\thanks{
    Departament de Matem\`atiques, Universitat Polit\`ecnica de Catalunya (UPC), Carrer de Pau Gargallo 14, 08028  Barcelona, Spain.
    Email: \texttt{pmorrismaths@gmail.com}.
    Research supported in part by the Deutsche Forschungsgemeinschaft (DFG, German Research
    Foundation) Walter Benjamin program - project number 504502205.
  }
  \and
  Wojciech Samotij\thanks{
    School of Mathematical Sciences,
    Tel Aviv University, Tel Aviv 6997801, Israel.
    Email: \texttt{samotij@tauex.tau.ac.il}.
    Research supported in part by the European Research Council Consolidator Grant 101044123 (RandomHypGra) and by the Israel Science Foundation grant 2110/22.
  }
}

\usepackage{tikz}
\usepackage{amsmath,amsthm, amssymb,latexsym, amsfonts}
\usepackage{mathtools, thm-restate}
\usepackage{verbatim}
\usetikzlibrary{arrows,
                petri,
                topaths}
\usepackage[position=top]{subfig}

\usepackage{caption}
\usepackage{mleftright}
\usepackage{dsfont}
\usepackage{bm}

\usepackage{enumitem}

\usepackage{hyperref}

\usepackage[english,noresetcount,norelsize]{algorithm2e}

\usepackage[total={17cm,26cm}]{geometry}

\begin{document}
\maketitle
\newtheorem{lemma}{Lemma}[section]
\newtheorem{corol}{Corollary}[section]
\newtheorem{thmtool}{Theorem}[section]
\newtheorem{corollary}[thmtool]{Corollary}
\newtheorem{lem}[thmtool]{Lemma}
\newtheorem{defi}[thmtool]{Definition}
\newtheorem{prop}[thmtool]{Proposition}
\newtheorem{clm}[thmtool]{Claim}
\newtheorem{conjecture}{Conjecture}
\newtheorem{problem}{Problem}

\theoremstyle{definition}
\newtheorem{rem}[thmtool]{Remark}

\newcommand{\Proof}{\noindent{\bf Proof.}\ \ }
\newcommand{\Remarks}{\noindent{\bf Remarks:}\ \ }
\newcommand{\Remark}{\noindent{\bf Remark:}\ \ }

\newcommand{\Cbad}{\mathcal{C}_{\textrm{bad}}}
\newcommand{\Ccores}{\mathcal{C}^*}

\newcommand{\eps}{\varepsilon}
\newcommand{\cA}{\mathcal{A}}
\newcommand{\cH}{\mathcal{H}}
\newcommand{\cI}{\mathcal{I}}
\newcommand{\cP}{\mathcal{P}}
\newcommand{\cS}{\mathcal{S}}
\newcommand{\cT}{\mathcal{T}}
\newcommand{\cQ}{\mathcal{Q}}
\newcommand{\cC}{\mathcal{C}}
\newcommand{\cW}{\mathcal{W}}
\newcommand{\cF}{\mathcal{F}}
\newcommand{\cK}{\mathcal{K}}
\newcommand{\cB}{\mathcal{B}}
\newcommand{\N}{\mathbb{N}}
\newcommand{\vS}{\mathbf{S}}
\newcommand{\cX}{\mathcal{X}}
\newcommand{\cY}{\mathcal{Y}}
\newcommand{\cZ}{\mathcal{Z}}
\newcommand{\Ex}{\mathbb{E}}

\def\EE{\mathbb E}
\def\PP{\mathbb P}
\def\NN{\mathbb N}

\newcommand{\sig}{\mathrm{sig}}

\newcommand{\Bin}{\mathrm{Bin}}
\newcommand{\red}{\mathrm{red}}
\newcommand{\blue}{\mathrm{blue}}

\let\phi=\varphi

\newcommand{\mybinom}[2]{
    \mleft(
    \begin{array}{@{}c@{\,}} #1\\#2 \end{array}
    \mright)}

\newenvironment{claimproof}{
\let\origqed=\qedsymbol
\renewcommand{\qedsymbol}{$\blacktriangleleft$}
\begin{proof}}{\end{proof}\let\qedsymbol=\origqed}

\newcommand{\MDF}{\mathrm{MDF}}

\newcommand{\WS}[1]{{\color{red} \small{(WS: #1)}}}

\newcommand{\PM}[1]{{\color{blue} \small{(PM: #1)}}}

\begin{abstract}
  Motivated by the investigation of sharpness of thresholds for Ramsey properties in random graphs, Friedgut, Kohayakawa, R\"odl, Ruci\'nski and Tetali introduced two variants of a single-player game whose goal is to colour the edges of a~random graph, in an online fashion, so as not to create a monochromatic triangle.  In the two-round variant of the game, the player is first asked to find a triangle-free colouring  of the edges of a random graph $G_1$ and then extend this colouring to a triangle-free colouring of the union of $G_1$ and another (independent) random graph $G_2$, which is disclosed to the player only after they have coloured $G_1$.  Friedgut et al.\ analysed this variant of the online Ramsey game in two instances: when $G_1$ has $\Theta(n^{4/3})$ edges and when the number of edges of $G_1$ is just below the threshold above which a random graph typically no longer admits a triangle-free colouring, which is located at $\Theta(n^{3/2})$.

  The two-round Ramsey game has been recently revisited by Conlon, Das, Lee and M\'esz\'aros, who generalised the result of Friedgut at al.\ from triangles to all strictly $2$-balanced graphs.
  We extend the work of Friedgut et al.\ in an orthogonal direction and analyse the triangle case of the two-round Ramsey game at all intermediate densities.  More precisely, for every $n^{-4/3} \ll p \ll n^{-1/2}$, with the exception of $p = \Theta(n^{-3/5})$, we determine the threshold density $q$ at which it becomes impossible to extend any triangle-free colouring of a typical $G_1 \sim G_{n,p}$ to a triangle-free colouring of the union of $G_1$ and $G_2 \sim G_{n,q}$.  An interesting aspect of our result is that this threshold density $q$ `jumps' by a polynomial quantity as $p$ crosses a `critical' window around $n^{-3/5}$.
\end{abstract}

\section{Introduction} \label{sec-intro}

Given graphs $G$ and $H$, we say that $G$ is \emph{$H$-Ramsey} if any red/blue-colouring of the edges of $G$ results in a~monochromatic copy of $H$.
The classical theorem of Ramsey~\cite{R30}, from which the term \emph{Ramsey theory} stems, implies that $K_n$ is $H$-Ramsey for all $n$ large enough in terms of $H$.
It is natural to ask what other graphs $G$ are also $H$-Ramsey and a prominent theme has been to explore the
existence of Ramsey graphs $G$ that are \emph{sparse}; see, for example, \cite{NR84} and the references therein.
One famous example is the work of Frankl and
R\"odl~\cite{frankl1986large}, who constructed $K_3$-Ramsey graphs that are $K_4$-free by considering sparse random graphs.
This prompted {\L}uczak, Ruci\'nski and Voigt~\cite{luczak1992ramsey} to initiate the systematic study
   of thresholds for Ramsey properties in random graphs, which has
   since become a prominent theme in probabilistic combinatorics. In particular, this pair of papers established the following. 
(Here and throughout, we denote by $G_{n,p}$ the binomial random graph with $n$ vertices and edge probability $p$ and, for a graph property $P$, we say that $P$ holds \emph{asymptotically almost surely}, a.a.s.\ for short, in $G_{n,p}$ if the probability that $P$ holds tends to $1$ as $n$ tends to infinity.)

\begin{thmtool}[\cite{frankl1986large,luczak1992ramsey}]
  \label{thm:triangle ramsey threshold}
  There exist constants $0<c<C$ such that:
  \begin{enumerate}[label=(\alph*)]
    \item \label{triangle ramsey 0 statement}
    If $p_0\leq cn^{-1/2}$, then a.a.s.\ $G_{n,p_0}$ is \emph{not} $K_3$-Ramsey. 
     \item \label{triangle ramsey 1 statement}
       If $p_1\geq Cn^{-1/2}$, then a.a.s.\ $G_{n,p_1}$ \emph{is}  $K_3$-Ramsey. 
    \end{enumerate}
\end{thmtool}

The $1$-statement~\ref{triangle ramsey 1 statement} was implicit in the work of Frankl and R\"odl~\cite{frankl1986large}; {\L}uczak, Ruci\'nski and Voigt~\cite{luczak1992ramsey} proved the $0$-statement~\ref{triangle ramsey 0 statement} and provided an alternative proof of~\ref{triangle ramsey 1 statement}.  The study of Ramsey properties of random graphs culminated in a seminal series of papers by R\"odl and Ruci\'nski~\cite{rodl1993lower,rodl1994random,rodl1995threshold}, who greatly generalised Theorem~\ref{thm:triangle ramsey threshold}, establishing thresholds  for any graph $H$ and also any number of colours $2\leq r\in \mathbb{N}$. Recently, Nenadov and Steger~\cite{NenSte16} provided a short proof of~\ref{triangle ramsey 1 statement} (and the analogous $1$-statements for all $H$ and number of colours) by using the method of hypergraph containers~\cite{balogh2015independent,SaxTho15}; see Section~\ref{sec:containers} for more on this.

\subsection{Ramsey games on random graphs}
\label{sec:games}

The pioneering work~\cite{frankl1986large,luczak1992ramsey,rodl1993lower,rodl1994random,rodl1995threshold} on Ramsey properties of random graphs has been highly influential, with many extensions and variations being studied. In this paper, we will focus on \emph{Ramsey games} played on random graphs, a viewpoint introduced   by   Friedgut, Kohayakawa, R\"odl, Ruci\'nski and Tetali~\cite{friedgut2003ramsey}. The starting point of their work was to view Theorem~\ref{thm:triangle ramsey threshold} as a single-player game played against a random source. In this game, a player is presented with a set of $M$ random edges of $K_n$, for some $1\leq M\leq \binom{n}{2}$, and asked to 
give a \textit{$H$-free
colouring} of the edges, that is, a colouring avoiding monochromatic copies of $H$.   Standard results on the asymptotic equivalence of random graph models (see, for example~\cite[Section 1.4]{randomgraphbook}), along with Theorem~\ref{thm:triangle ramsey threshold}, show that, in the case that $H=K_3$, the player will asymptotically almost surely fail when $M\geq Cn^{3/2}$ and succeed when $M\leq c n^{3/2}$, for appropriately chosen $c,C>0$. In~\cite{friedgut2003ramsey}, the authors introduced the following two variants on this game:

\medskip
\noindent
\textbf{The online game.}
The player is presented with edges of $K_n$ one at a time, according to a uniformly random permutation. Upon seeing each edge, the player must colour the edge with only the knowledge of the previous (already coloured) edges. The game ends when a monochromatic $H$ occurs and the aim of the player is to last as long as possible. 

\medskip
\noindent
\textbf{The two-round game.}
Here, the player is given two random graphs: a uniformly random $n$-vertex graph $G_1$ with $M_1$ edges and a second, \emph{independent} random graph $G_2$ with $M_2$ edges. The player must colour the first random graph, avoiding monochromatic copies of $H$ and with no knowledge of the second random graph. The player is then presented with the second random graph and asked to extend their colouring of $G_1$ to an $H$-free colouring of $G_1\cup G_2$.

\medskip
The work of Friedgut et al.~\cite{friedgut2003ramsey} formalised these games, studied the case where $H=K_3$, and drew interesting connections between the games, the original random Ramsey problem and other related research directions.
In particular, the two-round game arose naturally in another work of a subset of the authors~\cite{friedgut2006sharp} that established that the property of being $K_3$-Ramsey has a~sharp threshold in $G_{n,p}$.

With regards to the online game, the authors of~\cite{friedgut2003ramsey} gave a simple argument showing that, when played with two colours, the game typically ends with a graph having $\Theta(n^{4/3}$) edges.
More precisely, for any\footnote{Here, and throughout, for real-valued functions $f=f(n)$ and $g=g(n)$, we write $f\ll g$ (or $g\gg f$) to denote that $f/g$ tends to $0$ as $n$ tends to infinity.} $M=M(n)\ll n^{4/3}$, a.a.s.\ the player has a strategy to colour the edges avoiding a monochromatic $K_3$. Moreover for any $M\gg n^{4/3}$, a.a.s.\ the player will be forced to create a monochromatic $K_3$ while colouring the first $M$ edges. Interestingly, as discussed in more detail below, the two-round game is used to prove this upper bound on the  running time of the online game. 
There has been a considerable amount of work~\cite{BalBut10,MarSpoSte09-I,MarSpoSte09-II,Noe17} generalising this result and establishing the expected running time of the online game under optimal play. However, a full understanding for all graphs $H$ and number of colours remains elusive; see~\cite{Noe17} and the references therein for the currently best known bounds. 

In the context of the two-round game with respect to $K_3$ and with two colours, Theorem~\ref{thm:triangle ramsey threshold} implies that, when $M_1\geq C n^{3/2}$ for a large enough $C$, a.a.s.\ the player will not be able to survive even the first round, as any colouring of $G_1$ will induce a monochromatic $K_3$. Friedgut, Kohayakawa, R\"odl, Ruci\'nski and Tetali~\cite{friedgut2003ramsey} explored what happens near this extreme, when $M_1=cn^{3/2}$ for some small constant $c>0$, as well as when $M_1=\Theta(n^{4/3})$. They established the results in Theorem \ref{thm:initial two-round results} below. Here and throughout, if $G_1$ is a random graph, we use the term \textit{$G_1$-measurable colouring} to refer to a colouring of the edges of $G_1$ that is determined by the outcome of~$G_1$;  one can think of a $G_1$-measurable colouring as a strategy for colouring the first graph in the two-round game.
Throughout the paper, given two independent random graphs $G_1$ and $G_2$, we shall write that \emph{a.a.s.\ there exists a $G_1$-measurable red/blue-colouring that can be extended to a $K_3$-free colouring of $G_1 \cup G_2$} to mean that there is a deterministic map (strategy) that assigns to every graph $G$ a red/blue-colouring $\varphi = \varphi_G $ of its edges such that, with probability $1-o(1)$ in the random choice of $G_1$,
\begin{equation}
  \label{eq:some-G1-measurable-extends}
  \Pr(\text{$\varphi_{G_1}$ can be extended to a $K_3$-free colouring of $G_1 \cup G_2$} \mid G_1, \varphi_{G_1}) = 1-o(1),
\end{equation}
where this second probability is in the random choice of $G_2$.
Similarly, we shall write that \emph{a.a.s.\ no $G_1$-measurable red/blue-colouring can be extended to a $K_3$-free colouring of $G_1 \cup G_2$} to mean that, for every such deterministic map $\varphi$, with probability $1-o(1)$ in the random choice of $G_1$,
\begin{equation}
  \label{eq:no-G1-measurable-extends}
  \Pr(\text{$\varphi_{G_1}$ can be extended to a $K_3$-free colouring of $G_1 \cup G_2$} \mid G_1, \varphi_{G_1}) = o(1).
\end{equation}
At times, we will restrict our attention to colourings of $G_1$ with a certain property $P$.
We shall say that a $G_1$-measurable colouring $\varphi$ has property $P$ if a.a.s.\ the outcome of $G\sim G_1$ is such that $\varphi_G$ has property $P$.
Thus saying that \emph{a.a.s.\ no $G_1$-measurable colouring with property $P$ can be extended to a $K_3$-free colouring of $G_1 \cup G_2$} has the same meaning as \eqref{eq:no-G1-measurable-extends} except that we now only consider the maps $\phi$ that have property $P$ with probability $1-o(1)$ (over the choice of $G_1$).
Similarly, saying that \emph{a.a.s.\ any $G_1$-measurable colouring with property $P$ can be extended to a $K_3$-free colouring of $G_1 \cup G_2$} means that any $G_1$-measurable colouring $\varphi$ that a.a.s.\ has property $P$ satisfies \eqref{eq:some-G1-measurable-extends}. 

The following theorem thus asserts that in the cases covered, a.a.s.\ no strategy will work to be able to colour both random graphs without monochromatic triangles.

\begin{thmtool}[\cite{friedgut2003ramsey}] \label{thm:initial two-round results} Let $G_1$ be a uniformly random graph on $n$ vertices with $M_1$ edges and $G_2$ an independent random graph on $n$ vertices with $M_2$ edges. Suppose either
\begin{enumerate}[label=(\alph*)]
    \item \label{two-round near threshold}
     $M_1=cn^{3/2}$ for some $c>0$ and $M_2\gg 1$; or
     \item \label{two-round bottom range}
       $M_1=cn^{4/3}$ for some $c>0$ and $M_2\gg n^{4/3}$. 
    \end{enumerate}
    Then a.a.s.\ no $G_1$-measurable  red/blue-colouring  can be extended to a $K_3$-free red/blue-colouring of $G_1\cup G_2$.
\end{thmtool}

Theorem~\ref{thm:initial two-round results}~\ref{two-round near threshold} shows that just below the threshold for the $K_3$-Ramsey property, the random graph $G$ is a.a.s.\ very close to being $K_3$-Ramsey in the sense that, although $G$ does admit a $K_3$-free colouring, no such colouring can be extended after one adds to $G$ some $\omega(1)$ random edges.

Theorem~\ref{thm:initial two-round results}~\ref{two-round bottom range} highlights the connection between the online game and the two-round game, as it implies that a.a.s.\ the online game cannot last $M\gg n^{4/3}$ steps. Indeed, even if we allow the player of the online game a `grace period' and do not ask for any colouring until $n^{4/3}$ edges are revealed, a.a.s.\ no matter how the player chooses to colour these, they will not be able to extend to the next $M-n^{4/3}\gg n^{4/3}$ random edges. 

In contrast to the online game, the two-round game has not been further explored until the recent work of Conlon, Das, Lee and M\'esz\'aros~\cite{ConDasLeeMes20}. They investigated to what extent Theorem~\ref{thm:initial two-round results}~\ref{two-round near threshold} can be extended to two-round games with respect to other graphs $H$.  Answering a question from~\cite{friedgut2003ramsey}, they showed that a large family of graphs $H$ (namely strictly 2-balanced graphs, see~\cite{ConDasLeeMes20} for a definition) exhibit the same behaviour as $K_3$ in  that just below the threshold for the $H$-Ramsey property a.a.s.\ all $H$-free colourings can be killed by adding a super-constant number of random edges. They also showed that this is not the case for all graphs $H$ and posed the interesting question as to what properties of $H$ determine this behaviour. Finally, we mention that both \cite{friedgut2003ramsey} and \cite{ConDasLeeMes20} explore the two-round game with three colours near the Ramsey threshold. 

\subsection{Our results} \label{sec:our results}

The aim of our work here is to  expand on the work of \cite{friedgut2003ramsey} on the two-round game in a different direction. We keep our focus on $H=K_3$ and two colours and investigate the outcome  of the two-round game as the number of edges in each round is varied. 
We will state and prove our results in the setting of binomial random graphs, as opposed to uniform graphs with a fixed number of edges.  It is well-known that these models are asymptotically equivalent~\cite[Section 1.4]{randomgraphbook}, but the independence of the binomial model makes it more convenient to work with. In order to systematically study two-round games on random graphs at different densities, we introduce the following notion of a \emph{Ramsey completion threshold}, which captures the critical density of the second graph at which the probability that the player succeeds in extending their $H$-free colouring from the first to the second graph jumps from $1-o(1)$ to $o(1)$.

\begin{defi} \label{def:Ramsey completion threshold}
  Given some probability $p=p(n)\in [0,1]$ and a graph $H$, we say that $q=q(n)$ is a \emph{Ramsey completion threshold} for $H$ with respect to $p$ if the following holds for $G_1 \sim G_{n,p}$:
  \begin{enumerate}[label=(\alph*)]
  \item \label{def:0 statement}  If  $q_0 \ll q$ and $G_2\sim G_{n,q_0}$,  then a.a.s.\ there exists a $G_1$-measurable red/blue-colouring that can be extended to an $H$-free colouring of $G_1\cup G_2$.
    \item \label{def:1 statement} If  $q_1 \gg q$ and $G_2\sim G_{n,q_1}$, then a.a.s.\  no $G_1$-measurable  red/blue-colouring  can be extended to an $H$-free colouring of $G_1\cup G_2$.
   \end{enumerate}
   If such a completion threshold exists, we denote it by\footnote{As usual, we abuse notation here as such a threshold will not be determined uniquely, but rather up to constants.} $q(n;H,p)$. If all red/blue-colourings of $G_{n,p}$ contain a monochromatic $H$ with probability $\Omega(1)$, we set~$q(n;H,p)=0$.
 \end{defi}

We refer to~\ref{def:0 statement} as the \emph{$0$-statement} of the definition and~\ref{def:1 statement} as the \emph{$1$-statement}.  Observe that Theorem~\ref{thm:triangle ramsey threshold} gives that there exists $C>0$ such that $q(n;K_3,p)=0$ for all $p\geq Cn^{-1/2}$. Moreover, Theorem~\ref{thm:initial two-round results}~\ref{two-round near threshold} gives  for \emph{any} constant $c>0$, if $p=cn^{-1/2}$ and $q(n;K_3,p)\neq 0$, then $q(n;K_3,p)=n^{-2}$ whereas Theorem~\ref{thm:initial two-round results}~\ref{two-round bottom range} gives that\footnote{In fact they only prove the 1-statement~\ref{def:1 statement} of Definition~\ref{def:Ramsey completion threshold} but the corresponding $0$-statement also holds and essentially follows from the analysis of the online game in~\cite{friedgut2003ramsey}, see Remark~\ref{rem:0-statement proof}.}  $q(n;K_3,p)=n^{-2/3}$ when $p=\Theta(n^{-2/3})$. Here, we complete the picture for almost all intermediate values of $p$.

\begin{thmtool}\label{thm:main}
We have that 
\[
  q(n;K_3,p) =
  \begin{cases}
    n^{-6}p^{-8} & \text{if } n^{-3/5} \ll p \ll n^{-1/2}; \qquad \text{(upper range)}  \\
    n^{-3}p^{-7/2} & \text{if }   n^{-2/3} \ll p \ll n^{-3/5}. \qquad \text{(lower range)}    
  \end{cases}
\]
\end{thmtool}

An interesting aspect of this result is that there is a `jump' in the completion threshold at around $n^{-3/5}$.  Indeed, when $p = n^{-3/5}$, then $n^{-3}p^{-7/2} = n^{-9/10}$ whereas $n^{-6}p^{-8} = n^{-6/5}$.  
We refer to the values of $p$ larger than $n^{-3/5}$ as the \emph{upper range} and those smaller than $n^{-3/5}$ as the \emph{lower range}. Determining the behaviour of $q(n;K_3,p)$ when $p=\Theta(n^{-3/5})$  remains an intriguing open question, as does exploring the range  $0< p\ll n^{-2/3}$ where the second graph becomes denser than the first. 
Finally, it would be interesting to determine the behaviour of the two-round game for different $H$ as the densities of the two random graphs vary. In particular, it is far from clear what properties of $H$ could determine this behaviour; Theorem~\ref{thm:main} does not suggest any obvious conjecture. 

Our proof of Theorem~\ref{thm:main} incorporates several different approaches to capture the different behaviour occurring at different densities. One particular feature that we would like to highlight is a novel use of the `discharging method' to prove the existence of a desired colouring of the first graph in the lower range (see the proof of Lemma~\ref{lem:collagescolorable}). This method is inspired by an argument in recent work of Friedgut, Kuperwasser, Schacht and the third author~\cite{friedgut2022sharp} that establishes sufficient conditions for sharpness of thresholds for various Ramsey properties.
Here, we build on this general idea of using discharging to find `easily colourable configurations' in graphs of small density, but we adopt a more involved discharging scheme catered to our purposes. We believe that this method may find further applications in the study of Ramsey properties of graphs and other discrete structures and our work here demonstrates its flexibility.

\paragraph*{Organisation.}
The proof of Theorem~\ref{thm:main} naturally splits into four parts.  In Section~\ref{sec:0-statements}, we establish the $0$-statement, that is, the lower bound on $q(n; K_3, p)$.  The (shorter) argument for the upper range is presented in Section~\ref{sec:0 upper} and the (longer) argument for the lower range -- in Section~\ref{sec:0 lower}.  In Section~\ref{sec:1-statements}, we establish the $1$-statement.  The $1$-statement---the upper bound on $q(n; K_3, p)$---for the (easier) lower range will be proved in Section~\ref{sec:1 lower} and for the (harder) upper range -- in Section~\ref{sec:1 upper}.  Before embarking on this, we collect the relevant notation and tools in Section~\ref{sec:prelims}.

\paragraph*{Acknowledgement.}
The research on this project was initiated during a joint research workshop of Tel Aviv University and the Freie Universit\"at Berlin on Ramsey Theory, held in Tel Aviv in March 2020, and partially supported by the GIF grant G-1347-304.6/2016. We would like to thank the German--Israeli Foundation (GIF) and both institutions for their support. We also thank Shagnik Das for suggesting this problem and the anonymous referee for their helpful feedback on earlier versions of the manuscript. 

\section{Preliminaries} \label{sec:prelims}

In this section, we present the notation and tools that we will use in our proofs. 

\subsection{Notation} \label{sec:notation}

We use standard probabilistic and graph theory notation throughout. For a graph $G$ and a subgraph $F$, we let $N_F(G)$ denote the number of copies of $F$ in $G$.  For vertex subsets $A,B\subseteq V(G)$ of a graph $G$, $G[A]$ denotes the graph induced by $G$ on $A$ and $e(A,B)$ denotes the number of edges of $G$ with one endpoint in $A$ and the other in $B$ (edges in $G[A\cap B]$ are counted \emph{once} here). For a vertex $v\in V(G)$, we let $d_G(v)$ denote the degree of $v$ in $G$ and $\Delta(G)\coloneqq\max_{v\in V(G)}d_G(v)$ denote the maximum degree. 

Given a hypergraph $\cH$, we denote the numbers of its edges and vertices by $e(\cH)$ and $v(\cH)$, respectively.  Further, for a vertex subset $T\subset V(G)$, $d_\cH(T)$ denotes the number of edges of $\cH$ containing $T$.  For an integer $\ell\ge 0$, we write $\Delta_\ell(\cH)\coloneqq\max\{d_\cH(T):T\subset V(\cH),  |T|=\ell\}$ for the maximum degree of a vertex set of size $\ell$ in $\cH$. A set of edges $M\subseteq E(\cH)$ in a hypergraph $\cH$ is a \emph{matching} if $e\cap f=\emptyset$ for all $e\neq f\in M$. The \emph{matching number} $\nu(\cH)\coloneqq\max\{|M|:M \text{ is a matching in } \cH \}$ is the size of the largest matching in a hypergraph $\cH$. 

For a set $W$, we let $\cP(W)\coloneqq\{U:U\subseteq W\}$ denote the power set of $W$, the set of all subsets of $W$.   For a Boolean statement $A$, we denote by $\mathds{1}[A]$, the indicator function which evaluates to $1$ if $A$ holds and $0$ if $A$ does not hold.

\subsection{Concentration inequalities}

We will frequently use  concentration inequalities for two families of random variables. The first inequality, often attributed to Chernoff~\cite{Chernoff1952} (see also~\cite[Theorem 2.1]{randomgraphbook}), deals with the case of binomial random variables.   

\begin{lem}[\textbf{Chernoff's inequality}]\label{bintail}
Let~$X\sim \Bin(n,p)$ be binomially distributed and let~$\mu=\EE[X]=np$. Then for any~$k\ge 0$, we have that 
\[\Pr[X\geq \mu+k]\leq  \exp\left(-\frac{k^2}{2(\mu +k/3)}\right) .\]
\end{lem}

The following inequality, known as Janson's inequality~\cite{janson1990poisson} (see also~\cite[Theorem 2.14]{randomgraphbook}) provides an exponential bound for the lower tail of the number of edges of a hypergraph induced by a random subset of its vertices.

\begin{lem}[\textbf{Janson's inequality}] \label{janson}
  Let $\Gamma $ be a finite set, let $p \colon \Gamma \to [0,1]$ and let $\Gamma _p$ be a random subset such that every element $a\in \Gamma$ is in $\Gamma _p$ with probability $p(a)$, independently of all other elements.

  Suppose that $A_1, \dotsc, A_m$ is a sequence of nonempty subsets of $\Gamma$. For each $i \in \{1, \dotsc, m\}$, denote by $I_i \coloneqq \mathds{1}[A_i\subseteq \Gamma_p]$ the indicator random variable for the event $A_i \subseteq \Gamma _p$. Finally, denote
  \[
    X\coloneqq \sum _{i=1}^m I_i,\qquad \mu \coloneqq \mathbb{E}[X] \qquad \text{and} \qquad \Delta \coloneqq \sum _{i,j=1}^m \mathds{1}[A_i \cap A_j \neq \emptyset] \cdot \mathbb{E}[I_iI_j].
  \]
  Then, for every $0\le k\le \mu$,
  \[
    \Pr[X\le \mu -k] \le \exp \left( -\frac{k^2}{2\Delta} \right) .
  \]
\end{lem}

In the setting of Janson's inequality (Lemma~\ref{janson}), we will also be interested in showing that a.a.s., there is a large collection of \emph{disjoint} subsets $A_i$ which all appear in $\Gamma_p$. Our next result provides an upper bound on the probability that this does not happen, and can be derived from Lemma~\ref{janson}. We provide the details of this derivation in Appendix~\ref{app:maxdisfam}. We recall that for a hypergraph $\cH$, $\nu(\cH)$ denotes the matching number of $\cH$, that is, the size of the  largest matching in $\cH$.

\begin{corollary}[\textbf{Maximal disjoint families}] \label{cor:maxdisfam}
  Let $\Gamma$, $p$, $A_1, \dotsc, A_m$, $\mu$ and $\Delta$ be as in the statement of Lemma~\ref{janson} and set $D \coloneqq \frac{\mu^2}{800\Delta}$.  Writing $\cA$ for the hypergraph with vertex set $\Gamma$ and edge set $\{A_1, \dotsc, A_m\}$, we have
  \[
    \Pr \left[ \nu\big(\cA[\Gamma_p]\big) \le D \right] \le \exp (- D).
  \]
\end{corollary}

\subsection{Containers}  \label{sec:containers}

We will appeal to the method of hypergraph containers, developed by Balogh, Morris and the third author of the present paper~\cite{balogh2015independent}, and independently, by Saxton and Thomason~\cite{SaxTho15}. The key idea underlying this method is that, given a uniform hypergraph whose edge set is evenly distributed, one can distribute its independent sets into a well-behaved collection of \emph{containers}. In more detail, these containers are vertex subsets that are almost independent (in that they induce few edges of the hypergraph), every independent set of the hypergraph lies in some container and, crucially, we have a bound on the number of containers. As there are many fewer containers than independent sets in the hypergraph, reasoning about containers rather than independent sets leads to more efficient arguments and this technique has proven to be extremely powerful. Indeed, the setting of independent sets in hypergraphs can be used to encode a wide range of problems in combinatorics and the method of hypergraph containers has been successfully exploited in a multitude of different settings, see~\cite{containersICMsurvey}. Particularly relevant to our work here are the applications of the method in sparse Ramsey theory, a program which was initiated by Nenadov and Steger~\cite{NenSte16}, who reproved the $1$-statement of Theorem~\ref{thm:triangle ramsey threshold} utilising containers. 

We state the container lemma in the following form, which follows from a general container lemma of Saxton and Thomason~\cite{SaxTho15}. We show how to derive this form of the container theorem in Appendix~\ref{app:containers}. We remark that the following version of the container theorem (and indeed the version of Saxton and Thomason~\cite{SaxTho15}) gives slightly more than we promised in the discussion of the general method given above:  The following theorem allows us to conclude not only that all independent sets lie in containers, but also sets that are very close to being independent (see condition~\ref{item:almostindep} of the theorem). Also, the theorem posits that every container is of the form $f(S_1,\ldots,S_t)$ for some family of small sets $S_i$, that is, each container is determined by a (constant-sized) collection of small subsets of the container. This in turn will allow us to run various union bounds over such collections of small subsets.

\begin{restatable}{thmtool}{containersthm}
  \label{thm:containers}
  For every positive integer $2 \le k\in \NN$ and all  $\eps\in(0,1)$ and $1 \le K\in \NN$, there exist  $t\in \NN$ and  $\delta>0$ such that the following holds. Suppose that a nonempty $k$-uniform hypergraph $\cH$ with vertex set $V$ and $\tau\in (0,1/t)$ satisfy
  \[
    \Delta_\ell(\cH) \le K \tau^{\ell-1} \cdot \frac{e(\cH)}{v(\cH)}
  \]
  for every $\ell \in \{2, \dotsc, k\}$. Then, there exists a function $f \colon \cP(V)^t \to \cP(V)$ with the following properties:
  \begin{enumerate}[label=(\roman*)]
  \item \label{item:almostindep}
    For every set $I \subseteq V$ satisfying $e(\cH[I]) \le \delta \tau^k e(\cH)$, there are $S_1, \dotsc, S_t \subseteq I$, each of size at most $\tau v(\cH)$ and  such that $I \subseteq f(S_1, \dotsc, S_t)$.
  \item \label{item:containersparse}
    For every $S_1, \dotsc, S_t \subseteq V$, the set $f(S_1, \dotsc, S_t)$ induces fewer than $\eps e(\cH)$ edges in $\cH$.
  \end{enumerate}
\end{restatable}

\subsection{Typical properties of $G_{n,p}$}
\label{sec:typical-properties}

In this section, we derive some properties of $G_{n,p}$ that a.a.s.\ hold. These will be useful throughout the paper.  We recall from Section~\ref{sec:notation}, that $N_F(G)$ denotes the number of copies of $F$ in $G$.  We also recall the standard notion of density of a graph $F$, which is $m(F) \coloneqq \max\{e_J/v_J : J \subseteq F\}$.

\begin{lem} \label{lem:randomgraph}
  Suppose that $G \sim G_{n,p}$ with $n^{-2/3} \ll p \ll n^{-1/2}$. Then a.a.s., $G$ has the following properties:
    \begin{enumerate}[{label=\textbf{(P\arabic*)}}] 
    \item \label{item:bound max deg}
      $\Delta (G) \le 2np$;
    \item
      \label{item:few F} $N_F(G)\le 2n^{v_F}p^{e_F}$ for all $F$ that satisfy both $v_F \le 8$ and $m(F) \le 3/2$; moreover, $N_F(G) \le n^{v_F}p^{e_F} \log n$ for all $F$ that only satisfy $v_F \le 8$. 
    \item \label{item:K210-count} $N_{K_{2,10}}(G)\le n^{11}p^{18}$;
    \item \label{item:many K3} 
      There is a constant $\theta > 0$ such that, for all $U\subseteq V(G)$ with $|U|\ge \frac{n}{2}$, the graph $G[U]$ contains a set of $\theta |U|^3p^3$ edge-disjoint triangles;
    \item
      \label{item:eAB-concentration}
      For all integer-valued $a=a(n),b=b(n)$ such that  $ (\log n)^7 / p\ll a,b\le n$, the following holds for any  $A,B\subseteq V$ with $|A|\le a, |B|\le b$:
      \[
        e(A,B) \le |A|\cdot |B| \cdot p + \frac{a\cdot b\cdot p}{\log ^3n}.
      \]
    \end{enumerate}
\end{lem}

\begin{proof}
  Property~\ref{item:bound max deg} and the first part of property~\ref{item:few F} are  standard properties of the distribution of the edges and small subgraphs in $G_{n,p}$ that can be proved using the second moment method and union bounds, see for example~\cite[Chapter 4]{alonspencer}. The `moreover' part of property~\ref{item:few F} follows from Markov's inequality and a union bound over all $F$ with $v_F\leq 8$.  Property~\ref{item:K210-count} also follows from Markov's inequality. Indeed, 
  we have that \[\EE[N_{K_{2,10}}(G)]\le n^{12}p^{20}=n^{11}p^{18}\cdot{np^2},\]
  and $np^2\ll 1$ due to the fact that $p\ll n^{-1/2}$.  

  We will establish property~\ref{item:many K3} using Corollary~\ref{cor:maxdisfam}.  To this end, fix any $U \subseteq V(G)$ with $|U|\ge \tfrac{n}{2}$, let $X\coloneqq N_{K_3}(G[U])$  and note that $\mu\coloneqq\EE[N_{K_3}(G[U])] = \binom{|U|}{3}p^3$ for $n$ large.  Denoting by $I_T$, for every triangle $T$ in $U$, the indicator random variable of the event that $T$ appears in $G$, we have
  \[
    \Delta\coloneqq\sum\left\{\EE[I_TI_{T'}]: T,T' \text{ triangles in }K_n[U], E(T)\cap E(T')\neq \emptyset\right\} 
    \leq \mu \cdot (1+ np^2)\le 2 \mu,
  \]
  where $\mu$ accounts for a choice of $T$, the first summand in the bracket accounts for a choice of $T'$ such that $|E(T')\cap E(T)|\ge 2$ (and hence $T=T'$) and the second summand accounts for a choice of $T'$ such that $|E(T)\cap E(T')|=1$; we also used that $p\ll n^{-1/2}$ in the last inequality.  By Corollary~\ref{cor:maxdisfam}, with $\Gamma = K_n[U]$ and $\cA$ the collection of all triangles in $\Gamma$, the probability that every collection of pairwise edge-disjoint triangles in $G[U]$ is smaller than
  \[
    D \coloneqq \frac{\mu^2}{800 \Delta} \ge \frac{\mu}{1600} \ge \frac{|U|^3p^3}{10000}
  \]
  is at most $e^{-D}$. Property~\ref{item:many K3} then follows from a union bound over the (less than $2^n$) choices of vertex subset $U$, using that $D\gg n$ for all such $U$ because  $p\gg n^{-2/3}$.  

Finally, in order to prove that $G$ has~\ref{item:eAB-concentration}, note that, for any choice of $a,b\gg (\log n)^7 / p$ and $A,B\subseteq V$ such that $|A|\le a, |B|\le b$, we have that $\mu \coloneqq \EE[e(A,B)]=\left(|A||B|-\binom{|A\cap B|}{2}\right)p\le |A||B|p$. Moreover $e(A,B)$ is binomially distributed and, setting $k \coloneqq abp/\log^3n$, we have that $\mu+k/3\le 2abp$. Hence, by Lemma~\ref{bintail},  we have that 
\begin{align*}
  \Pr\big[e(A,B)\ge |A||B|p+k\big]
  &\le \Pr\big[e(A,B)\ge \mu +k \big]
    \le \exp \left(-\frac{k^2}{2(\mu+k/3)}\right) \\
  & \le \exp\left(-\frac{(abp)^2}{\log^6 n\cdot 4abp}\right)
    \le \exp\left(-\frac{abp}{4\log^6 n}\right).
\end{align*}
Therefore, applying a union bound over the choice of  $a,b,A,B$, we get:
\begin{align*}
  \Pr\left[ G \notin~\ref{item:eAB-concentration} \right]
  & \le \sum _{a,b} \sum _{ \substack{ 1\le \alpha \le a \\ 1\le \beta \le b } } \binom{n}{\alpha} \cdot \binom{n}{\beta} \cdot \exp\left(-\frac{abp}{4\log^6 n}\right)\\
  & \le \sum _{a,b} \sum _{ \substack{ 1\le \alpha \le a \\ 1\le \beta \le b } } n^\alpha \cdot n^\beta \cdot \exp\left(-\frac{abp}{4\log^6 n}\right)\\
  & \le n^2\cdot \sum _{a,b } n^a \cdot n^b \cdot\exp\left(-\frac{abp}{4\log^6 n}\right) \\
  & \le n^2\cdot \sum _{a,b} \exp \left( a\log n + b\log n - \frac{abp}{4\log ^6 n}  \right), 
\end{align*}
where the (outer) sum goes over all choices of $a=a(n)\in \NN$ and $b=b(n)\in \NN$ such that  $(\log n)^7 / p\ll a,b\le n$. 
For all such $a,b$, we have that $\frac{abp}{4\log ^6n} \gg a\log n,b\log n$, and so
\[
  \Pr\left[ G \notin~\ref{item:eAB-concentration} \right] \le n^2 \cdot \sum _{a,b } n^{-g_{a,b}(n)} \ll 1,
\] where $g_{a,b}(n)= \frac{abp}{4\log ^7 n}- a - b  \gg 1$ for all $a,b$. 
\end{proof}

\section{Proof of the 0-statements}
\label{sec:0-statements}

In this section, we prove our $0$-statements, establishing the lower bounds on $q(n;K_3,p)$ in Theorem~\ref{thm:main}.  Our proofs in the lower and upper ranges follow the same scheme.  First, we will show the existence of a \emph{good} colouring of $G_1 \sim G_{n,p}$, i.e., a colouring with certain desirable properties specified in Definition~\ref{def:Crrbb} below.  Second, we will show that any such good colouring of $G_1$ can be a.a.s.\ extended to the second independent random graph $G_2\sim G_{n,q}$ when $q$ is chosen appropriately.  While the arguments showing existence of a good colouring will be different in the lower and the upper ranges, extendability of good colourings, stated as Proposition~\ref{prop:0-statements} below, is proved in the full range of interest.

\begin{defi} \label{def:Crrbb}
For  a colouring $\varphi \colon E(G) \to \{\red, \blue\}$ of the edges of a graph $G$, we define the coloured graph $C_{rrbb}$ to be a 4-cycle with two adjacent red edges and two adjacent blue edges.  Then for $t\geq 0$, we say a colouring $\varphi\colon E(G) \to \{\red, \blue\}$ is \emph{$t$-good} if it has the following properties:
\begin{enumerate}
  \item  \label{item:no mc tris} $\varphi$  has no  monochromatic triangles;
  \item \label{item:mostly blue}  every edge of $G$ that is not in a triangle is coloured blue;
  \item \label{item:few Crrbbs} the number of $C_{rrbb}$ in $G$ coloured by $\phi$ is at most $t$.
  \end{enumerate}
  Moreover, if the colouring is $0$-good,  we  will refer to it as being \emph{very good}.
\end{defi}

We remark that conditions~\ref{item:no mc tris} and~\ref{item:mostly blue} will be easy to impose on a colouring of $G_1\sim G_{n,p}$. Indeed, since $p$ is always below the $K_3$-Ramsey threshold of Theorem~\ref{thm:triangle ramsey threshold}, a $K_3$-free colouring exists and one can recolour edges not in triangles so that condition~\ref{item:mostly blue} is also satisfied. Thus, the critical condition in the definition is condition~\ref{item:few Crrbbs}. 
The motivation for considering copies of $C_{rrbb}$ is that they pose a direct threat to being able to extend a~colouring to $G_2$. Indeed, consider an  edge that forms a triangle with both the red edges and the blue edges of a~copy of  $C_{rrbb}$ in $G_1$. If this edge appears in $G_2$, then clearly there is no way to colour the edge without creating a monochromatic triangle. 

\begin{prop} \label{prop:0-statements}
  Suppose that $n^{-2/3} \ll p \ll  n^{-1/2}$, $t>0$ and  $0<q\ll \min\{t^{-1},n^{-3}p^{-7/2}\}$ and let $G_1 \sim G_{n,p}$ and $G_2 \sim G_{n,q}$ be independent. Then a.a.s.\ any $G_1$-measurable colouring that is $t$-good can be extended to a~triangle-free colouring of $G_1 \cup G_2$.
\end{prop}

In the proof of Proposition~\ref{prop:0-statements} and the further proofs of the $0$-statements, we will consider copies of certain uncoloured subgraphs, which are defined in Figure~\ref{fig:F0}.  Our next lemma explains why the graphs $F_0, F_1, K_4$, and their edge-deleted companions $F_0^-, F_1^-, K_4^-$, appear naturally when one looks for colourings with few copies of~$C_{rrbb}$.

\begin{figure}[h]
\centering
\begin{tikzpicture}[vertex/.style={draw,circle,color=black,fill=black,inner sep=1,minimum width=4pt}]
    
	\pgfmathsetmacro{\scale}{0.7}    
    
    \node[vertex] (u1) at (0*\scale, 0*\scale) {};
    \node[vertex] (u2) at (-1*\scale, -1*\scale) {};
    \node[vertex] (u3) at (1*\scale, -1*\scale) {};
    
    \node[vertex] (w1) at (0*\scale, -2*\scale) {};
    \node[vertex] (w2) at (2*\scale, 0*\scale) {};
    \node[vertex] (w3) at (-2*\scale, 0*\scale) {};
    
    \draw[thick] (u1) to (u2);
    \draw[thick] (u2) to (u3);
    \draw[thick] (u1) to (u3);

	\draw[thick] (u1) to (w3);
	\draw[thick] (u1) to (w2);
	\draw[thick] (u2) to (w1);
	\draw[thick] (u2) to (w3);
	\draw[thick] (u3) to (w1);
	\draw[thick] (u3) to (w2);
	
	\node at (0, -2*\scale - 0.4) {$F_0$};
    
\end{tikzpicture}
\quad \quad \quad 
\begin{tikzpicture}[vertex/.style={draw,circle,color=black,fill=black,inner sep=1,minimum width=4pt}]
    
    	\pgfmathsetmacro{\scale}{0.7}    
    
    \node[vertex] (u1) at (0*\scale, 0*\scale) {};
    \node[vertex] (u2) at (-1*\scale, -1*\scale) {};
    \node[vertex] (u3) at (1*\scale, -1*\scale) {};
    
    \node[vertex] (w1) at (0*\scale, -2*\scale) {};
    \node[vertex] (w2) at (2*\scale, 0*\scale) {};
    \node[vertex] (w3) at (-2*\scale, 0*\scale) {};
    
    \draw[thick] (u1) to (u2);

    \draw[thick] (u1) to (u3);

	\draw[thick] (u1) to (w3);
	\draw[thick] (u1) to (w2);
	\draw[thick] (u2) to (w1);
	\draw[thick] (u2) to (w3);
	\draw[thick] (u3) to (w1);
	\draw[thick] (u3) to (w2);
	
	\node at (0, -2*\scale - 0.4) {$F_0^-$};

\end{tikzpicture}
\quad \quad \quad \quad 
\quad \quad 
\begin{tikzpicture}[vertex/.style={draw,circle,color=black,fill=black,inner sep=1,minimum width=4pt}]
    
	\pgfmathsetmacro{\scale}{0.7}    
    
    \node[vertex] (u1) at (0*\scale, 0*\scale) {};
    \node[vertex] (u2) at (-1*\scale, -1*\scale) {};
    \node[vertex] (u3) at (1*\scale, -1*\scale) {};
    
    \node[vertex] (w1) at (0*\scale, -2*\scale) {};
    \node[vertex] (w) at (0*\scale, 1*\scale) {};
    
    \draw[thick] (u1) to (u2);
    \draw[thick] (u2) to (u3);
    \draw[thick] (u1) to (u3);

	\draw[thick] (u1) to (w);
	\draw[thick] (u2) to (w1);
	\draw[thick] (u2) to (w);
	\draw[thick] (u3) to (w1);
	\draw[thick] (u3) to (w);
	
	\node at (0, -2*\scale - 0.4) {$F_1$};
\end{tikzpicture}
\quad \quad \quad \quad 
\begin{tikzpicture}[vertex/.style={draw,circle,color=black,fill=black,inner sep=1,minimum width=4pt}]
    
	\pgfmathsetmacro{\scale}{0.7}    
    
    \node[vertex] (u1) at (0*\scale, 0*\scale) {};
    \node[vertex] (u2) at (-1*\scale, -1*\scale) {};
    \node[vertex] (u3) at (1*\scale, -1*\scale) {};
    
    \node[vertex] (w1) at (0*\scale, -2*\scale) {};
    \node[vertex] (w) at (0*\scale, 1*\scale) {};
    
    \draw[thick] (u1) to (u2);

    \draw[thick] (u1) to (u3);

	\draw[thick] (u1) to (w);
	\draw[thick] (u2) to (w1);
	\draw[thick] (u2) to (w);
	\draw[thick] (u3) to (w1);
	\draw[thick] (u3) to (w);
	
	\node at (0, -2*\scale - 0.4) {$F^-_1$};
\end{tikzpicture}
\caption{The graphs $F_0,F_0^-,F_1$ and $F_1^-$.}\label{fig:F0}
\end{figure}
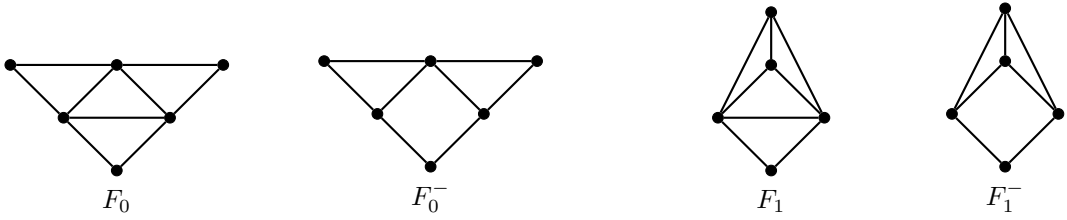

\begin{lem} \label{lem:4 cycle containment}
  Suppose that a 4-cycle $C$ in a graph $G$ has two adjacent edges $e_1,e_2$ such that, for $i=1,2$, $e_i$ is contained in a triangle of $G$ that does not contain $e_{3-i}$. Then $C$ is contained in a copy of $F^-$ in $G$, for some  $F\in\{F_0,F_1,K_4\}$. Moreover, in each case, adding the edge that forms a triangle with $e_1$ and $e_2$ completes this copy  of $F^-$ to a copy of $F$.
\end{lem}
\begin{proof}
  Label the vertices of $C$ as $x,y,w,z$ so that $e_1=xy$ and $e_2=xw$.  Further, for $i=1,2$, let $u_i$ be the vertex which forms the triangle with $e_i$ given by the statement of the lemma;  we therefore have that $u_1\neq w$ and $u_2\neq y$. We consider the following cases. Firstly, if one of the $u_i$ is equal to $z$, then the vertices of $C$ host a $K_4^-$ in $G$ and we are done.  So we can assume that neither $u_i$ is equal to $z$. If  $u_1=u_2$, then we get a copy of $F_1^-$ that contains $C$, whilst if $u_1\neq u_2$, we get a copy of $F_0^-$ containing $C$.  The moreover statement can also be easily checked in each case. 
\end{proof}

The following simple consequence of Lemma~\ref{lem:4 cycle containment} is more easily applicable in some of our proofs. 

\begin{corollary} \label{cor:Crrbb containment}
  Let $G$ be a graph coloured by some $\varphi\colon E(G)\rightarrow \{\red,\blue\}$ which satisfies conditions~\ref{item:no mc tris} and~\ref{item:mostly blue} of Definition~\ref{def:Crrbb}. Then any copy of $C_{rrbb}$ in $G$ is contained in some copy of $F^-$ in $G$ for some  $F\in\{F_0,F_1,K_4\}$.
\end{corollary}

\begin{proof}
  Let $e_1$ and $e_2$ be the red edges in the copy of $C_{rrbb}$.  Condition~\ref{item:mostly blue} in Definition~\ref{def:Crrbb} implies that each of these edges is in a triangle of $G$.  Moreover, $G$ does not contain the edge $e_3$ that forms a triangle with $e_1$ and $e_2$, as otherwise there would be no way to colour $e_3$ without creating a monochromatic triangle, contradicting the fact that $\varphi$ is triangle-free, which is condition~\ref{item:no mc tris} in Definition~\ref{def:Crrbb}.
\end{proof}

We are now in a position to prove Proposition~\ref{prop:0-statements}.

\begin{proof}[Proof of Proposition~\ref{prop:0-statements}]
  We first show the following claim.
  \begin{clm}
    \label{clm:Mixed bad subgraphs}
    A.a.s.\ every copy of $F_0$, $F_1$ and $K_4$ in $G_1 \cup G_2$ has at most one edge in $G_2$.
  \end{clm}
  \begin{claimproof}
    Let $X_{F_0}$ count the number of copies of $F_0$ in $G_1\cup G_2$ with at least two edges in $G_2$.  There are at most $n^6$ copies of $F_0$ in $K_n$ and the probability that each such copy appears in $G_1\cup G_2$ with at least two edges in $G_2$ is at most $\binom{9}{2} (p+q)^7 q^2$.  Since $q\ll p$, by our assumptions on $q$ and $p$, we can bound the expectation of $X_{F_0}$ as follows:
    \[
      \mathbb{E}[X_{F_0}] \le \binom{9}{2} n^6 (p+q)^7 q^2 \leq 2 \binom{9}{2} n^6 p^7 q^2 \ll 1,
    \]
    using that $q\ll n^{-3}p^{-7/2}$ in the last inequality here. 
    Similarly, defining $X_{F}$ to be the number of copies $F$ in $G_1\cup G_2$ with at least two edges in $G_2 $ for $F\in \{F_1,K_4\}$, we have
    \[ \mathbb{E}[X_{F_1}]\leq 2^8n^5p^6q^2 \ll n^{-1} p^{-1} \ll 1, \qquad \qquad
      \mathbb{E}[X_{K_4}]\leq 2^6n^4p^4q^2 \ll n^{-2} p^{-3} \ll 1.
    \]
    The assertion of the claim now follows by Markov's inequality.
  \end{claimproof}

  Given $G_1$ and a red/blue-colouring of its edges, we call an edge $uv \in K_n\setminus G_1$ \emph{dangerous} if there is a copy of $C_{rrbb}$ in $G_1$ with edges $uw,wv,vx,xu$ for some vertices $w, x \notin \{u,v\}$ such that $uw, wv$ are coloured red and $vx, xu$ are coloured blue.

  \begin{clm}
    \label{clm:Very good colouring}
    A.a.s.~$G_1$ has the following properties:
    \begin{enumerate}[label=(\roman*)]
    \item
      \label{cond:wb mixed F and K_4}
      A.a.s.\ (with respect to $G_2$), every copy of $F\in\{F_0,F_1,K_4\}$ in $G_1 \cup G_2$ has at most one edge in $G_2$,
    \item
      \label{cond:no dangerous edges}
      For every $t$-good colouring $\varphi$ of $G_1$, a.a.s.\ $G_2$ contains no dangerous edges.
    \end{enumerate}
  \end{clm}
  \begin{claimproof}
    Property~\ref{cond:wb mixed F and K_4} follows from Claim~\ref{clm:Mixed bad subgraphs} and Fubini's theorem.  Further, each $t$-good colouring $\varphi$ of $G_1$ contains at most $t$ copies of $C_{rrbb}$ and hence at most $t$ dangerous edges.  Since each such edge appears in $G_2$ with probability $q \ll t^{-1}$, the expected number of dangerous edges that appear in $G_2$ is $o(1)$ and~\ref{cond:no dangerous edges} follows from Markov's inequality.
  \end{claimproof}

  In view of Claim~\ref{clm:Very good colouring}, and again appealing to Fubini's theorem, it suffices to show that for every instance of $G_1$ that has the two properties described in the claim, every $t$-good colouring $\varphi$ of $G_1$ can be extended to a $K_3$-free colouring of $G_1 \cup G_2$ for every instance of the graph $G_2$ that satisfies the events described in both~\ref{cond:wb mixed F and K_4} and~\ref{cond:no dangerous edges}, which a.a.s.\ hold  (for fixed $G_1$ and $\varphi$).  To this end, consider an arbitrary ordering of the edges of $G_2\setminus G_1$ and colour them one-by-one according to the following rule.  We colour $e\in G_2\setminus G_1$ blue \emph{unless} $e$ forms a blue triangle with previously coloured edges (of $G_1\cup G_2$), in which case we colour $e$ red.  We claim that  the resulting colouring of $G_1 \cup G_2$ is $K_3$-free. 
  
  Note that the only possible monochromatic triangles that can occur in our colouring of $G_1 \cup G_2$ must be red and contain an edge of $G_2\setminus G_1$, as $\varphi$ is good and thus triangle-free.  
  Suppose that $e$ is the \emph{last} edge of $G_2\setminus G_1$ that completes a red triangle; denote this triangle by $T$ and its two remaining edges by $f_{r}$ and $g_r$.  Note that $e$ is also contained in a triangle with two blue edges, say $f_b$ and $g_b$, already coloured in $G_1\cup G_2$, as otherwise our rule would colour $e$ blue.
  We claim that $f_r$ and $g_r$ are both contained in triangles other than $T$ in $G_1 \cup G_2$.  Indeed, let $h \in \{ f_r, g_r\}$ be arbitrary.  If $h \in G_1$, then $h$ must be contained in a (non-monochromatic) triangle of $G_1$ due to the fact that our colouring of $G_1$ was good and the fact that $h$ is coloured red.  If $h \in G_2 \setminus G_1$, then it must be contained in a triangle whose remaining two edges are blue, as otherwise we would have coloured it blue.  Consequently, by Lemma~\ref{lem:4 cycle containment}, we have that the $4$-cycle $C \coloneqq \{f_r,g_r,f_b,g_b\}$ is contained in a copy of $F^-$ in $G_1\cup G_2$, for some $F\in\{F_0,F_1,K_4\}$, and the edge $e$ completes this copy of $F^-$ to a copy of $F$.  However, as we have assumed that $G_2$ contains no dangerous edges, one of the edges of the $C_{rrbb}$-copy $C\coloneqq\{f_r,g_r,f_b,g_b\}$ must belong to $G_2$.  This contradicts the assumed conclusion of~\ref{cond:wb mixed F and K_4}, as the copy of $F$ containing $C$ has at least two edges in $G_2$, namely $e$ and one of the edges in $C$. 
  This shows that no red triangle can occur, which concludes the proof of the proposition.
\end{proof}

\begin{rem}
  \label{rem:0-statement proof}
  Our proof of the 0-statements adopts a greedy strategy to colour the second random graph. 
  In fact, our colouring is identical to the colouring used in~\cite{friedgut2003ramsey} to prove that the online game a.a.s.\ lasts $\Omega(n^{4/3})$ rounds under optimal play.
  Indeed, they also colour each edge blue as it appears unless it creates a blue triangle, in which case they colour it red.
  The authors of~\cite{friedgut2003ramsey} show that this colouring will only fail to avoid monochromatic triangles if a copy of $F_0$ (Figure~\ref{fig:F0}) or $K_4$ appears, which a.a.s.\ does not happen with $\ll n^{4/3}$ rounds/edges.
  Similarly, one can adjust our proof presented above to show that Theorem~\ref{thm:initial two-round results}~\ref{two-round bottom range} is tight in the following sense: When $M_1=cn^{4/3}$ for some $c>0$ and $M_2\ll n^{4/3}$, then a.a.s.\ there is a red/blue-colouring of $G_1$ that can be extended to the edges of $G_2$ avoiding monochromatic copies of $K_3$. Indeed, as in our proof of Proposition~\ref{prop:0-statements}, one can colour $G_1$ avoiding monochromatic triangles such that every edge not in a triangle is blue. Colouring the second random graph according to the online greedy approach, as in~\cite{friedgut2003ramsey}, the player will only fail if a copy of $F_0$ or $K_4$ appears in $G_1\cup G_2$, with at least one edge of $G_2$. Such copies a.a.s.\ do not exist when $M_2\ll n^{4/3}$ and so the player a.a.s.\ succeeds. 
\end{rem}

\subsection{The $0$-statement in the upper range}
\label{sec:0 upper}

In this section, we prove the lower bound on $q(n;K_3,p)$ in the upper range $n^{-3/5} \ll p \ll n^{-1/2}$ of Theorem~\ref{thm:main}.  In this case, the proof follows  easily from Proposition~\ref{prop:0-statements}.

\begin{thmtool} \label{thm:0 upper}
 Suppose that $n^{-2/3} \ll p \ll  n^{-1/2}$ and $q \ll n^{-6}p^{-8}$ and let $G_1 \sim G_{n,p}$ and $G_2 \sim G_{n,q}$ be independent. Then a.a.s.\ there is a $G_1$-measurable colouring $\varphi \colon E(G_1)\to \{\red,\blue\}$ that can be extended to a~triangle-free colouring of $G_1\cup G_2$.
\end{thmtool}

\begin{proof}
  By Proposition~\ref{prop:0-statements}, it suffices to show that a.a.s.\ $G_1\sim G_{n,p}$ admits a $t$-good colouring $\phi$ with $t=150n^6p^8$.  Firstly, we claim that a.a.s.\ $G_1$ contains at most $n^6p^8$ copies of  $F^-$, for each  $F\in \{F_0, F_1,K_4\}$.  Indeed, this follows from Lemma~\ref{lem:randomgraph}~\ref{item:few F}, which applies since $m(F_0^-) = 4/3$, $m(F_1^-) = 7/5$, and $m(K_4^-) = 5/4$, and the fact that $n^4p^5, n^5p^7\ll n^6p^8$ for $p\gg n^{-2/3}$.  We also have that  a.a.s.\  $G_1$ is \emph{not} $K_3$-Ramsey due to Theorem~\ref{thm:triangle ramsey threshold}~\ref{triangle ramsey 0 statement}.  It is therefore enough to show that $G_1$ admits a $t$-good colouring $\phi$ under the assumption that these two asymptotically-almost-sure events occur.

  We define $\phi$ as follows. Take any triangle-free colouring of $G_1$ and recolour any edge not in a triangle blue. As we only changed the colour of edges not in triangles, it is clear that $\phi$ remains triangle-free; it only remains to show that $\phi$ induces at most $t$ copies of $C_{rrbb}$.  This follows from Corollary~\ref{cor:Crrbb containment}, as each copy of $C_{rrbb}$ is contained in some copy of $F^-$ for some $F\in\{F_0, F_1,K_4\}$.  Each such copy of some $F^-$ with $F\in\{F_0, F_1,K_4\}$, hosts at most\footnote{We state this loose upper bound for simplicity.  It is easy to check that there are at most $45$ copies of $C_4$ in any graph with at most $6$ vertices, but of course our specific $F^-$ contain many fewer $4$-cycles.} $50$ copies of $C_{rrbb}$ and so the number of $C_{rrbb}$ is at most $50$ times the number of copies of some $F^-$ with $F\in\{F_0, F_1,K_4\}$. This completes the proof due to our upper bounds on the number of these copies above.
\end{proof}

\subsection{The $0$-statement in the lower range} \label{sec:0 lower}

In this section, we prove the lower bound on $q(n;K_3,p)$ in the lower range $n^{-2/3} \ll p \ll n^{-3/5}$ of Theorem~\ref{thm:main}.  We will again appeal to Proposition~\ref{prop:0-statements}, but now we will be able to show the existence of a $t$-good colouring of $G_1$ for the much larger value~$t = n^3p^{7/2}$.

\begin{thmtool} \label{thm:0 lower}
  Suppose that $n^{-2/3} \ll p \ll n^{-3/5}$ and $q \ll n^{-3}p^{-7/2}$ and let $G_1 \sim G_{n,p}$ and $G_2 \sim G_{n,q}$ be independent.  Then a.a.s.\ there is a $G_1$-measurable colouring $\varphi \colon E(G_1)\to \{\red,\blue\}$ that can be extended to a~triangle-free colouring of $G_1\cup G_2$.
\end{thmtool}

Recall that in the proof of Theorem~\ref{thm:0 upper}, we showed that every copy of $C_{rrbb}$ is contained in a copy of $F^-$, for some $F\in \{F_0,F_1,K_4\}$, and we used simple upper bounds on the number of copies of $F^-$.  Here, such simple bounds will no longer suffice and we will have to explore how the copies of these fixed graphs interact.   Clearly, any singular copy of such an $F^-$ can be coloured so that it avoids both monochromatic triangles \emph{and} copies of $C_{rrbb}$.  Therefore, we are only forced to create copies of $C_{rrbb}$ if these copies of some $F^-$ and the copies of triangles interact in certain ways. The following definition captures the subgraphs of $K_n$ that correspond to a collection of interacting copies of $K_3$, $F_0^-$ and $F_1^-$ (we exclude $K_4^-$ from this list, as it is composed of two interacting triangles).

\begin{defi} \label{def:collage}
  Let $H$ be the hypergraph with vertex set $E(K_n)$ whose hyperedges are all ($3$-, $7$- or $8$-element) sets of edges that form a copy of $K_3$, $F_0^-$ or $F_1^-$ in $K_n$.  We will call a graph $C\subseteq K_n$ a \emph{collage} if $C$ induces a~connected subhypergraph of $H$.  We will denote the collection of collages in $K_n$ by $\mathcal{C}$. 
\end{defi}

We also define collages that are  \emph{well-behaved} as follows. 

\begin{defi} \label{def:well behaved}
We say a collage $C\in \cC$ is well-behaved if 
\begin{enumerate}[label=(\roman*)] 
    \item \label{wb:small} $v(C)\le \log n$;
    \item \label{wb:sparse} For any subgraph $C'\subseteq C$ such that $C'\in \cC$, we have that $e(C')/v(C')<5/3$.
\end{enumerate}
Moreover, we say that $C\in \cC$ is \emph{very well-behaved} if, in addition to~\ref{wb:small} and~\ref{wb:sparse}, $C$ satisfies the following further condition:
\begin{enumerate}[label=(\roman*)] 
\setcounter{enumi}{2}
\item \label{wb:subgraph} $C$ contains no copies of a graph $F$ with  $(v_F,e_F)\in\{(4,6),(5,7),(8,12)\}$.
\end{enumerate}
\end{defi}

We will reduce Theorem~\ref{thm:0 lower} to two key lemmas. The first shows that our random graph $G_1\sim G_{n,p}$ will a.a.s.\ only contain well-behaved collages. 

\begin{lem} \label{lem:collages}
  Suppose that $n^{-2/3} \ll p \ll n^{-3/5}$ and let $G_1 \sim G_{n,p}$. Then a.a.s.\  every collage $C\in \cC$ such that $C\subseteq G_1$ is well-behaved. 
\end{lem}

Our second lemma asserts that \emph{very} well-behaved collages can be coloured avoiding any copies of $C_{rrbb}$.

\begin{lem}
  \label{lem:collagescolorable}
  Every very well-behaved collage $C\in \mathcal{C}$  admits a \emph{very} good colouring.
\end{lem}

We remark that condition~\ref{wb:small} of Definition~\ref{def:well behaved} is in fact irrelevant here and we will prove that the conclusion of Lemma~\ref{lem:collagescolorable} holds for all collages that satisfy conditions~\ref{wb:sparse} and~\ref{wb:subgraph}.  Before proving these lemmas, let us see how they imply Theorem~\ref{thm:0 lower}.

\begin{proof}[Proof of Theorem~\ref{thm:0 lower}]
  By Proposition~\ref{prop:0-statements}, it suffices to show that a.a.s.\ $G_1$ admits a $t$-good colouring with $t \coloneqq n^3p^{7/2}$.  Let us assume the asymptotically-almost-sure conclusions of Lemma~\ref{lem:collages} and Lemma~\ref{lem:randomgraph}~\ref{item:few F} and also that $G_1$ is not $K_3$-Ramsey, which happens a.a.s.\ due to Theorem~\ref{thm:triangle ramsey threshold}~\ref{triangle ramsey 0 statement}.

  We colour the edges of $G_1$ according to the following scheme, where we define $\cC(G_1)$ to be the collection of collages $C \in \cC$ such that $C \subseteq G_1$:
  \begin{enumerate}
  \item Colour all maximal subgraphs $C\in \mathcal{C}(G_1)$ which are very well-behaved with a very good colouring (this is possible, due to Lemma~\ref{lem:collagescolorable}).
  \item Colour all the other maximal subgraphs in $\mathcal{C}(G_1)$ in a triangle-free way, such that all edges not in a triangle are blue (this is possible as $G_1$ is not $K_3$-Ramsey).
  \end{enumerate}
        
  We claim that the resulting colouring is $t$-good.  Firstly, note that all edges of $G_1$ are indeed coloured as every edge of $G_1$ lies in some maximal collage contained in $G_1$ (the collage may just consist of the single edge).  Clearly, all edges not in a triangle are coloured blue.  Moreover, as every triangle of $G_1$ lies in some maximal collage, the resulting colouring contains no monochromatic triangles.  It remains to show that there are at most $t$ copies of $C_{rrbb}$.  

  Corollary~\ref{cor:Crrbb containment} implies that any copy of $C_{rrbb}$ must lie in some copy of $F_1^-$, $F_0^-$, or $K_4^-$.  However, each of these graphs is contained in some maximal collage (they all induce connected subhypergraphs in $H$) and thus all copies of $C_{rrbb}$ are contained in maximal collages that are not very well-behaved.  It thus suffices to show that the total number of $4$-cycles that lie in such collages is at most $t$.  By the assumed conclusion of Lemma~\ref{lem:collages},  every $C\in \mathcal{C}(G_1)$ is well-behaved, and so if $C$ is not  very well-behaved, then it contains some copy of a subgraph $F$ with $v_F$ vertices and $e_F$ edges such that $(v_F,e_F)\in\{(4,6),(5,7),(8,12)\}.$  However, the assumed conclusion of  Lemma~\ref{lem:randomgraph}~\ref{item:few F} implies that there are at most
  \[
    n^4p^6\log n+\binom{\binom{5}{2}}{7}n^5p^7\log n+\binom{\binom{8}{2}}{12}n^8p^{12}\log n\leq 2^{10}n^5p^7\log n
  \]
  copies of such an $F$ in $G_1$ (using that $n^4p^6\ll n^8p^{12}\ll n^5p^7$ here).  Therefore, there are at most $2^{10}n^5p^7\log n$ 
maximal collages $C\in \cC(G_1)$ that are not very well-behaved. A collage $C\in \cC(G_1)$ contains at most $v(C)^4$ copies of 4-cycles, and each collage $C\in \cC(G_1)$ has at most $\log n$ vertices on account of it being well-behaved. So in total, there are at most $2^{10}n^5p^7 \cdot \log^5n\ll n^3p^{7/2}$ copies of $C_{rrbb}$ in our colouring of $G_1$, finishing our proof.
\end{proof}

It remains to prove Lemmas~\ref{lem:collages} and~\ref{lem:collagescolorable}. We begin by proving Lemma~\ref{lem:collages}.

\begin{proof}[Proof of Lemma~\ref{lem:collages}]
  Let $\Cbad$ be the collection of all $C \in \mathcal{C}$ such that either $e(C) / v(C) \ge 5/3$ or $v(C) \ge \log n$. It suffices to show that a.a.s.\ $G_1 \sim G_{n,p}$ does not contain any subgraphs in $\Cbad$.  In fact, we will focus on another  family $\Ccores$ of (nonempty) subgraphs of $K_n$ with the following properties:
  \begin{enumerate}[label=(\alph*)]
  \item
    \label{item:Ccores-1}
    Every set in $\Cbad$ contains some element of $\Ccores$.
  \item
    \label{item:Ccores-2}
    Every $C^* \in \Ccores$ satisfies $e(C^*) \ge 5v(C^*)/3$ or both $v(C^*) \ge \log n$ and $e(C^*) \ge 5v(C^*)/3-3$.
  \item
    \label{item:Ccores-3}
    For every $5\le k\le n$, there are at most $(2k)^{150}(16n)^k$ graphs $C^* \in \Ccores$ with $v(C^*) = k$.
  \end{enumerate}
  Assuming we can find such a family $\Ccores$ of subgraphs, we claim that we are done. Indeed, by~\ref{item:Ccores-1},
  \[ 
    \Pr[\exists C \in \Cbad : C \subseteq G_1] \le \Pr[\exists C^* \in \Ccores :C^* \subseteq G_1] \le \sum_{C^* \in \Ccores} p^{e(C^*)}.
  \]
  Moreover, by~\ref{item:Ccores-2} and~\ref{item:Ccores-3},
  \[
    \sum_{C^* \in \Ccores} p^{e(C^*)} = \sum^n_{k=5} \sum_{\substack{C^* \in \Ccores \\ v(C^*) = k}} p^{e(C^*)} \le \sum^{\log n}_{k =5} (2k)^{150} (16n)^k p^{5k/3}  + \sum^{n}_{k =\log n} (2k)^{150} (16n)^k p^{5k/3-3} \ll 1,
  \]
  where the last inequality follows from the assumption that $p \ll n^{-3/5}$.  We also used that~\ref{item:Ccores-2} easily implies that there are no $C^*\in \Ccores$  with less than 5 vertices.

  It remains to define a family $\Ccores$ of subgraphs of $K_n$ satisfying conditions~\ref{item:Ccores-1}--\ref{item:Ccores-3} above. First, we fix  some order $\sigma$ on $E(K_n)$. Now given a collage $C \in \Cbad$, we construct the `core' of $C$ algorithmically as follows.  We initiate our algorithm with logs $L_V$, $L_E$, $L_T$, $L_O$ and $L_D$ all being empty. Throughout the algorithm, we will have that $L_V$ is a sequence of distinct vertices in $V(C)\subseteq V(K_n)$, $L_E$ and $L_T$ are sequences of distinct edges in $E(C)\subseteq E(K_n)$, $L_O$ is a sequence of  integers  and $L_D$ is a sequence whose each entry indicates a time step $i\geq 0$ and some set of edges $F\subseteq E(C)\subseteq E(K_n)$. We will maintain, at the end of every time step $i\geq 0$ of the algorithm, that the set of vertices in $L_V$ and the set of edges in $L_E$ define a subgraph of $C$, which we denote as $C_i$ (and so $C_0$ is the empty graph). Moreover at the end of each time step $i\geq 0$,  the collection of edges featuring in $L_T$ will be precisely those edges in $L_E$ which are contained in a triangle in $C_i$.

  Now, in the first step of the algorithm, we take $e_1\in C$ to be the first edge in $C_1$ according to the order $\sigma$ on $E(K_n)$, add its endpoints (in an arbitrary order) to $L_V$ and add $e_1$ to $L_E$ (so that $C_1$ is the one-edge graph $e_1$). In every subsequent step $i\geq 2$, we do the following: 
  \begin{itemize}
      \item We terminate and output $C^*=C_{i-1}$  if one of the following is true:
        \[
          |L_D|=7, \qquad |L_V|\geq \log n \qquad \text{or} \qquad C_{i-1}=C.
        \] 
      \item
        Otherwise, since $C_{i-1} \neq C$ and $C$ is a collage, there must be a copy of $K_3$, $F_0^-$ or $F_1^-$ in $C$ that intersects $C_{i-1}$ in at least one edge but is not fully contained in $C_{i-1}$.  Call such a copy \emph{regular} if it is a copy of $F_0^-$ and it intersects $C_{i-1}$ in a triangle; otherwise, call it \emph{degenerate}.  We say that a regular copy of $F_0^-$ is \emph{rooted} at $e$ if $e$ belongs to its intersection with $C_{i-1}$ and it is the edge of the triangle that also participates in the (unique) $4$-cycle of $F_0^-$.
        \begin{itemize}
        \item If there exist regular copies of $F_0^-$, then to each copy associate a number $x\in \NN$ which is the position in $L_T$ of the edge $e_x$ that the copy is rooted at. Take a copy of $F_0^-$ that minimises this position and add this minimum $x$ to $L_O$ if the vertices of $e_x$ appear in $L_V$ in an order consistent with their  degrees in the copy of $F_0^-$. If, on the other hand,  the vertex  of $e_x$ which has degree 4 in the copy of $F_0^-$ appears in $L_V$ before the vertex of $e_x$ with degree 3, we add $-x$ to $L_O$. We also update $L_V$ and $L_E$ by appending  the vertices and the edges of our copy of $F_0^-$ that do not lie in $C_{i-1}$: the five such edges are added according to their relative order in $\sigma$ and the three vertices in some canonical order.  Note that this defines $C_i$ with $e(C_i) = e(C_{i-1}) + 5$ and $v(C_i) = v(C_{i-1})+3$. Finally, we add to $L_T$ all edges that are in a triangle in $C_i$ but not in a triangle in $C_{i-1}$, adding them in the order that they appear in $L_E$. 
          \item If there are no regular copies of $F_0^-$, fix some degenerate copy of $K_3$, $F_0^-$ or $F_1^-$, append to $L_V$ and to $L_E$ the vertices and the edges of this copy that do not lie in $C_{i-1}$: the edges are added according to their relative order in $\sigma$ and the vertices (if there are any) in an arbitrary order; this again defines $C_i$. We then add to $L_T$ all edges that are in a triangle in $C_i$ but not in a triangle in $C_{i-1}$, adding them in the order that they appear in $L_E$.  Finally, detail this degenerate step $i$ by logging it in $L_D$ along with the set of edges in $E(C_i)\setminus E(C_{i-1})$.
      \end{itemize}
    \end{itemize}
    Since each step increases $L_E$ by at least one, this algorithm terminates for any $C\in \Cbad$.  We may thus define $\Ccores$ as the set of all its outputs, that is, $\Ccores \coloneqq \{C^* : C \in \Cbad\}$.  This definition guarantees that $\Ccores$ satisfies~\ref{item:Ccores-1} above; we will show that it also satisfies~\ref{item:Ccores-2} and~\ref{item:Ccores-3}. First though we establish the following key estimate that bounds the distance of $e(C_i)$ from $5v(C_i)/3$ in terms of the number of degenerate steps (equivalently, the size of $|L_D|$) at the end of step $i$, which we denote by $d(i)$.

  \begin{clm}\label{clm:collages_density}
  For all $i\geq 1$, we have that  $
      d(i) \le 3e(C_i) - 5v(C_i) + 7\le 21d(i).$
  \end{clm}
  \begin{claimproof}
  Both inequalities hold with equality when $i=1$ since $e(C_1)=1$, $v(C_1)=2$, and $d(1) = 0$.  Suppose that $i \ge 2$ and the claim holds for $i-1$.  If the $i$th step is regular, the claim continues to hold since $e(C_i) = e(C_{i-1}) + 5$, $v(C_i) = v(C_{i-1}) + 3$, and $d(i) = d(i-1)$.  If the $i$th step is degenerate, then there is some $H^{\prime} = H\cap C_{i-1}$ such that $H$ is a copy of $K_3$, $F_0^-$ or $F_1^-$ and  $H^{\prime}$ is a proper subgraph of $H$ that contains at least one edge and $(H,H^{\prime})\neq (F_0^-,K_3)$. In this case we have that $e(C_i) = e(C_{i-1}) + e, \ v(C_i) = v(C_{i-1}) + v$, where $e \coloneqq e(H)-e(H^{\prime})$ and $v \coloneqq v(H)-v(H^{\prime})$. For every such $H,H^{\prime}$, we have $1 \le 3e-5v \le 21$, and so the claim will hold also for $i$. Indeed the upper bound of $3e-5v \le 21$ follows from the fact that $e\le 7$ as  $H$ has at most $8$ edges and $H'$ is non-empty. The lower bound $3e-5v\geq 1$ follows from a simple case analysis considering $v=1,2,3,4$ and noting that one can assume that $H'$ is an induced subgraph of $H$. We leave the details to the reader.
\end{claimproof}

Property~\ref{item:Ccores-2} now follows from Claim~\ref{clm:collages_density}. Indeed, if $C^*=C$ and $v(C^*)<\log n$ then certainly $e(C^*)\geq 5v(C^*)/3$ as $C\in \Cbad$. If $C^*\neq C$ and  $v(C^*) < \log n$, then at the time $\tau$ at which the algorithm terminates we have that $C^*=C_{\tau-1}$ and $d(\tau-1)=7$ and so 
  $0\le 3e(C^*)-5v(C^*)$ as required. Finally if $v(C^*)\geq \log n$, the lower bound on $e(C^*)$ also follows from Claim~\ref{clm:collages_density}, using that, trivially, $0$ is a lower bound on the number of degenerate steps taken when the algorithm terminates.

 It remains to prove property~\ref{item:Ccores-3} of the collection $\cC^*$, so let us fix some $ 5\leq k \leq n$.  We bound the number of $C^*\in \Ccores$ with $v(C^*)=k$ as follows.  Firstly, we note that $C^*$ can be completely determined by the logs $L_V, L_O$ and $L_D$ when the algorithm terminates. That is, we can recover $L_E$  (and $L_T$) from these logs. Indeed, the first two vertices in $L_V$ determine the first edge in $L_E$. Now, suppose we have recovered $L_E$ and $L_T$ up to time $i-1$ and consider time step $i$. If the step is not degenerate (which we know from $L_D$), the new edges of the regular copy of $F_0^-$ are completely determined by the edge that the copy is rooted at, which we know from $L_O$ and our recovery of $L_T$ so far, which vertex of this edge has degree 4 in the copy of $F_0^-$, which we know from the sign of the entry in $L_O$ and the order of the vertices of this edge in $L_V$, and the next three vertices, which appear (in a canonical order) in $L_V$. Hence, we can add the new edges (in order according to the order $\sigma$ on $E(K_n)$) to $L_E$ and recover $L_E$ up to time $i$. If, on the other hand, the step $i$ is degenerate, then $L_D$ will signify this and indicate the new edges that need to be added to $L_E$. Again the order they are added is determined by $\sigma$. In either case, once we know $L_E$ at time $i$, we can recover $L_T$ up to time $i$ also. 
 
 This shows that in order to bound the number of $C^*$ with $v(C^*)=k$, it suffices to bound the number of possibilities for the logs $L_V,L_O$ and $L_D$ output by the algorithm with $|L_V|=k$. For the logs $L_V$, we use the simple upper bound that there are at most $n^k$ choices. For the logs $L_O$, note first that Claim~\ref{clm:collages_density} implies that any $C^*$ with $v(C^*)=k$ has \[e(C^*)\le \frac{5k+21\cdot 7}{3}\le 2k+49, \]
 using here that there are at most $7$ degenerate steps before the algorithm terminates. 
Now for any log $L_O$, let $L_O^{|\cdot|}$ be the log obtained from $L_O$ by replacing each entry by its  absolute value. In each instance of the algorithm, we therefore have that the log $L^{|\cdot|}_O$ is a nondecreasing sequence of natural numbers bounded by $e(C^*)\leq 2k+49$.  Moreover, the length of the sequence is the number of regular steps which is certainly less than $k$ and we can append repeated entries with value $2k+50$ to make all the sequences length $k$. Hence we can bound the number of possible $L^{|\cdot|}_O$ by the number of nondecreasing sequences of length $k$ with elements in $\{1, \dotsc, 2k+50\}$, which is $\binom{k+(2k+50)-1}{k}\le 2^{3k+50}$. The log $L_O$ can then be recovered from $L^{|\cdot|}_O$ by a choice of sign for the at most $k$ entries. Finally, each of the at most 7 entries of $L_D$ indicates a step (at most $k+7$ choices) and a selection of at most 7 edges which lie on vertices in $L_V$ (at most $k^2$ choices for each edge). Hence there at most $\sum_{j=0}^7((k+7)(\sum_{h=1}^7(k^2)^h))^j\le k^{140}$ choices for $L_D$. Combining our estimates of the number of choices of $L_V,L_O$ and $L_D$ gives~\ref{item:Ccores-3} and completes the proof of Lemma~\ref{lem:collages}.
\end{proof}

Finally, it remains to prove Lemma~\ref{lem:collagescolorable}, which is the subject of the rest of this section. In order to prove this, we use a novel `discharging'  method, similar in spirit to the method used in the recent work of Friedgut, Kuperwasser, Schacht and the third author~\cite{friedgut2022sharp} in proving sharp thresholds for Ramsey properties. 

\begin{proof}[Proof of Lemma~\ref{lem:collagescolorable}]
  Assume the contrary and let $C \in \mathcal{C}$ be a smallest counterexample.  We claim that every proper subgraph $D \subsetneq C$ admits a very good colouring.  Indeed, let $D = D_1 \cup \dotsb \cup D_t$ be the partition of $D$ into maximal collages and note each $D_i$ is also very well-behaved.  Since $C$ is a smallest counterexample, each $D_i$ admits a very good colouring.  We claim that the union of these colourings is a very good colouring of $D$.  Indeed, every triangle in $D$ is contained in some $D_i$ (as it is a maximal collage) and thus it is not monochromatic.  It is clear that every edge not in a triangle is coloured blue.  If there was a copy of $C_{rrbb}$ in $D$, it would lie in some copy of $F_0^-$, $F_1^-$ or $K_4^-$, by Corollary~\ref{cor:Crrbb containment}, and thus it would lie in some $D_i$ (as each of $F_0^-$, $F_1^-$ and $K_4^-$ induces a connected subhypergraph of $H$), a contradiction.

  Our aim is now to remove from $C$ a carefully chosen selection of edges and show that any very good colouring of the remaining subgraph (which exists, from above) can be extended to the removed edges while remaining very good, thus contradicting our assumption that $C$ is a counterexample.  In order to find such a  \emph{removable} set of edges, we define a discharging procedure which assigns weights to small subgraphs of $C$ that we call  \emph{blocks}.

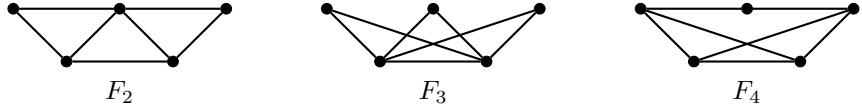
\begin{figure}[h]
\centering
\begin{tikzpicture}[vertex/.style={draw,circle,color=black,fill=black,inner sep=1,minimum width=4pt}]
    
	\pgfmathsetmacro{\scale}{0.7}    
    
    \node[vertex] (u1) at (0*\scale, 0*\scale) {};
    \node[vertex] (u2) at (-1*\scale, -1*\scale) {};
    \node[vertex] (u3) at (1*\scale, -1*\scale) {};

    \node[vertex] (w2) at (2*\scale, 0*\scale) {};
    \node[vertex] (w3) at (-2*\scale, 0*\scale) {};
    
    \draw[thick] (u1) to (u2);
    \draw[thick] (u2) to (u3);
    \draw[thick] (u1) to (u3);

	\draw[thick] (u1) to (w3);
	\draw[thick] (u1) to (w2);

	\draw[thick] (u2) to (w3);

	\draw[thick] (u3) to (w2);
	
	\node at (0, -1*\scale - 0.4) {$F_2$};
    
\end{tikzpicture}
\quad \quad \quad
\begin{tikzpicture}[vertex/.style={draw,circle,color=black,fill=black,inner sep=1,minimum width=4pt}]
    
    	\pgfmathsetmacro{\scale}{0.7}    
    
       \node[vertex] (u1) at (0*\scale, 0*\scale) {};
    \node[vertex] (u2) at (-1*\scale, -1*\scale) {};
    \node[vertex] (u3) at (1*\scale, -1*\scale) {};

    \node[vertex] (w2) at (2*\scale, 0*\scale) {};
    \node[vertex] (w3) at (-2*\scale, 0*\scale) {};
    
    \draw[thick] (u1) to (u2);
    \draw[thick] (u2) to (u3);
    \draw[thick] (u1) to (u3);

	\draw[thick] (u3) to (w3);
	\draw[thick] (u2) to (w2);

	\draw[thick] (u2) to (w3);

	\draw[thick] (u3) to (w2);
	
	\node at (0, -1*\scale - 0.4) {$F_3$};

\end{tikzpicture}
\quad \quad \quad
\begin{tikzpicture}[vertex/.style={draw,circle,color=black,fill=black,inner sep=1,minimum width=4pt}]
    
    	\pgfmathsetmacro{\scale}{0.7}    
    
       \node[vertex] (u1) at (0*\scale, 0*\scale) {};
    \node[vertex] (u2) at (-1*\scale, -1*\scale) {};
    \node[vertex] (u3) at (1*\scale, -1*\scale) {};

    \node[vertex] (w2) at (2*\scale, 0*\scale) {};
    \node[vertex] (w3) at (-2*\scale, 0*\scale) {};
    
    \draw[thick] (u1) to (w2);
    \draw[thick] (u2) to (u3);
    \draw[thick] (u1) to (w3);

	\draw[thick] (u3) to (w3);
	\draw[thick] (u2) to (w2);

	\draw[thick] (u2) to (w3);

	\draw[thick] (u3) to (w2);
	
	\node at (0, -1*\scale - 0.4) {$F_4$};

      \end{tikzpicture}
      \captionsetup{width=0.7\linewidth}
      \caption{The three graphs constructed from $K_4^-$ by connecting a pair of its vertices by a path of length two.}\label{fig:F23}
\end{figure}

To define our blocks, notice that condition~\ref{wb:subgraph} of Definition~\ref{def:well behaved} implies that $C$ does not contain copies of $K_4$ and the graphs $F_2$ and $F_3$ depicted in Figure~\ref{fig:F23}.  This in turn implies that every triangle in $C$ shares edges with at most one other triangle in $C$.  We define our \emph{blocks} $\cB=\cB(C)$ to be all copies of $K_4^-$ in $C$ and all triangles in $C$ that are not contained in a $K_4^-$; this definition guarantees that blocks are pairwise edge-disjoint. Fix an ordering $\sigma$ of $\cB$ such that every triangle in $\cB$ precedes every copy of $K_4^-$ and assign weights to the blocks in $\cB$ as follows:
\begin{enumerate}
	\item \label{stage:1} Assign weight $5$ to each vertex of $C$ and weight $-3$ to each edge. 
	\item \label{stage:2} For every $v\in V(C)$ contained in exactly one block, send its weight to this block. 
	\item \label{stage:3} For every vertex $v\in V(C)$ contained in more than one block, redistribute its weight equally to  the two smallest blocks  that contain $v$ according to the ordering $\sigma$ on $\cB$. Note that in this step, if $v$ is a vertex in $i\geq 0$  triangles in $\cB$, then   $i'\coloneqq\min\{i,2\}$ triangles containing $v$ and $2-i'$ copies of $K_4^-$ containing $v$ increase their  weight  by $5/2$.
	\item \label{stage:4} For every $e\in E(C)$ contained in a block, redistribute its weight to the block containing it (recall that blocks are pairwise edge-disjoint).
	\item \label{stage:5} For every $v\in V(C)$ not contained in a block, redistribute its weight equally to the edges incident to~$v$ (note that these edges were not yet handled, since they are also not in a block).
	\item \label{stage:6} For every $e\in E(C)$ not contained in a block, redistribute its weight in the following way: Since $e$ belongs to a copy of $F_0^-$ or $F_1^-$, at least one of its endpoints must also be a vertex in a block.
          \begin{enumerate}
          \item
            If only one of the endpoints belongs to some block, distribute $e$'s weight equally among all the blocks it belongs to.
          \item
            Otherwise, split $e$'s weight equally among its two endpoints and, for each of the endpoints, distribute the weight equally among all its blocks.	
          \end{enumerate}
\end{enumerate}

Note that by the end of this process, the total weight of all the edges and vertices in $C$ has been redistributed to $\cB$, and the total weight remains unchanged.  By our assumption that $e(C)/v(C) < 5/3$, the total weight of the vertices and edges before redistribution to blocks was positive, and therefore so is the total weight of all the blocks.  Therefore, $C$ must contain (at least) one positive-weight block $X\in \cB$. We will split the argument into two cases, depending on whether $X$ is a copy of $K_3$ or $K_4^-$.  In each case, we will find a removable set of edges.  Before embarking on this, we prove a technical claim that allows us to reason  about the weight of $X$ by inspecting the graph $C$ locally, without knowledge of the whole of $C$.

\begin{clm} \label{clm:local block}
  Suppose that $C'\subseteq C$ satisfy $\cB(C')=\cB(C) \eqqcolon \cB$ and let $w_C, w_{C'} \colon \cB \to \mathbb{R}$ be the weight assignments defined by the above process on $C$ and $C'$, respectively (with the same order $\sigma$ on~$\cB$). Then $w_C(X) \le w_{C^{\prime}}(X)$ for each $X \in \cB$.
\end{clm}

\begin{claimproof}
  Since stages~\ref{stage:1}--\ref{stage:4} depend only on the set of blocks, by the end of stage~\ref{stage:4}, and $\cB(C) = \cB(C')$, all the vertices, edges and blocks in $C$ and $C'$ have the same weight.  Now let $J$ be the graph comprising all the edges of $C$ that do not lie in a triangle (and so have not been dealt with by the end of Stage~\ref{stage:4} of the process).  It is enough to show that, after stage~\ref{stage:5}, the $C'$-weight of each edge of $J \cap C'$ is at least as large as its $C$-weight and that the $C$-weight of each edge of $J$ is at most $-1/2$.  This implies the assertion of the claim, as in stage~\ref{stage:6}, the change in weight of every block depends only on the edges of $J$ and their $C$-weights are negative and never larger than their $C'$-weights.

  Pick an arbitrary $e \in J$.  If both endpoints of $e$ lie in blocks, its weight is $-3$, in both processes.  Otherwise, exactly one endpoint of $e$ does not lie in a block.  If we denote this endpoint by $v$, then, for both $H \in \{C, C'\}$, the $H$-weight of $e$ at the end of stage~\ref{stage:5} is $-3+5/d_H(v)$.  Since $d_{C'}(v) \le d_C(v)$ for all $v\in V(C')$ and $d_C(v) \ge 2$ for all $v\in V(C)$, the $C$-weight of $e$ is at most $-1/2$ and not larger than its $C'$-weight.
\end{claimproof}

\begin{clm} \label{clm:Xistri}
  If $X$ is a positive-weight copy of $K_3$, then one of its edges $e \in E(X)$ is not in a $4$-cycle in a copy of $F_0^-$ in $C$.
\end{clm}

\begin{clm} \label{clm:Xisk4-}
  Suppose $X$ is a positive-weight copy of $K_4^-$ with edges $e_1,f_1,e_2,f_2,g$ such that $e_i,f_i$ and $g$ form a triangle for $i=1,2$. Then there exists an $i\in\{1,2\}$ such that neither $e_i$ nor $f_i$ belong to a $4$-cycle in a copy of $F_0^-$ in $C$. 
\end{clm}

Before proving these claims, let us see how  we can use them to contradict our assumption that $C$ is a minimal counterexample.  Firstly, consider the case that our positive-weight block $X$ is a triangle.  Let $e$ be an edge of $X$ from the assertion of Claim~\ref{clm:Xistri}. As shown at the beginning of the proof, $C\setminus \{e\}$ has a very good colouring.  We may extend this colouring to a very good colouring of $C$ as follows.  We colour $e$ blue unless the other two edges of $X$ are coloured blue, in which case we colour $e$ red.  As $X$ is the only triangle containing $e$, the colouring remains triangle-free and every edge of $C$ that is not in a triangle is still coloured blue.  Thus, we just need to show that there are no copies of $C_{rrbb}$ in $C$, see property~\ref{item:few Crrbbs} of Definition~\ref{def:Crrbb}.  Suppose that there was such a copy.  As the colouring of $C\setminus \{e\}$ is very good, this copy of $C_{rrbb}$ would contain $e$.  Corollary~\ref{cor:Crrbb containment} would then imply that $e$ lies in the $4$-cycle in a copy of $K_4^-$, contradicting the assumption that $X$ is a block, a copy of $F_0^-$, contrary to our choice of $e$, or a copy of $F_1^-$, contradicting the property~\ref{wb:subgraph} of being very well-behaved.

The case when $X$ is a copy of $K_4^-$ is resolved similarly.  Without loss of generality, we can assume that the edges of $X$ are labelled as in Claim~\ref{clm:Xisk4-} and neither $e_1$ nor $f_1$ belong to a $4$-cycle in a copy of $F_0^-$.  As above, $C \setminus \{e_1,f_1\}$ has a very good colouring, which we may extend to a very good colouring of $C$ as follows.  We colour $e_1$ red and $f_1$ blue unless that creates a copy of $C_{rrbb}$ with the edges $e_2$ and $f_2$, in which case we colour $e_1$ blue and $f_1$ red.  We claim that this gives a very good colouring of $C$.  Since the only triangle in $C$ containing $e_1$ or $f_1$ is the triangle containing both of them, the colouring remains triangle-free; every edge not in a triangle is still blue.  We just need to verify that there are no copies of $C_{rrbb}$.  As in the previous case, we can apply Corollary~\ref{cor:Crrbb containment} and rule out that $e_1$ and $f_1$ are in copies of $F_1^-$ and $F_0^-$ using condition~\ref{wb:subgraph} of being very well-behaved and the key property of $e_1$ and $f_1$ coming from Claim~\ref{clm:Xisk4-}.  The only case left to consider then, is that $e_1$ or $f_1$ lie in some copy of $C_{rrbb}$ that lies in a copy of $K_4^-$ in $C$.  Since $C$ is $\{ K_4, F_2, F_3 \}$-free, the only copy of $K_4^-$ containing our copy of $C_{rrbb}$ is $X$ itself, which is impossible, as we coloured $X$ to avoid having a copy of $C_{rrbb}$.

Now that we have proved that Claims~\ref{clm:Xistri} and~\ref{clm:Xisk4-} contradict the assumption that $C$ is a minimal counterexample, it remains only to prove these two claims.

\begin{claimproof}[Proof of Claim~\ref{clm:Xistri}]
  Denote $V(X)=\{x,y,z\}$ and suppose towards a contradiction that each of $xy$, $xz$ and $yz$ belongs to a $4$-cycle in some copy of $F_0^-$.  To get a contradiction, due to Claim~\ref{clm:local block}, it suffices to show that there is some $C^{\prime} \subseteq C$ such that $\cB(C)=\cB(C^{\prime})$ and $w_{C^{\prime}}(X) \le 0$. 

  We begin by considering $C'$ to be the union of all the triangles in $C$ (so that $\cB(C)=\cB(C')$).  By our assumption, each edge of $X$ must share a vertex with a triangle other than $X$; indeed, otherwise it cannot lie on a $4$-cycle in a copy of $F_0^-$.  Consequently, at least two of $X$'s vertices, say $y$ and $z$, are also in other triangles and hence blocks.  Therefore, their contribution to $X$'s weight (in stage~\ref{stage:3} of the weight redistribution process) is at most $5/2$ each, and so $w_{C'}(X) \le 1$.  We may further assume that $x$ does not lie in an additional triangle, since otherwise $w_{C'}(X) \le -3/2$, see Figure~\ref{fig1}.

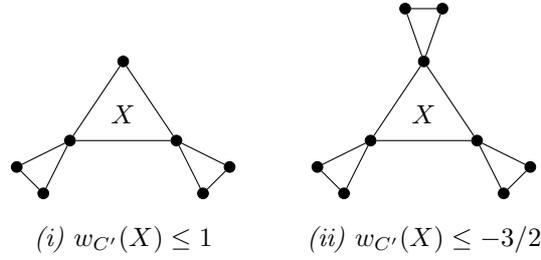
\begin{figure}[h] 
\centering
\begin{tikzpicture}[vertex/.style={draw,circle,color=black,fill=black,inner sep=1,minimum width=4pt}]
    
	\pgfmathsetmacro{\scale}{0.7}    
    
    \node[vertex] (v1) at (-1*\scale, 0*\scale) {};
    \node[vertex] (v2) at (1*\scale, 0*\scale) {};
    \node[vertex] (v3) at (0*\scale, 1.5*\scale) {};
    
    \node[vertex] (u1) at (-2*\scale, -0.5*\scale) {};
    \node[vertex] (u2) at (-1.5*\scale, -1*\scale) {};
    \node[vertex] (w1) at (2*\scale, -0.5*\scale) {};
    \node[vertex] (w2) at (1.5*\scale, -1*\scale) {};

    \draw[] (v1) to (v2);
    \draw[] (v1) to (v3);
    \draw[] (v3) to (v2);
    
    \draw[] (v1) to (u1);
    \draw[] (v1) to (u2);
    \draw[] (u1) to (u2);
    
    \draw[] (v2) to (w1);
    \draw[] (v2) to (w2);
    \draw[] (w1) to (w2);

	\node at (0, -1.3*\scale - 0.4) {\emph{(i)} $w_{C'}(X) \le 1$};
	\node at (0, 0.5*\scale) {$X$};
    
\end{tikzpicture}
\quad \quad
\begin{tikzpicture}[vertex/.style={draw,circle,color=black,fill=black,inner sep=1,minimum width=4pt}]
    
	\pgfmathsetmacro{\scale}{0.7}      
    
    \node[vertex] (v1) at (-1*\scale, 0*\scale) {};
    \node[vertex] (v2) at (1*\scale, 0*\scale) {};
    \node[vertex] (v3) at (0*\scale, 1.5*\scale) {};
    
    \node[vertex] (u1) at (-2*\scale, -0.5*\scale) {};
    \node[vertex] (u2) at (-1.5*\scale, -1*\scale) {};
    \node[vertex] (w1) at (2*\scale, -0.5*\scale) {};
    \node[vertex] (w2) at (1.5*\scale, -1*\scale) {};
    \node[vertex] (x1) at (0.35*\scale, 2.5*\scale) {};
    \node[vertex] (x2) at (-0.35*\scale, 2.5*\scale) {};

    \draw[] (v1) to (v2);
    \draw[] (v1) to (v3);
    \draw[] (v3) to (v2);
    
    \draw[] (v1) to (u1);
    \draw[] (v1) to (u2);
    \draw[] (u1) to (u2);
    
    \draw[] (v2) to (w1);
    \draw[] (v2) to (w2);
    \draw[] (w1) to (w2);
    
    \draw[] (v3) to (x1);
    \draw[] (v3) to (x2);
    \draw[] (x1) to (x2);
	
	\node at (0, -1.3*\scale - 0.4) {\emph{(ii)} $w_{C'}(X) \le -3/2$};
	\node at (0, 0.5*\scale) {$X$};
    
\end{tikzpicture}

\caption{Triangle configurations on the vertices of $X=K_3$.}\label{fig1}
\end{figure}

Now, since both $xy$ and $xz$ are in some copies of $F_0^-$ and $X$ is the only triangle that $x$ belongs to, $C$ must have an edge $xv$ that does not lie in a block; we now add $xv$ to $C'$.  If $v$ supported a triangle (and hence a block) in $C$, then $xv$ would contribute $-3/2$ to $w_{C'}(X)$ in stage~\ref{stage:6} of the weight redistribution process and therefore $w_{C'}(X)\le -1/2$ would be negative (see Figure~\ref{fig2} for illustration).  We may thus further assume that $v$ does not belong to a triangle in $C$.  Now, add to $C'$ all edges in copies of $F_0^-$ that  contain one of $xy$, $xz$.  If $xv$ lied in a $4$-cycle containing $xy$ and in a $4$-cycle containing $xz$, we would have that $d_{C'}(v) \ge 3$ and so $xv$ would have weight at most $-3+5/3 = -4/3$ after stage~\ref{stage:5}, all of which would go to $X$ in stage~\ref{stage:6}, and once again $w_{C'}(X) \le -1/3$ would be negative.  Thus, the $4$-cycles containing $xy$ and $xz$ do not share edges so there must be a vertex $u \neq v$ such that $xu \in C'$.  The contribution of each of $xu$ and $xv$ to $X$'s weight is at most $-1/2$ and therefore $w_{C'}(X) \le 0$, a contradiction. This concludes the proof of Claim~\ref{clm:Xistri}. \end{claimproof}

\begin{figure}[h] 
\centering
\begin{tikzpicture}[vertex/.style={draw,circle,color=black,fill=black,inner sep=1,minimum width=4pt}]
    
	\pgfmathsetmacro{\scale}{0.7}    
    
    \node[vertex] (v1) at (-1*\scale, 0*\scale) {};
    \node[vertex] (v2) at (1*\scale, 0*\scale) {};
    \node[vertex] (v3) at (0*\scale, 1.5*\scale) {};
    
    \node[vertex] (u1) at (-2*\scale, -0.5*\scale) {};
    \node[vertex] (u2) at (-1.5*\scale, -1*\scale) {};
    \node[vertex] (w1) at (2*\scale, -0.5*\scale) {};
    \node[vertex] (w2) at (1.5*\scale, -1*\scale) {};
    
    \node[vertex] (x1) at (0*\scale, 2.7*\scale) {};
    \node[vertex] (x2) at (0.35*\scale, 3.7*\scale) {};
    \node[vertex] (x3) at (-0.35*\scale, 3.7*\scale) {};

    \draw[] (v1) to (v2);
    \draw[] (v1) to (v3);
    \draw[] (v3) to (v2);
    
    \draw[] (v1) to (u1);
    \draw[] (v1) to (u2);
    \draw[] (u1) to (u2);
    
    \draw[] (v2) to (w1);
    \draw[] (v2) to (w2);
    \draw[] (w1) to (w2);
    
    \draw[] (v3) to (x1);
    
    \draw[] (x1) to (x2);
    \draw[] (x1) to (x3);
    \draw[] (x3) to (x2);
    
    \node at (-0.4*\scale, 1.5*\scale) {$x$};
    \node at (-0.4*\scale, 2.7*\scale) {$v$};
	
	\node at (0, -1.3*\scale - 0.4) {\emph{(i)} $w_{C'}(X) \le -1/2$};
	\node at (0, 0.5*\scale) {$X$};
    
\end{tikzpicture}
\quad \quad
\begin{tikzpicture}[vertex/.style={draw,circle,color=black,fill=black,inner sep=1,minimum width=4pt}]
    
	\pgfmathsetmacro{\scale}{0.7}    
    
    \node[vertex] (v1) at (-1*\scale, 0*\scale) {};
    \node[vertex] (v2) at (1*\scale, 0*\scale) {};
    \node[vertex] (v3) at (0*\scale, 1.5*\scale) {};
    
    \node[vertex] (u1) at (-2*\scale, -0.5*\scale) {};
    \node[vertex] (u2) at (-1.5*\scale, -1*\scale) {};
    \node[vertex] (w1) at (2*\scale, -0.5*\scale) {};
    \node[vertex] (w2) at (1.5*\scale, -1*\scale) {};
    
    \node[vertex] (x1) at (0*\scale, 2.7*\scale) {};

    \draw[] (v1) to (v2);
    \draw[] (v1) to (v3);
    \draw[] (v3) to (v2);
    
    \draw[] (v1) to (u1);
    \draw[] (v1) to (u2);
    \draw[] (u1) to (u2);
    
    \draw[] (v2) to (w1);
    \draw[] (v2) to (w2);
    \draw[] (w1) to (w2);
    
    \draw[] (v3) to (x1);
    
    \draw[] (x1) to (u1);
    \draw[] (x1) to (w1);
    
    \node at (-0.4*\scale, 1.5*\scale) {$x$};
    \node at (-0.4*\scale, 2.7*\scale) {$v$};
	
	\node at (0, -1.3*\scale - 0.4) {\emph{(ii)} $w_{C'}(X) \le -1/3$};
	\node at (0, 0.5*\scale) {$X$};
    
\end{tikzpicture}
\quad \quad
\begin{tikzpicture}[vertex/.style={draw,circle,color=black,fill=black,inner sep=1,minimum width=4pt}]
    
	\pgfmathsetmacro{\scale}{0.7}    
    
    \node[vertex] (v1) at (-1*\scale, 0*\scale) {};
    \node[vertex] (v2) at (1*\scale, 0*\scale) {};
    \node[vertex] (v3) at (0*\scale, 1.5*\scale) {};
    
    \node[vertex] (u1) at (-2*\scale, -0.5*\scale) {};
    \node[vertex] (u2) at (-1.5*\scale, -1*\scale) {};
    \node[vertex] (w1) at (2*\scale, -0.5*\scale) {};
    \node[vertex] (w2) at (1.5*\scale, -1*\scale) {};
    
    \node[vertex] (x1) at (-1.7*\scale, 2*\scale) {};
	\node[vertex] (x2) at (1.7*\scale, 2*\scale) {};

    \draw[] (v1) to (v2);
    \draw[] (v1) to (v3);
    \draw[] (v3) to (v2);
    
    \draw[] (v1) to (u1);
    \draw[] (v1) to (u2);
    \draw[] (u1) to (u2);
    
    \draw[] (v2) to (w1);
    \draw[] (v2) to (w2);
    \draw[] (w1) to (w2);
    
    \draw[] (v3) to (x1);
    \draw[] (v3) to (x2);
    
    \draw[] (x1) to (u1);
    \draw[] (x2) to (w1);
    
    \node at (0*\scale, 1.9*\scale) {$x$};    
    \node at (-1.7*\scale, 2.4*\scale) {$v$};
    \node at (1.7*\scale, 2.4*\scale) {$u$};

	\node at (0, -1.3*\scale - 0.4) {\emph{(iii)} $w_{C'}(X) \le 0$};
	\node at (0, 0.5*\scale) {$X$};
    
\end{tikzpicture}

\caption{Edge configurations on the third vertex of $X=K_3$.}\label{fig2}
\end{figure}
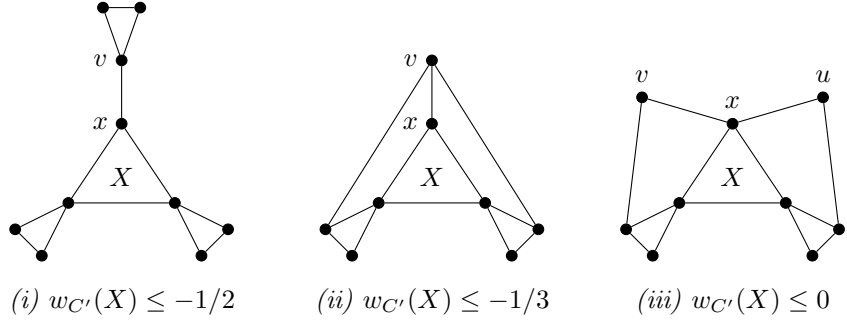

\begin{claimproof}[Proof of Claim~\ref{clm:Xisk4-}]  
  Denote $V(X)=\{x_1,x_2,y,z\}$ so that $g=yz$ and $e_i = x_iy, f_i =x_iz$ for $i \in \{1, 2\}$.  Assume towards contradiction that one of $e_1, f_1$ as well as one of $e_2, f_2$ belongs to a 4-cycle in a copy of $F_0^-$.  We let $C'$ be the union of all triangles in $C$, so that $\cB(C') = \cB(C)$.  By Claim~\ref{clm:local block}, it is enough to show that $w_{C'}(X) \le 0$.
  
  Now, note that stage~\ref{stage:4} of the redistribution process on $C'$ moves weight from the edges of $X$ to $X$.  If two or more vertices of $X$ belonged to blocks other than $X$, then the contribution to $X$'s weight coming from its vertices, in stages~\ref{stage:2} and~\ref{stage:3}, would be at most $2\cdot 5+2\cdot(5/2)\leq 15$ and  we would have $w_{C'}(X)\leq 0$.  Hence, it must be the case that at most one vertex in $X$ belongs to a block that is not $X$.  Moreover, if none of $y,z,x_i$ belonged to a block other than $X$, then none of $e_{i}, f_{i}$ would be contained in a 4-cycle of a copy of $F_0^-$.  Consequently, one of $y, z$ must belong to a block other than $X$; without loss of generality, assume that $y$ is the only vertex of $X$ contained in a block other than $X$.  Since neither $f_1$ nor $f_2$ can lie in a $4$-cycle of a copy of $F_0^-$, it must be that both $e_1$ and $e_2$ do.

  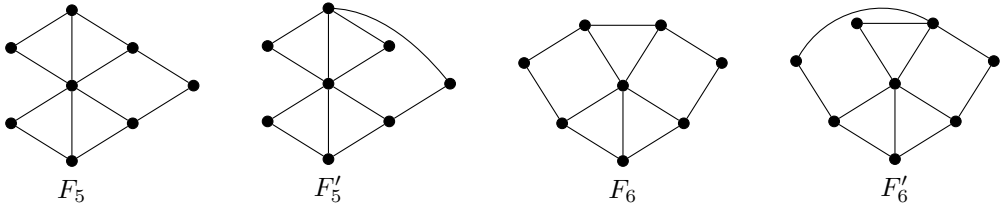
\begin{figure}[h] 
    \centering
    \begin{tikzpicture}[vertex/.style={draw,circle,color=black,fill=black,inner sep=1,minimum width=4pt}]
      
      \pgfmathsetmacro{\scale}{1}    
      
      \node[vertex] (v1) at (-0.8*\scale, 0.5*\scale) {};
      \node[vertex] (v2) at (-0.8*\scale, -0.5*\scale) {};
      \node[vertex] (v3) at (0*\scale, 1*\scale) {};
      \node[vertex] (v4) at (0*\scale, 0*\scale) {};
      \node[vertex] (v5) at (0*\scale, -1*\scale) {};
      \node[vertex] (v6) at (0.8*\scale, 0.5*\scale) {};
      \node[vertex] (v7) at (0.8*\scale, -0.5*\scale) {};
      \node[vertex] (v8) at (1.6*\scale, 0*\scale) {};

      \draw[] (v1) to (v4);
      \draw[] (v1) to (v3);
      \draw[] (v3) to (v4);
      
      \draw[] (v2) to (v4);
      \draw[] (v2) to (v5);
      \draw[] (v5) to (v4);
      
      \draw[] (v6) to (v4);
      \draw[] (v6) to (v3);

      \draw[] (v7) to (v4);
      \draw[] (v7) to (v5);
      
      \draw[] (v7) to (v8);
      \draw[] (v6) to (v8);
      
      \node at (0, -1*\scale - 0.4) {$F_5$};

    \end{tikzpicture}
    \quad \quad
     \begin{tikzpicture}[vertex/.style={draw,circle,color=black,fill=black,inner sep=1,minimum width=4pt}]
      
      \pgfmathsetmacro{\scale}{1}    
      
      \node[vertex] (v1) at (-0.8*\scale, 0.5*\scale) {};
      \node[vertex] (v2) at (-0.8*\scale, -0.5*\scale) {};
      \node[vertex] (v3) at (0*\scale, 1*\scale) {};
      \node[vertex] (v4) at (0*\scale, 0*\scale) {};
      \node[vertex] (v5) at (0*\scale, -1*\scale) {};
      \node[vertex] (v6) at (0.8*\scale, 0.5*\scale) {};
      \node[vertex] (v7) at (0.8*\scale, -0.5*\scale) {};
      \node[vertex] (v8) at (1.6*\scale, 0*\scale) {};

      \draw[] (v1) to (v4);
      \draw[] (v1) to (v3);
      \draw[] (v1) to (v6);
      
      \draw[] (v2) to (v4);
      \draw[] (v2) to (v5);
      \draw[] (v5) to (v4);
      
      \draw[] (v6) to (v4);
      \draw[] (v6) to (v3);

      \draw[] (v7) to (v4);
      \draw[] (v7) to (v5);
      
      \draw[] (v7) to (v8);
      \draw[] (v6) to (v8);
      
      \node at (0, -1*\scale - 0.4) {$F'_5$};

    \end{tikzpicture}
    \quad \quad 
    \begin{tikzpicture}[vertex/.style={draw,circle,color=black,fill=black,inner sep=1,minimum width=4pt}]
      
      \pgfmathsetmacro{\scale}{1}    
      
      \node[vertex] (v1) at (-0.8*\scale, 0.5*\scale) {};
      \node[vertex] (v2) at (-0.8*\scale, -0.5*\scale) {};
      \node[vertex] (v3) at (0*\scale, 1*\scale) {};
      \node[vertex] (v4) at (0*\scale, 0*\scale) {};
      \node[vertex] (v5) at (0*\scale, -1*\scale) {};
      \node[vertex] (v6) at (0.8*\scale, 0.5*\scale) {};
      \node[vertex] (v7) at (0.8*\scale, -0.5*\scale) {};
      \node[vertex] (v8) at (1.6*\scale, 0*\scale) {};

      \draw[] (v1) to (v4);
      \draw[] (v1) to (v3);
      \draw[] (v3) to (v4);
      
      \draw[] (v2) to (v4);
      \draw[] (v2) to (v5);
      \draw[] (v5) to (v4);
      
      \draw[] (v6) to (v4);
      \draw[] (v6) to (v3);

      \draw[] (v7) to (v4);
      \draw[] (v7) to (v5);
      
      \draw[] (v7) to (v8);
      
      \draw(v3) parabola (v8);
      
      \node at (0, -1*\scale - 0.4) {$F_5''$};
      
    \end{tikzpicture}
    \quad \quad
    \begin{tikzpicture}[vertex/.style={draw,circle,color=black,fill=black,inner sep=1,minimum width=4pt}]
      
      \pgfmathsetmacro{\scale}{1}    
      
      \node[vertex] (v1) at (-1.3*\scale, 0.3*\scale) {};
      \node[vertex] (v2) at (-0.8*\scale, -0.5*\scale) {};
      \node[vertex] (v3) at (-0.5*\scale, 0.8*\scale) {};
      \node[vertex] (v4) at (0*\scale, 0*\scale) {};
      \node[vertex] (v5) at (0*\scale, -1*\scale) {};
      \node[vertex] (v6) at (0.5*\scale, 0.8*\scale) {};
      \node[vertex] (v7) at (0.8*\scale, -0.5*\scale) {};
      \node[vertex] (v8) at (1.3*\scale, 0.3*\scale) {};

      \draw[] (v1) to (v2);
      \draw[] (v1) to (v3);
      \draw[] (v3) to (v4);
      
      \draw[] (v2) to (v4);
      \draw[] (v2) to (v5);
      \draw[] (v5) to (v4);
      
      \draw[] (v6) to (v4);
      \draw[] (v6) to (v3);

      \draw[] (v7) to (v4);
      \draw[] (v7) to (v5);
      
      \draw[] (v7) to (v8);
      \draw[] (v6) to (v8);
      
      \node at (0, -1*\scale - 0.4) {$F_6$};
      
    \end{tikzpicture}
    \quad \quad
    \begin{tikzpicture}[vertex/.style={draw,circle,color=black,fill=black,inner sep=1,minimum width=4pt}]
      
      \pgfmathsetmacro{\scale}{1}    
      
      \node[vertex] (v1) at (-1.3*\scale, 0.3*\scale) {};
      \node[vertex] (v2) at (-0.8*\scale, -0.5*\scale) {};
      \node[vertex] (v3) at (-0.5*\scale, 0.8*\scale) {};
      \node[vertex] (v4) at (0*\scale, 0*\scale) {};
      \node[vertex] (v5) at (0*\scale, -1*\scale) {};
      \node[vertex] (v6) at (0.5*\scale, 0.8*\scale) {};
      \node[vertex] (v7) at (0.8*\scale, -0.5*\scale) {};
      \node[vertex] (v8) at (1.3*\scale, 0.3*\scale) {};

      \draw[] (v1) to (v2);
      \draw[] (v3) to (v4);
      
      \draw[] (v2) to (v4);
      \draw[] (v2) to (v5);
      \draw[] (v5) to (v4);
      
      \draw[] (v6) to (v4);
      \draw[] (v6) to (v3);

      \draw[] (v7) to (v4);
      \draw[] (v7) to (v5);
      
      \draw[] (v7) to (v8);
      \draw[] (v6) to (v8);
      
      \draw(v1) to[out=60, in=150] (v6);
      
      \node at (0, -1*\scale - 0.4) {$F_6'$};
      
    \end{tikzpicture}
    \caption{The four graphs appearing in $X \cup X_1’ \cup X_2' \cup C_1 \cup C_2$, each of which has 8 vertices and 12 edges.}\label{fig:Xisk4-}
  \end{figure}

  For each $i \in \{1, 2\}$, denote by $C_i$ the $4$-cycle in a copy of $F_0^-$ that passes through $e_i$.  We claim that $e_i$ is the only edge of $X$ in $C_i$.   Indeed, the union of a copy of $K_4^-$ and a copy of $F_0^-$ whose $4$-cycle intersects this $K_4^-$ in more than one edge contains a copy of $F_2$, $F_3$ or $F_4$ depicted in Figure~\ref{fig:F23}.  However, $C$ cannot contain any of these graphs, see condition~\ref{wb:subgraph} in Definition~\ref{def:well behaved}.  Further, the second edge of $C_i$ that is incident with $y$ belongs to some block $X_i' \neq X$.  We claim that $X_i'$ is not a copy of $K_4^-$.  Indeed, if it were, then $X \cup X_i' \cup C_i$ would be a copy of either $F_5, F_5'$ or $F_5''$ from Figure~\ref{fig:Xisk4-} (recall that $C_i$ is not allowed to intersect a copy of $K_4^-$ in more than one edge) and this graph is too dense to be contained in $C$, see condition~\ref{wb:subgraph} in Definition~\ref{def:well behaved}.  Therefore, both $X_1'$ and $X_2'$ are triangle blocks.  Finally, if $X_1' \neq X_2'$, then, since the ordering $\sigma$ prioritises triangles over copies of $K_4^-$, the vertex $y$ would give none of its weight to $X$ in stage~\ref{stage:3} of the redistribution process, yielding $w_{C'}(X)\leq 0$, as desired. Thus $X' \coloneqq X_1' = X_2'$ is a triangle block. Moreover, each $C_i$ shares only one edge with $X'$, as otherwise $X' \cup C_i$ would be a copy of $K_4^-$, which is impossible due to the assumption that $X'$ is a triangle block.  Consequently, each $C_i$ contains a unique vertex $w_i \notin V(X) \cup V(X')$.  If $w_1 = w_2$, then $X \cup C_1 \cup C_2$ contains a copy of $F_4$, which is too dense to be contained in $C$, so we may assume that $w_1 \neq w_2$.  But then, $X \cup X' \cup C_1 \cup C_2$ would be a copy of one of $F_6$ or $F_6'$ from Figure~\ref{fig:Xisk4-}, a contradiction.
\end{claimproof}
This concludes the proof of Lemma~\ref{lem:collagescolorable}.
\end{proof}

\section{Proof of the 1-statements}
\label{sec:1-statements}

In this section, we prove our $1$-statements, establishing the upper bounds on $q(n;K_3,p)$ in Theorem~\ref{thm:main}.  Our aim is to prove that, if $q\gg q(n;K_3,p)$, then  a.a.s.\ \emph{no} $K_3$-free colouring of $G_{n,p}$ can be extended to the edges of an independent copy $G_{n,q}$ without creating monochromatic triangles.  We will achieve this by showing that every $K_3$-free colouring of the edges of a typical $G_{n,p}$ results in many local obstructions -- individual edges or copies of $K_{1,2}$  that one cannot colour without introducing a monochromatic triangle, see Figure~\ref{fig:Crbbbb}. More precisely, we will show that there are either $\omega(q^{-1})$ such \emph{dangerous edges} or $\omega(q^{-2})$ such \emph{dangerous copies of $K_{1,2}$} in $K_n$.  Standard probabilistic arguments will then show that a.a.s.\ at least one such local obstruction will appear in $G_{n,q}$, precluding the existence of a $K_3$-free extension of our colouring of $G_{n,p}$.  In fact, with just a little more work, it will be enough for us to find either $\omega(q^{-1})$ copies of $C_{rrbb}$, the $4$-cycle whose edges are coloured red, red, blue, blue, or $\omega(q^{-2})$ copies of $C_{rbbbb}$, the $5$-cycle with four edges coloured blue and one edge coloured red.

\begin{prop}
  \label{prop: many Crrbb}
  Suppose that $n^{-2/3} \ll p \ll n^{-1/2} $, $t\geq n^7p^{10}$ and $t^{-1}\ll q<1$ and let  $G_1 \sim G_{n,p}$ and $G_2 \sim G_{n,q}$ be independent. Then a.a.s.\ any $G_1$-measurable colouring that contains at least $t$ copies of $C_{rrbb}$ can be extended to a triangle-free colouring of $G_1 \cup G_2$.
\end{prop}

\begin{prop} \label{prop: many Crbbbb}
  Suppose that $n^{-2/3} \ll p \ll n^{-1/2} $, $t\geq n^7p^9$ and $t^{-1/2}\ll q\leq 1$ and let  $G_1 \sim G_{n,p}$ and $G_2 \sim G_{n,q}$ be independent. Then a.a.s.\ any $G_1$-measurable colouring that contains at least $t$ copies of $C_{rbbbb}$ can be extended to a triangle-free colouring of $G_1 \cup G_2$.
\end{prop}

In order to find the required number of copies of $C_{rrbb}$ or $C_{rbbbb}$, we will use three different arguments.  We first split our analysis depending on the structure of the colouring.  If the colouring is \emph{balanced}, in that a positive proportion of the edges of $G_{n,p}$ are coloured in each colour, then the number of $C_{rrbb}$s is of order $n^4p^4$.  We prove this in Section~\ref{sec:balanced} using the method of hypergraph containers (Theorem~\ref{thm:containers}).  Noting that $n^{-4}p^{-4}\ll n^{-6}p^{-8}\ll n^{-3}p^{-7/2}$ in our full range $n^{-2/3}\ll p\ll n^{-1/2}$, this settles the desired result for balanced colourings in both the lower and upper ranges. 

It thus remains to consider colourings that are \emph{unbalanced}, that is, when there is one colour, say $\red$, that appears only $o(n^2p)$ times.  Here, we need more delicate arguments based on Janson's inequality and careful union bounds over all unbalanced colourings. In Section~\ref{sec:1 lower}, we show that every unbalanced colouring contains $\Omega(n^6p^7)$ copies of $C_{rbbbb}$, which implies, by Proposition~\ref{prop: many Crbbbb}, that no such colouring can be extended to a typical copy of $G_{n,q}$ as soon as $q \gg n^{-3}p^{-7/2}$.  Combining this result with the case of balanced colourings covers all possible colourings and shows that $q(n;K_3,p)\leq n^{-3}p^{-7/2}$ in the full range of interest $n^{-2/3}\ll p\ll n^{-1/2}$.  Finally, in Section~\ref{sec:1 upper}, we improve on this in the upper range, when $p\gg n^{-3/5}$, showing that every unbalanced colouring contains $\Omega(n^6p^8)$ copies of $C_{rrbb}$, which renders any unbalanced colouring nonextendable to $G_{n,q}$ as soon as $q\gg n^{-6}p^{-8}$, see Proposition~\ref{prop: many Crrbb}.  Here, in order to perform a union bound over all unbalanced, $K_3$-free colourings of $G_{n,p}$, we face some serious technicalities when $p$ approaches $n^{-3/5}$.

We complete this lengthy introduction with proofs of Propositions~\ref{prop: many Crrbb} and~\ref{prop: many Crbbbb} that supply sufficient conditions on nonextendability of colourings in terms of the number of copies of $C_{rrbb}$ and $C_{rbbbb}$.

\begin{proof}[Proof of Proposition~\ref{prop: many Crrbb}]
  Fix some $G_1$ satisfying property~\ref{item:K210-count} of Lemma~\ref{lem:randomgraph} (which occurs a.a.s.) and some colouring $\varphi \colon E(G_1)\to \{\red,\blue\}$ that contains at least $t$ copies of $C_{rrbb}$.  Recall that a pair of vertices $\{x,y\} \in E(K_n)$ is \emph{dangerous} if it is the `colour-splitting' diagonal of at least one such $C_{rrbb}$, that is, if there exist $u,v\in V(G_1)\setminus \{x,y\}$ with $xu,yu,xv,yv\in E(G_1)$, $\varphi(xu)=\varphi(yu)=\blue$ and $\varphi(xv)=\varphi(yv)=\red$, see Figure~\ref{fig:Crbbbb}.  We will show that there are at least $t/50$ dangerous pairs.  Note that this immediately implies the assertion of the lemma.  Indeed, the number of dangerous pairs that appear in $G_{n,q}$ is bounded from below by a $\Bin(t/50, q)$, which is positive with probability $1-o(1)$, by our assumption that $tq \gg 1$.
  
  Let $D$ be the collection of dangerous pairs.  For each pair $\rho\in D$, let $r_\rho\geq 1$ be the number of red copies of $K_{1,2}$ in $G_1$ that form a triangle with $\rho$ and likewise let $b_\rho$  be the number of blue copies of $K_{1,2}$ in $G_1$ that form a triangle with $\rho$, so that $\rho$ is the colour-splitting diagonal
for $r_\rho b_\rho$ copies of $C_{rrbb}$ in $G_1$ and $ \sum_{\rho\in D}r_\rho b_\rho\geq t$.  We further say that $\rho\in D$ is \emph{heavy} if $r_\rho b_\rho\geq 25$ and let $D_H\subseteq D$ be the collection of heavy dangerous pairs. Now, for each heavy pair $\rho$, by the AM-GM inequality, we have that $r_\rho+b_\rho\geq 2\sqrt{r_\rho b_\rho}\geq 10$ and $\rho$ forms the part of size 2 in $\binom{r_\rho + b_\rho}{10}$ copies of $K_{2,10}$. Therefore, by the assumed conclusion of  Lemma~\ref{lem:randomgraph}~\ref{item:K210-count},
\[\sum_{\rho\in D_H}\frac{r_\rho b_\rho}{25}\leq
  \sum_{\rho\in D_H} \left(\frac{r_\rho +b_\rho}{10}\right)^2 \leq  \sum_{\rho\in D_H}  \left(\frac{r_\rho +b_\rho}{10}\right)^{10} \leq \sum_{\rho\in D_H} \binom{r_\rho +b_\rho}{10}\leq N_{K_{2,10}}(G_1)\leq n^{11}p^{18}\leq \frac{t}{50}, \]
where in the last inequality we used that $n^{11}p^{18}\ll n^7p^{10}\leq t$ due to the fact that $p\ll n^{-1/2}$. Hence $\sum_{\rho\in D\setminus D_H}r_\rho b_\rho\geq t/2$ and, as each $\rho\in D\setminus D_H$ has $r_\rho b_\rho\leq 25$, we indeed obtain $|D|\geq |D\setminus D_H|\geq t/50$. 
\end{proof}

\begin{proof}[Proof of Proposition~\ref{prop: many Crbbbb}]

Fix some $G_1$ satisfying properties~\ref{item:bound max deg} and~\ref{item:K210-count} of Lemma~\ref{lem:randomgraph} (which occur a.a.s.) and some colouring $\varphi \colon E(G_1)\to \{\red,\blue\}$ that contains at least $t$ copies of $C_{rbbbb}$. A copy $K$ of $K_{1,2}$ in $G_1$ with vertices $w,u_1,u_2$ (so that $K$ is formed from edges $wu_i$ for $i=1,2$) is \emph{dangerous} if there are distinct vertices $w_1,w_2\in V(G)\setminus \{w,u_1,u_2\}$ such that $u_1u_2,u_1w_1,w_1w,ww_2,w_2u_2\in E(G_1)$, 
  $\varphi(u_1u_2)=\red$
  and  
  $\varphi(u_1w_1)=\varphi(w_1w)=\varphi(ww_2)=\varphi(w_2u_2)=\blue$. We say that $K$ \emph{hosts} this copy of $C_{rbbbb}$ on vertices $u_1,u_2,w_2,w,w_1$. See Figure~\ref{fig:Crbbbb} for a depiction. 
  
  \begin{figure}[h]
    \centering
    \begin{tikzpicture}[vertex/.style={draw,circle,color=black,fill=black,inner sep=1,minimum width=4pt}]
    
	\pgfmathsetmacro{\scale}{0.7}    
    
    \node[vertex] (y) at (2*\scale, 0*\scale) {};
    \node[vertex] (x) at (-2*\scale, 0*\scale) {};
    \node[vertex] (v) at (0*\scale, 2*\scale) {};
    \node[vertex] (u) at (0*\scale, -2*\scale) {};
    
    \draw[thick] (x) to (y);

    \draw[thick,red] (x) to (v);
    \draw[thick,red] (y) to (v);
    
    \draw[thick,blue] (x) to (u);
    \draw[thick,blue] (y) to (u);
	
    \node at (2*\scale+0.4, 0*\scale) {$y$};
    \node at (-2*\scale-0.4, 0*\scale) {$x$};
    \node at (0*\scale, 2*\scale+0.4) {$v$};
    \node at (0*\scale, -2*\scale-0.4) {$u$};
    
  \end{tikzpicture}
  \qquad\qquad\qquad
  \begin{tikzpicture}[vertex/.style={draw,circle,color=black,fill=black,inner sep=1,minimum width=4pt}]
    
    \pgfmathsetmacro{\scale}{0.7}    
    
    \node[vertex] (w2) at (2*\scale, 0*\scale) {};
    \node[vertex] (w1) at (-2*\scale, 0*\scale) {};
    \node[vertex] (w) at (0*\scale, 2*\scale) {};
    
    \node[vertex] (u1) at (-1*\scale, -2*\scale) {};
    \node[vertex] (u2) at (1*\scale, -2*\scale) {};

    \draw[thick] (u1) to (w);
    \draw[thick] (u2) to (w);
    \draw[thick,red] (u1) to (u2); 
    
	\draw[thick,blue] (u1) to (w1);
	\draw[thick,blue] (u2) to (w2);
	\draw[thick,blue] (w) to (w1);
		\draw[thick,blue] (w) to (w2);
	
	\node at (2*\scale+0.4, 0*\scale) {$w_2$};
	\node at (-2*\scale-0.4, 0*\scale) {$w_1$};
	\node at (0*\scale, 2*\scale+0.4) {$w$};
	\node at (-1*\scale, -2*\scale-0.4) {$u_1$};
	\node at (1*\scale, -2*\scale-0.4) {$u_2$};;
    
\end{tikzpicture}

\caption{A dangerous edge and a dangerous copy of $K_{1,2}$ in $G$.}\label{fig:Crbbbb}
\end{figure}
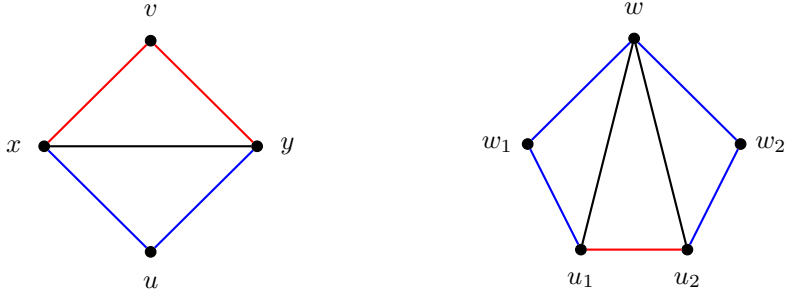
  
  Moreover, for such a dangerous copy $K$ of $K_{1,2}$, we define $x^K_1$ to be the number of choices of $w_1$ such that $\phi(u_1w_1)=\phi(w_1w)=\blue$ and $x^K_2$  to be the number of choices of $w_2$ such that $\phi(u_2w_2)=\phi(w_2w)=\blue$.  
  Therefore, each $K$ hosts at most $x_1^Kx_2^K$ copies of $C_{rbbbb}$ (this is not equality as some choices could have $w_1=w_2$).  Taking $\cK$ to be the collection of dangerous copies of $K_{1,2}$ on $V(G)$, we then have that  $\sum_{K\in \cK}x^K_1x^K_2\geq t$.  As in the proof of Proposition~\ref{prop: many Crrbb}, our aim is to prove that $\cK$ is large.  Define a copy $K$ to be \emph{heavy} if $m^K \coloneqq\max\{x^K_1,x^K_2\} \ge 10$ and let $\cK_H$ be the collection of heavy dangerous copies of $K_{1,2}$. We then have that 
  \[
    \sum_{K\in \cK_H}x_1^Kx_2^K\leq \sum_{K\in \cK_H}(m^K)^2\leq \sum_{K\in \cK_H}100\left(\frac{m^K}{10}\right)^{10}\leq 100\sum_{K\in \cK_H}\binom{m^K}{10}.
  \]
  We claim that the sum in the right-hand side of the above inequality is at most $N_{K_{2,10}^+}(G_1)$, where $K^+_{2,10}$ is the graph obtained  from $K_{2,10}$ by adding a pendant edge to one of its vertices of degree 10.  Indeed, fixing some $K$ in the summand, label the vertices  $u_1, u_2$ and $w$ as we did above and suppose that $m^K$ is achieved by $x_i^K$ for $i\in[2]$ (if $x_1^K=x^K_2$ then choose $i$ arbitrarily).  Then, for every set $W_i$ of $10$ vertices that can play the role of $w_i$ in the sense that they are all connected to both $u_i$ and $w$ by blue edges in $G_1$, we get a copy of $K^+_{2,10}$ on the vertices $w,u_1,u_2$ and $W_i$. This gives $\binom{m^K}{10}$ copies of  $K^+_{2,10}$ in the summand corresponding to $K$; these copies are distinct, as each copy of $K^+_{2,10}$ determines $K$ completely. So we have that 
  \[\sum_{K\in \cK_H}x_1^Kx_2^K\leq 100N_{K^+_{2,10}}(G_1)\leq 100\cdot n^{11}p^{18}\cdot 4np \leq 400 n^{12}p^{19}\leq  t/2,\]
  where we used properties~\ref{item:bound max deg} and~\ref{item:K210-count} of Lemma~\ref{lem:randomgraph} to bound $N_{K^+_{2,10}}(G_1)$ and we used that $n^{12}p^{19}\ll n^7p^9\leq t$ in the last inequality. This implies that $\sum_{K\in \cK\setminus \cK_H}x^K_1x^K_2\geq t/2$ and as every $K\in \cK\setminus \cK_H$ has $x^K_1x^K_2\leq 9^2\le 100$, we have that $|\cK|\geq |\cK\setminus \cK_H|\geq t/200$.
  
  Now, taking $G_2\sim G_{n,q}$, for each dangerous copy $K\in \cK$, let $I_K \coloneqq \mathds{1}[K\subseteq G_2]$ be the indicator random variable for the event that both edges of $K$ appear in $G_2$ and let $X \coloneqq \sum_{K\in \cK} I_K$.  As each dangerous copy of $K_{1,2}$ appears with probability $q^2$, we have that $\mu \coloneqq \mathbb{E}[X] \geq tq^2/200\gg 1$. Moreover, writing $K \sim K'$ when a pair $K, K'$ of dangerous copies of $K_{1,2}$ share at least one edge, we have
  \[
    \Delta\coloneqq\sum_{K\sim K'}\mathbb{E}[I_KI_{K'}]\leq \mu \cdot \big(1 + 4\Delta(G_1)q\big) \leq 2\max\{ \mu, 8npq\mu\},
  \]
  using again the assumed conclusion of  Lemma~\ref{lem:randomgraph}~\ref{item:bound max deg} in $G_1$.  
  Indeed, given some copy $K$ of $K_{1,2}$, we can obtain an upper bound on the number of dangerous $K'$ that intersect $K$ (but are not equal to $K$) by the number of choices of an edge $e$ of $K$ and a $G_1$-neighbour of one of the endpoints of $e$.  Using that 
  \[\frac{\mu^2}{npq\mu}\geq \frac{tq}{200np}\geq \frac{t^{1/2}}{200np}\geq \frac{n^{5/2}p^{7/2}}{200}\gg 1,\]  
  it follows from Janson's inequality (Lemma~\ref{janson}) that
  \[
    \Pr[X = 0] \le \exp\left(-\frac{\mu^2}{2\Delta}\right) \ll 1.
  \]
  Therefore, a.a.s.\ there is a dangerous copy of $K_{1,2}$, on vertices $w,u_1,u_2$ say, such that $E(K)=\{wu_1,wu_2\}\subseteq E(G_2)$.  This precludes the possibility of extending $\varphi$ to $G_2$.  Indeed, as $K$ is dangerous, it hosts some copy of $C_{rbbbb}$ in $G$ (under $\varphi$). If either $wu_1$ or $wu_2$ are coloured blue, then they will form a blue triangle with edges in the copy of $C_{rbbbb}$ whilst if they are both red then there is a red triangle formed with the edge $u_1u_2$. This completes the proof. 
 \end{proof}

\subsection{Balanced colourings} \label{sec:balanced}

In this section, we prove the following proposition which deals with balanced colourings of $G_{n,p}$ in our full range of interest. 

\begin{prop}
  \label{prop:balanced-colouring}
  For every  $\beta>0$, there exists a $\lambda>0$ such that the following holds.  Suppose that $ n^{-2/3}\ll p \ll n^{-1/2}$ and let $G \sim G_{n,p}$.  Then, a.a.s.\ every colouring $\varphi \colon E(G) \to \{\red, \blue\}$ such that $\left| \varphi ^{-1}(c)\right| \ge \beta n^2p$ for both $c\in \{\red,\blue\}$,  contains at least $\lambda n^4p^4$ copies of $C_{rrbb}$.
\end{prop}

As $n^4p^4\gg n^6p^8$ for $p\ll n^{-1/2}$, Propositions~\ref{prop: many Crrbb} and~\ref{prop:balanced-colouring} give that, for $q\gg n^{-6}p^{-8}$, a.a.s.\ no balanced $K_3$-free colouring of $G_{n,p}$ with $n^{-2/3}\ll p\ll n^{-1/2}$ can be extended to the edges of $G_{n,q}$ without creating monochromatic triangles.  Before embarking on the proof of Proposition~\ref{prop:balanced-colouring}, we need a deterministic lemma that deals with the complete  graph, that is, the case $p=1$ in the proposition.

\begin{lem}
  \label{lem:balanced-colouring}
  For every $\beta > 0$, there exists a $\lambda$ such that the following holds.  For all sufficiently large $n$, every $\psi \colon E(K_n) \to \{\red, {\blue}\}$ that satisfies $|\psi^{-1}(c)| \ge \beta n^2$ for both $c \in \{\red, \blue\}$  contains at least $\lambda n^4$ copies of $C_{rrbb}$.
\end{lem}
\begin{proof}
  For every ordered pair $x, y$ of distinct vertices, let $V_{x,y}$ denote the number of $z$ such that $\psi(xz) = \red$ and $\psi(yz) = \blue$.  Letting $C$ be the number of copies of $C_{rrbb}$, we have, by convexity,
  \[
    C = \sum_{x,y} \binom{V_{x,y}}{2} \ge n(n-1) \cdot \binom{\bar{V}}{2},
  \]
  where
  \[
    \bar{V} = \frac{1}{n(n-1)} \cdot \sum_{x,y} V_{x,y}.
  \]
  Observe that, if $a,b,c,d$ are four distinct vertices such that $\psi(ab) = \red$ and  $\psi(cd) = \blue$, then either $ab$ and $bc$ or $bc$ and $cd$ are counted by $V_{a,c}$ or $V_{b,d}$, respectively. This implies that, for all large $n$,
  \[
    \bar{V} \ge \frac{1}{n(n-1)} \cdot \frac{\beta n^2 \cdot (\beta n^2 - 2n)}{2n} \ge
    \frac{\beta^2 n}{4},
  \]
  which gives the claimed lower bound on $C$.
\end{proof}

We now turn to proving Proposition~\ref{prop:balanced-colouring}

\begin{proof}[Proof of Proposition~\ref{prop:balanced-colouring}] We can assume that $0<\beta<1/4$. 
  Let $\cH$ be the $4$-uniform hypergraph with vertex set $E(K_n) \times \{\red,\blue\}$ whose edges are all copies of $C_{rrbb}$ in $K_n$, that is, sets of the form
  \[
    \big\{(uv,\red), (uw,\red), (vx,\blue), (wx,\blue)\big\},
  \]
  where $u$, $v$, $w$, and $x$ are any four distinct vertices of $K_n$. Observe that
  \[
    v(\cH) = 2\binom{n}{2}, \quad e(\cH) = 12 \cdot \binom{n}{4},  \quad \Delta_2(\cH) \le n, \quad \Delta_3(\cH) = \Delta_4(\cH) = 1.
  \]

  Let $\eps \coloneqq \lambda_{\ref{lem:balanced-colouring}}(\beta/2)/2$ and let $\delta>0$ and $t\in \N$ be the constants provided by Theorem~\ref{thm:containers} invoked with $k = 4$ and $K = 1$.  Set
  \[
    \gamma \coloneqq \min\left\{\eps, \frac{\beta}{32}\right\},
    \qquad
    \sigma \coloneqq \min\left\{\frac{\beta}{4t}, \frac{\gamma^2}{2^t}\right\}
    \qquad
    \text{and}
    \qquad
    \lambda \coloneqq \frac{\delta \sigma^4}{4}.
  \]
  Let $\cI(\cH)$ be the family of all sets $I \subseteq V(\cH)$ that induce fewer than $\lambda n^4p^4$ edges of $\cH$.  Since $p \gg n^{-2/3}$, we may apply Theorem~\ref{thm:containers} to $\cH$ with $\tau \coloneqq \sigma p$ to obtain a function $f \colon \cP(V(\cH))^t \to \cP(V(\cH))$ such that:
  \begin{enumerate}[label=(\roman*)]
  \item
    For every $I \in \cI(\cH)$, there are $S_1, \dotsc, S_t \subseteq V(\cH)$ each of size at most $\tau v(\cH)$  such that $S_1 \cup \dotsb \cup S_t \subseteq I \subseteq f(S_1, \dotsc, S_t)$.
  \item \label{item: container app ii}
    For each $S_1, \dotsc, S_t \subseteq V(\cH)$, the set $f(S_1, \dotsc, S_t)$ induces fewer than $\eps n^4$ edges in $\cH$.
  \end{enumerate}
  Suppose that $\varphi$ is a \emph{bad} colouring of $G$, that is, a colouring with $|\varphi^{-1}(c)| \ge \beta n^2p$ for both $c \in \{\red, \blue\}$ but fewer than $\lambda n^4p^4$ copies of $C_{rrbb}$.  Then $\varphi \in \cI(\cH)$ and thus $\varphi \subseteq f(S_1, \dotsc, S_t)$ for some $S_1, \dotsc, S_t \subseteq \varphi$.  We will call such $\vS \coloneqq (S_1, \dotsc, S_t)$ the \emph{signature} of $\varphi$ and denote it by $\sig(\varphi)$.  Denote by $\pi \colon V(\cH) \to E(K_n)$ the projection to the first coordinate and, with slight abuse of notation, let $\pi(\vS) \coloneqq \pi(S_1 \cup \dotsb \cup S_t)$;  note that $\pi(\sig(\varphi)) \subseteq G$.

  \begin{clm}
    \label{clm:colouring-sig}
    For every $S_1, \dotsc, S_t \subseteq V(\cH)$, letting $\vS \coloneqq (S_1, \dotsc, S_t)$, we have
    \[
      \Pr\big[\text{$G$ has a bad colouring $\varphi$ with $\sig(\varphi) = \vS$}\big] \le \Pr\big[ \pi(\vS) \subseteq G\big] \cdot \exp(-\gamma n^2p).
    \]
  \end{clm}
  \begin{claimproof}
    Suppose that $G$ has a bad colouring $\varphi$ with $\sig(\varphi) = \vS$.  This means, in particular, that $\pi(\vS) \subseteq G$, so it is enough to show that
    \[
      \Pr\big[\text{$G$ has a bad colouring $\varphi$ with $\sig(\varphi) = \vS$} | \pi(\vS) \subseteq G\big] \le \exp(-\gamma n^2p).
    \]
    Define
    \begin{align*}
      R(\vS) & \coloneqq \big\{e \in E(K_n) : (e,\red) \in f(\vS)\big\}, \\
      B(\vS) & \coloneqq \big\{e \in E(K_n) : (e,\blue) \in f(\vS)\big\}, \\  
      X(\vS) & \coloneqq E(K_n) \setminus \big(R(\vS) \cup B(\vS)\big)
    \end{align*}
    and observe that $\varphi \subseteq f(\vS)$ means that $G$ is disjoint from $X(\vS)$ and that
    \[
      \varphi^{-1}(\red) \subseteq  R(\vS) \cap G \qquad \text{and} \qquad  \varphi^{-1}(\blue) \subseteq  B(\vS) \cap G.
    \]
    We claim that at least one of the following must be true:
    \begin{enumerate}[label=(\alph*)]
    \item
      \label{item:sig-uncoloured}
      The set $X(\vS)$ has at least $\eps n^2$ edges.
    \item
      \label{item:sig-unbalanced}
      One of the sets $R(\vS)$ or $B(\vS)$ has at most $\beta n^2/2$ edges.
    \end{enumerate}
    Suppose that~\ref{item:sig-unbalanced} does not hold and let $\psi \colon E(K_n) \to \{\red, \blue\}$ be an arbitrary colouring of $K_n$ satisfying $\psi^{-1}(\red) \subseteq R(\vS) \cup X(\vS)$, $\psi^{-1}(\blue) \subseteq B(\vS) \cup X(\vS)$ and $|\psi^{-1}(c)| \ge \beta n^2/2$ for each $c \in \{\red, \blue\}$; such a colouring exists as $\lceil \beta n^2/2 \rceil \le \lfloor \binom{n}{2} / 2\rfloor$ due to our upper bound on $\beta$.  It follows from Lemma~\ref{lem:balanced-colouring} and our definition of $\eps$ that $\psi$ has at least $2\eps n^4$ copies of $C_{rrbb}$.  Any such copy corresponds to an edge of $\cH[f(\vS)]$ unless it contains an edge of $X(\vS)$.  However, the number of $4$-cycles with an edge of $X(\vS)$ is at most $X(\vS) \cdot n^2$.  This implies that $e(\cH[f(\vS)]) \ge 2\eps n^4 - |X(\vS)| \cdot n^2$, which gives $|X(\vS)| > \eps n^2$ due to condition~\ref{item: container app ii} on $f(\vS)$ from the outcome of Theorem~\ref{thm:containers}.

    If~\ref{item:sig-uncoloured} holds, then
    \[
      \Pr\big[G \cap X(\vS) = \emptyset | \pi(\vS) \subseteq G\big] \le (1-p)^{|X(\vS)|} \le \exp(-\eps n^2p),
    \]
    so we may assume that~\ref{item:sig-unbalanced} holds; without loss of generality, $|R(\vS)| \le \beta n^2/2$.  Conditioned on the event that $\pi(\vS) \subseteq G$, the distribution of $e\big(R(\vS) \cap G\big)$ is stochastically dominated by the random variable $|\pi(\vS)| + \Bin\big(|R(\vS)|, p\big)$. Since $|\pi(\vS)| \le t \tau n^2 \le t \sigma n^2 p \le \beta n^2p/4$,
    we have
    \[
      \Pr\left[e\big(R(\vS) \cap G\big) \ge \beta n^2p | \pi(\vS) \subseteq G\right] \le \Pr\big[\Bin(\beta n^2/2, p) \ge (3/4)\beta n^2p\big] \le \exp(-\beta n^2p/32),
    \]
by Lemma~\ref{bintail}.     This proves the assertion of the claim.
  \end{claimproof}
  Using Claim~\ref{clm:colouring-sig}, we may conclude that
  \[
    \begin{split}
      \Pr\big[\text{$G$ has a bad colouring}\big] & \le \sum_{\vS} \Pr\big[\text{$G$ has a bad colouring $\varphi$ with $\sig(\varphi) = \vS$}\big] \\
      & \le \exp(-\gamma n^2p) \cdot \sum_{\vS} \Pr\big[\pi(\vS) \subseteq G\big].
    \end{split}
  \]
  Finally, since there are at most $2^{t|U|}$ sequences $\vS \coloneqq (S_1, \dotsc, S_t)$ satisfying $\pi(\vS) = U$, we have
  \[
    \sum_{\vS} \Pr\big[\pi(\vS) \subseteq G\big] \le \sum_{u \le t\tau v(\cH)} \binom{\binom{n}{2}}{u} \cdot 2^{tu} \cdot p^u \le \sum_{u \le t \sigma n^2p} \left(\frac{2^ten^2p}{u}\right)^u \le n^2 \left(\frac{2^te}{t\sigma}\right)^{t\sigma n^2p} \le e^{\gamma n^2p/2},
  \]
 for $n$ sufficiently large, where the penultimate inequality follows from the fact that, for every $a > 0$, the function $u \mapsto (ea/u)^u$ is increasing when $u \in (0,a]$ and the last inequality follows from the fact that $t\sigma\log (2^te/t\sigma)\leq \gamma/4$ due to our choice of $\sigma$. Hence we have that a.a.s.\ there are no bad colourings of $G$, concluding the proof. 
\end{proof}

\subsection{Unbalanced colourings in the lower range} \label{sec:1 lower}

In this section, we establish the following theorem, proving that $q\left( n; K_3 ,p \right) \le n^{-3}p^{-7/2}$ when $n^{-2/3}\ll p\ll n^{-1/2}$ and hence giving the $1$-statement for the lower range in Theorem~\ref{thm:main}.

\begin{thmtool} \label{thm:1-lower}
  Suppose that $ n^{-2/3}\ll p \ll n^{-1/2}$ and $q\gg n^{-3}p^{-7/2}$ and  let $G_1 \sim G_{n,p}$ and $G_2\sim G_{n,q}$ be independent.  Then, a.a.s.\ no $G_1$-measurable $K_3$-free colouring $\varphi \colon E(G_1) \to \{\red, \blue\}$  can be extended to a $K_3$-free colouring of $G_1\cup G_2$.
\end{thmtool}

This theorem will follow from the following proposition which deals with unbalanced colourings.
 
\begin{prop}
  \label{prop:unbalanced-lower}
  There exist $\beta, \zeta>0$ such that the following holds.  Suppose that $ n^{-2/3}\ll p \ll n^{-1/2}$ and let $G \sim G_{n,p}$.  Then, a.a.s.\ every $K_3$-free $\varphi \colon E(G) \to \{\red, \blue\}$ such that $\left| \varphi ^{-1}(\red)\right| < \beta n^2p$ results in at least $\zeta n^6p^7$ copies of $C_{rbbbb}$.
\end{prop}

Indeed, with Proposition~\ref{prop:unbalanced-lower} and our previous results, Theorem~\ref{thm:1-lower} follows readily.

\begin{proof}[Proof of Theorem~\ref{thm:1-lower}]
  Let $\beta, \zeta>0$ be the constants from the statement of Proposition~\ref{prop:unbalanced-lower}. Further, let $\lambda>0$ be the constant output by Proposition~\ref{prop:balanced-colouring} with input $\beta$ and let $t_1\coloneqq\lambda n^4p^4\geq n^7p^{10}$ and $t_2\coloneqq\zeta n^6p^7\geq n^7p^9$. Now fixing $G_1\sim G_{n,p}$ and $G_2\sim G_{n,q}$, we have that a.a.s.\ the conclusions of Propositions~\ref{prop: many Crrbb} with $t_{\ref{prop: many Crrbb}}=t_1$,~\ref{prop: many Crbbbb} with $t_{\ref{prop: many Crbbbb}}=t_2$,~\ref{prop:balanced-colouring} and~\ref{prop:unbalanced-lower} all hold. We claim that this implies the theorem.  Indeed, consider some  $K_3$-free colouring $\varphi \colon E(G_1) \to \{\red, \blue\}$. Suppose first  that $\left| \varphi ^{-1}(c)\right| \ge \beta n^2p$ for both $c\in \{\red,\blue\}$.  By the assumed conclusion of Proposition~\ref{prop:balanced-colouring}, there are at least $t_1$ copies of $C_{rrbb}$ induced by $\phi$.  Since $q\gg n^{-3}p^{-7/2}\gg t_1^{-1}$, the assumed conclusion of Proposition~\ref{prop: many Crrbb} gives that $\phi$ cannot be extended to $G_2$ whilst avoiding monochromatic triangles.  Likewise, if $|\phi^{-1}(\red)|<\beta n^2p$, then Proposition~\ref{prop:unbalanced-lower} gives that there are at least $t_2$ copies of $C_{rbbbb}$ induced by $\phi$ and Proposition~\ref{prop: many Crbbbb} then gives that we cannot extend $\phi$ to $G_2$ without getting monochromatic triangles, using that $q \gg t_2^{-1/2}$.  Since both colours play symmetric roles, the same conclusion holds under the assumption $|\phi^{-1}(\blue)| < \beta n^2p$.  This covers all colourings and completes the proof.
\end{proof}

It remains to prove Proposition~\ref{prop:unbalanced-lower}. 
Our proof works by taking a union bound over  all possibilities $T$ for the red subgraph. For each $T$, we  use Janson's inequality (Lemma~\ref{janson}) to prove that  it is very unlikely that we avoid creating many $C_{rbbbb}$ when we colour $T$ red and $G \setminus T$ blue. This simple approach almost works -- it turns out that in order to get strong enough error probabilities in the Janson argument, we need to consider only red subgraphs $T$ that  are well behaved, in that they satisfy a maximum degree condition. Before embarking on the proof of Proposition~\ref{prop:unbalanced-lower}, we prove that any red subgraph $T$ that we are interested in contains a large induced subgraph that is well behaved. 

\begin{lem} \label{lem:5-cycle-preprocessing}
  For any $c>0$, there exists a $\beta>0$ such that the following holds for all sufficiently large $n \in \N$ and $p=p(n)\in [0,1]$.  Let $T$ be a graph on $n$ vertices such that $e(T) < \beta n^2p$ and $e(T[U]) \ge 1$ for every $U\subseteq V(G)$ with $|U|\ge \frac{n}{2}$.
  Then, there exists a vertex subset $W\subseteq V(T)$ such that $|W|\ge \frac{n}{2}$ and $S\coloneqq T[W]$ satisfies $\Delta(S)\le\frac{cnp}{\log (n^2p)-\log(e(S))}$.
\end{lem}
\begin{proof}
  Suppose that $T$ satisfies the assumptions of the lemma.  We can assume that $0<c<1/10$ and we fix some $\eps \in (0, c^2)$ and $\beta \in (0, \eps c/8)$.  Consider the following iterative process of peeling off vertices:

  \medskip
  
  \begin{algorithm}[H]
    \SetAlgoNoEnd
    Let $T_0 \coloneqq T$ and $t_0 \coloneqq e(T_0)$.
    
    \For{$i = 1, 2, \dotsc$}{    
      \eIf{$\Delta(T_{i-1}) \le \frac{cnp}{\log(n^2p) - \log(t_{i-1})}$}{
        Terminate with $S \coloneqq T_{i-1}$.
      }{
        Let $v_i$ be an arbitrary vertex of $T_{i-1}$ with degree exceeding $\frac{cnp}{\log(n^2p) - \log(t_{i-1})}$, let $T_i \coloneqq T_{i-1} - v_i$ and $t_i \coloneqq e(T_i)$.
      }
    }
  \end{algorithm}

  \medskip

  We claim that the process terminates after fewer than $n/2$ steps and thus outputs an appropriate $S$.  Suppose for a contradiction that this is not the case and that the process is still running after $n/2$ steps.  We will show that this contradicts our upper bound on $e(T)$.  Firstly, note that 
  \[
    t_{n/4}\ge \frac{n}{4} \cdot \frac{cnp}{\log (n^2p)-\log(t_{n/2})} \ge \frac{n}{4} \cdot \frac{cnp}{\log (n^2p)} \ge \frac{cn^2p}{2^3\log n},
  \]
  using here that $t_{n/2}\ge 1$ due to our assumption on $T$. Now, define $\tau_0\coloneqq 0$ and
  \[
    \tau_i\coloneqq\min\left\{\tau:t_{n/4-\tau}\ge 2^{i-3}\frac{cn^2p}{\log n}\right\}
  \]
  for $i=1, 2,\ldots, k \coloneqq \log_2 (\eps \log n)$.

  \begin{clm}
    $\tau_k\le n/4$ (and hence all $\tau_i$ are well defined).
  \end{clm}
  Note that the claim implies that
  \[e(T)=t_0\ge t_{n/4-\tau_k}\ge 2^{k-3}\frac{cn^2p}{\log n}=\frac{\eps cn^2p}{2^3}>\beta n^2p,\]
  contradicting our upper bound on $e(T)$. It thus remains to prove the claim.

  For this, note that, for each $0\le i \le k-1$ and all $\tau\in \N$ with $\tau\le n/4-\tau_i$, we have that
  \[
    t_{n/4-\tau_i-\tau} - t_{n/4-\tau_i} \ge \tau \cdot \frac{cnp}{\log (n^2p)- \log(t_{n/4-\tau_i}) }
    \ge \tau \cdot \frac{cnp}{\log (2^{3-i}c^{-1}\log n) }.
  \]
  Consequently, for $i \in \{0,\dotsc,k-1\}$,
  \[
    \tau_{i+1}-\tau_i\le 2^{i-2} \cdot \frac{cn^2p}{\log n} \cdot \frac{\log\left(2^{3-i}c^{-1}\log n\right)}{cnp}=  \frac{2^i n \left(\log\big(2^{3}c^{-1}\log n\big)-i \log 2\right)}{4\log n}
  \]
  and so 
  \begin{align*}
    \tau_k&=\tau_0+\sum_{i=0}^{k-1}(\tau_{i+1}-\tau_i) \\
          &\le \frac{n}{4\log n}\left(\sum_{i=0}^{k-1} 2^{i}  \log \big(2^{3}c^{-1}\log n\big) - \log 2 \cdot \sum_{i=0}^{k-1} i2^{i}\right)\\
          &\le \frac{n}{4\log n}\left(2^{k}  \log \big(2^3c^{-1}\log n\big)-\log 2 \cdot (k-2)2^k\right)
    \\ & =  \frac{2^kn}{4 \log n} \log (2^{5-k}c^{-1}\log n)
    \\ &\le \frac{\eps n}{4} \log
         \left(2^5c^{-1}\eps^{-1} \right)
         \leq \frac{n}{4},
  \end{align*}
  as required, where we used our upper bounds on $\eps$ and  $c$ in the final inequality. 
\end{proof}

We now use Lemma~\ref{lem:5-cycle-preprocessing} to establish Proposition~\ref{prop:unbalanced-lower}.

\begin{proof}[Proof of Proposition~\ref{prop:unbalanced-lower}]
  Let $\theta$ be the constant from the statement of Lemma~\ref{lem:randomgraph}~\ref{item:many K3}, let $\zeta = \theta/2^{10}$, let $c = 2^{-17}$ and let $\beta \coloneqq \beta_{\ref{lem:5-cycle-preprocessing}}(c)$ be the constant from the statement of Lemma~\ref{lem:5-cycle-preprocessing}.  Define $\cS$ to be the set of graphs $S$ such that $\theta n^3p^3/8\le e(S)< \beta n^2p$  and  $\Delta(S)\le\frac{cnp}{\log (n^2p)-\log(e(S))}$.     Given an $S\in \cS$, let $B(S)$ be the event  that for $G \sim G_{n,p}$, the graph $G\cup S$ contains fewer than $\zeta n^6p^7$ copies of $C_5$ whose  vertices all lie in $V(S)$ and that have one edge in $S$ and four edges of $G \setminus S$.  The following key claim  bounds the probability of $B(S)$ for all $S\in \cS$.
  
      \begin{clm}
      \label{clm:B(S) bound}
      For any $S\in \cS$, letting $s=e(S)$, we have that 
      \[
        \Pr[B(S)] \le \left(\frac{s}{n^2p}\right)^{2s}.
      \]
    \end{clm}

    Before proving this claim, let us see how it implies the proposition. Firstly, let  $\cF$ be the family of all graphs $T$ on $V(G)$ that satisfy the following:
  \begin{enumerate}[label=(\roman*)]
  \item \label{cond:T1} $e(T) < \beta n^2p$;
  \item \label{cond:T2} for every $U\subseteq V(G)$ with $|U|\ge \frac{n}{2}$, we have that $e(T[U])\geq \theta|U|^3p^3$.
  \end{enumerate}
  Now, if  $G$ satisfies property~\ref{item:many K3} of Lemma~\ref{lem:randomgraph}, then every $K_3$-free colouring $\varphi \colon E(G) \to \{\red, \blue\}$ that colours fewer than $\beta n^2 p$ edges red satisfies $\phi^{-1}(\red) \in \cF$.  Indeed, \ref{item:many K3} gives a collection of at least $\theta|U|^3p^3$ edge-disjoint triangles in each $U\subseteq V(G)$ with $|U|\geq \frac{n}{2}$ and at least one edge in each triangle must be coloured red.  Further, Lemma~\ref{lem:5-cycle-preprocessing} implies that, for every $T\in \cF$, there is some $S=S(T)\in \cS$ such that $S=T[W]$ for some $W \subseteq V(G)$ with $|W| \ge n/2$; indeed, the fact that $T\in \cF$ gives that $e(S)\geq \theta n^3p^3/8$. This implies that, if there is a $K_3$-free $\varphi \colon E(G) \to \{\red, \blue\}$ with $|\varphi^{-1}(\red)| < \beta n^2p$ and fewer than $\zeta n^6p^7$ copies of $C_{rbbbb}$ (and $G$ satisfies property~\ref{item:many K3} of Lemma~\ref{lem:randomgraph}), then there is some $S\in \cS$ such that $B(S)$ occurs and $S\subseteq G$.  Indeed, $T \coloneqq \varphi^{-1}(\red) \in \cF$ and $B(S(T))$ occurs as otherwise we get at least $\zeta n^6p^7$ copies of $C_5$ on $V(S)$ each of which has exactly one edge in $S$ and the other 4 edges in $G\setminus S$ and hence gives a copy of $C_{rbbbb}$ in $G$. Finally, note that, for each $S\in \cS$, the events $B(S)$ and $S\subseteq G$ are independent. Therefore, the probability that $G$ has a colouring  $\varphi \colon E(G) \to \{\red, \blue\}$ with $|\varphi^{-1}(\red)| < \beta n^2p$ and fewer than $\zeta n^6p^7$ copies of $C_{rbbbb}$ is less than 
    \[
      \sum_{S \in \cS} \Pr[B(S) \wedge S \subseteq G] + \Pr[G\notin~\ref{item:many K3}]
      \le \sum_{S \in \cS} \Pr[B(S)]\cdot\Pr[S \subseteq G] + \Pr[G\notin~\ref{item:many K3}].
    \]
    
    By Lemma~\ref{lem:randomgraph}, we have that $\Pr[G\notin\ref{item:many K3}]\ll 1$.  We split the sum over $S \in \cS$ depending on $e(S)=s$.  As there are at most $2^n \binom{\binom{n}{2}}{s}$ graphs $S \in \cS$ with $s$ edges (the factor of $2^n$ bounds the number of choices for $V(S)$), appealing to Claim~\ref{clm:B(S) bound}, we therefore have that 
  \begin{align*}
    \sum_{S \in \cS} \Pr[B(S)]\cdot\Pr[S \subseteq G] & \le \sum_{ s} 2^n \binom{\binom{n}{2}}{s} \cdot \left(\frac{s}{n^2p} \right)^{2s} \cdot p^s 
    \\ 
           & \le \sum_{s} 2^n \cdot \left(\frac{en^2}{2s} \cdot \left(\frac{s}{n^2p}\right)^2 \cdot p\right)^s  \\
    & \le \sum_{s} 2^n \cdot \left(\frac{es}{n^2p}\right)^s  \ll 1,
  \end{align*}
  where the sum goes over all $s \in (\theta n^3p^3/8, \beta n^2p)$ and, in the last inequality, we used
  \[
    s\log\left(\frac{n^2p}{s}\right)\ge s\ge \theta n^3p^3/8 \gg n.
  \]
  Therefore, it remains only to establish Claim~\ref{clm:B(S) bound}.
  
  \begin{claimproof}[Proof of Claim~\ref{clm:B(S) bound}]
    Fix some $S\in \cS$ and let $s \coloneqq e(S)$ and $ W \coloneqq V(S)$.  We will appeal to Janson's inequality (Lemma~\ref{janson}) to obtain the required upper bound on $\Pr[B(S)]$.  Let $\Gamma \coloneqq E(K_n[W]) \setminus S$ and let $\cC$ be the set of all $5$-cycles in $K_n[W]$ comprising of one edge of $S$ and four edges of $\Gamma$.  For each such $C \in \cC$,  let $I_C$ be the indicator random variable for the event that $C \cap \Gamma \subseteq \Gamma_p$ and note that $\mathbb{E}[I_C] = p^4$.  For two cycles $C,C'\in \cC$, write $C\sim C'$ if $C\cap C'\cap \Gamma \neq \emptyset$.  Then, following the notation of Lemma~\ref{janson}, we define 
    \[
      X\coloneqq \sum _{C\in\cC} I_C,\qquad
      \mu \coloneqq \mathbb{E}[X]\qquad \text{and} \qquad
      \Delta \coloneqq \sum _{C\sim C'} \mathbb{E}[I_CI_{C'}],
    \]
    where the sum in the definition of $\Delta$ ranges over all pairs $(C,C')\in \cC\times \cC$ such that $C\sim C'$. Now $B(S)$ is precisely the event that $X\le \zeta n^6p^7$ and we can use Lemma~\ref{janson} to upper bound the probability of this event occurring. 
        
    We begin by estimating $\mu=\mathbb{E}[X]$. We first observe that, for each $u_0u_4\in E(S)$, there are at least $(n/4)^3=n^3/2^6$ choices of $u_1,u_2,u_3\in W$ such that $u_iu_{i+1} \in \Gamma$ for $i=0,1,2,3$.  Indeed, this follows from the fact that $|W| \ge n/2$ and $\Delta(S) \le cnp< n/8$ and so $u_1, u_2, u_3$ can be chosen greedily, avoiding edges of~$S$, with at least $n/4$ choices at each step.  Consequently,
    \begin{equation}
      \label{eq:B(S)-mu-lower}
      \mu=\mathbb{E}[X]\ge \frac{sn^3p^4}{2^6}\geq \frac{\theta n^6p^7}{2^9} \ge 2\zeta n^6p^7,
    \end{equation}
    using that $s\geq \theta n^3p^3/8$, due to the fact that $S\in \cS$.
    
    In order to estimate $\Delta$, we fix some arbitrary $C'\in \cC$ and estimate the number of $C\in\cC$ (whose vertices we will label $u_0,\dotsc, u_4$ as above)  that intersect $C'$. We split the analysis into cases.
    
    \begin{enumerate}
    \item Firstly assume that $|C\cap C'\cap \Gamma|=1$.
      There are at most
      \begin{equation} \label{eq:upper-Delta-1}
        4\cdot(4\cdot\Delta(S)\cdot n^2+4\cdot s\cdot n)\le 32\Delta(S)\cdot n^2
      \end{equation}
      choices of $C$ that intersect $C'$ in one edge (outside of $S$), using that $s\le \Delta(S)\cdot n$ in the inequality.  The factor $4$ comes from choosing an edge of $C'\cap \Gamma$, say $e$.  The first summand then comes from considering the case where $e=u_0u_1$ (or analogously $e=u_3u_4$, resulting in a factor of $2$).  Given that $e=u_0u_1$ and choice of labelling of the vertices (another factor of $2$), there are at most $\Delta(S)$ choices for $u_4$ and at most $n$ further choices for each of $u_2$ and $u_3$.  The second summand stems from the case where $e=u_1u_2$ (or analogously $e=u_2u_3$), where after labelling $e$  there are at most $s$ choices for $u_0u_4$ and at most $n$ further choices for $u_3$.
      
    \item Next assume  $|C\cap C'\cap \Gamma|=2$.
      There are at most
      \begin{equation} \label{eq:upper-Delta-2}
        6\cdot(2 \cdot \Delta(S)\cdot n+ 4 \cdot n+ 2s+8\cdot\Delta(S))\le 96\Delta(S)\cdot n
      \end{equation}
      choices of $C$ that intersect $C'$ in two edges (outside of $S$).  Indeed, the factor $6$ bounds the number of choices of two edges of $C' \cap \Gamma$, say $e_1$ and $e_2$. The first summand then treats the case where $\{e_1,e_2\}=\{u_0u_1,u_1u_2\}$ (equivalently, the case where $\{e_1,e_2\}=\{u_2u_3,u_3u_4\}$). We then have two options for choosing how to label the endpoints of the path $e_1e_2$ as $u_0$ and $u_2$, then at most $\Delta(S)$ choices for $u_4$, and $n$ choices for $u_3$.  In the second summand, we consider the case where $\{e_1,e_2\}=\{u_0u_1,u_3u_4\}$, which means that there are at most four choices for the edge of $C \cap S$ and at most $n$ further choices for $u_2$.  The third summand treats the case where $\{e_1,e_2\}=\{u_1u_2,u_2u_3\}$ and a choice of the edge in $S$ and a labelling of its vertices determines $C$.  Finally, in the fourth summand, we consider the case where $\{e_1,e_2\}=\{u_0u_1,u_2u_3\}$ (or $\{e_1,e_2\}=\{u_1u_2,u_3u_4\}$) and a choice of the edge $u_0u_4\in S$ adjacent to $u_0$ determines $C$.
      
    \item Next, consider the case where  $|C\cap C'\cap \Gamma|=3$.
      There are at most
      \begin{equation} \label{eq:upper-Delta-3}
        4 \cdot (4\cdot \Delta(S)+8)\le 48\Delta(S)
      \end{equation}
      choices of $C$ that intersect $C'$ in three edges (outside of $S$).  Indeed, there are at most $4$ choices for the edge $f \in C'\cap \Gamma$ which is \emph{not} on $C$. If $f = u_0u_1$ (or $f=u_3u_4$), all vertices of $C$ apart from $u_0$ are fixed and so a choice of neighbour of $u_4$ in $S$ defines $C$.  If $f=u_1u_2$ (or $f=u_2u_3$), then after labelling, $C$ is already completely determined, leading to the upper bound in the second summand.
      
    \item Finally, if $|C\cap C'\cap \Gamma|=4$, then clearly there is just one choice for $C$.
    \end{enumerate}
    
    We can now put together the bounds from above to conclude that 
    \[
      \Delta \le \mu \cdot \left(32\Delta(S)n^2p^3 + 96 \Delta(S) n p^2+48\Delta(S) p+1\right) \le \mu \cdot 40\Delta(S) n^2p^3.
    \]
    Therefore, we have, using~\eqref{eq:B(S)-mu-lower} and Lemma~\ref{janson}, that 
    \begin{align*}
      \Pr[B(S)] & =\Pr[X\le \zeta n^6p^7] \le \Pr[X\le \mu/2] 
                  \le \exp\left(-\frac{\mu^2}{8\Delta}\right) \\
                & \le \exp\left(-\frac{\mu}{2^{10} \Delta(S) n^2p^3}\right)
                  \le \exp\left(-\frac{snp}{2^{16}\Delta(S)}\right).
    \end{align*}
    Finally, since $\Delta(S)\le\frac{cnp}{\log (n^2p)-\log(s)}$, which follows from the fact that $S\in \cS$, and $c = 2^{-17}$, we have
    \[
      \Pr[B(S)] \le \exp\left(- 2s \cdot \left( \log(n^2p) - \log s\right)\right) = \left(\frac{s}{n^2p}\right)^{2s}.
    \]
    as claimed.
  \end{claimproof}
  The proof of Proposition~\ref{prop:unbalanced-lower} is now complete.
\end{proof}

\subsection{Unbalanced colourings in the upper range} \label{sec:1 upper}

In this section, we improve on Theorem~\ref{thm:1-lower} when $n^{-3/5}\ll p\ll n^{-1/2}$ and show that in this range we have that $q\left( n; K_3 ,p \right) \le n^{-6}p^{-8}$. This gives the $1$-statement for the upper range in Theorem~\ref{thm:main}.

\begin{thmtool} \label{thm:1-upper}
  Suppose that $ n^{-3/5}\ll p \ll n^{-1/2}$ and $q\gg n^{-6}p^{-8}$ and  let $G_1 \sim G_{n,p}$ and $G_2\sim G_{n,q}$ be independent.  Then, a.a.s.\ no $G_1$-measurable $K_3$-free colouring $\varphi \colon E(G_1) \to \{\red, \blue\}$ can be extended to a $K_3$-free colouring of $G_1\cup G_2$.
\end{thmtool}

As in the previous section, we first reduce Theorem~\ref{thm:1-upper} to the following proposition.
 
\begin{prop}
  \label{prop:unbalanced-upper}
  There exist $\beta, \zeta>0$ such that the following holds.  Suppose that $ n^{-3/5}\ll p \ll n^{-1/2}$ and let $G \sim G_{n,p}$.  Then, a.a.s.\ every $K_3$-free $\varphi \colon E(G) \to \{\red, \blue\}$ such that $\left| \varphi ^{-1}(\red)\right| < \beta n^2p$ results in at least $\zeta n^6p^8$ copies of $C_{rrbb}$.
\end{prop}

With Proposition~\ref{prop:unbalanced-upper} and Proposition~\ref{prop:balanced-colouring}, the proof of Theorem~\ref{thm:1-upper} follows almost immediately.

\begin{proof}[Proof of Theorem~\ref{thm:1-upper}]
  Let $\beta, \zeta > 0$ be the constants from the statement of Proposition~\ref{prop:unbalanced-upper}.  Further, let $\lambda>0$ be the constant output by Proposition~\ref{prop:balanced-colouring} with input $\beta$, let $t\coloneqq \zeta n^6p^8$ and note that $t \leq \lambda n^4p^4$.  Now, with $G_1\sim G_{n,p}$ and $G_2\sim G_{n,q}$, we have that a.a.s.\ the conclusions of Propositions~\ref{prop: many Crrbb}, \ref{prop:balanced-colouring} and \ref{prop:unbalanced-lower} all hold.  In particular, a.a.s.\ any  $K_3$-free colouring $\varphi \colon E(G_1) \to \{\red, \blue\}$  gives rise to at least $t$ copies of $C_{rrbb}$.  Indeed, this follows from Proposition~\ref{prop:balanced-colouring}  if $\left| \varphi ^{-1}(c)\right| \ge \beta n^2p$ for both $c\in \{\red,\blue\}$,  or from Proposition~\ref{prop:unbalanced-upper} if $|\phi^{-1}(c)|<\beta n^2p$ for some $c\in \{\red,\blue\}$. The conclusion of the theorem then follows from Proposition~\ref{prop: many Crrbb} as $q\gg t^{-1}$. 
\end{proof}

It remains to prove Proposition~\ref{prop:unbalanced-upper}.  Before embarking on this, we make some definitions and prove several auxiliary lemmas.   As in the proof of Proposition~\ref{prop:unbalanced-lower}, presented in the previous section, we will condition on the \emph{red} subgraph $T$ of $G_{n,p}$. The following definition captures important properties of $T$ that hold a.a.s.\ in $G_{n,p}$ and that we will thus be able to assume hold in our proof.  Throughout this section, we write $\theta$ for the constant from Lemma~\ref{lem:randomgraph}.

\begin{defi} \label{def:upper-red-conds}
  For $\beta>0$ and $n^{-3/5}\ll p\ll n^{-1/2}$,  let $\cF=\cF(\beta;p)$ be the set of subgraphs $T\subseteq K_n$ such that 
\begin{enumerate}[label=(\roman*)]
  \item \label{cond:Ta} $\theta n^3p^3\leq e(T) < \beta n^2p$;
  \item \label{cond:Tb} $T$ satisfies conditions~\ref{item:bound max deg} (upper bounding the maximum degree),~\ref{item:few F} (upper bounding the number of small subgraphs $F$),~\ref{item:K210-count} (upper bounding the number of $K_{2,10}$) and~\ref{item:eAB-concentration} (upper bounding the number of edges between vertex sets) of Lemma~\ref{lem:randomgraph}.
  \end{enumerate}
\end{defi}

In proving Proposition~\ref{prop:unbalanced-upper}, we will show that a.a.s.\ any $K_3$-free colouring $\varphi \colon E(G) \to \{\red, \blue\}$  such that $\left| \varphi ^{-1}(\red)\right| < \beta n^2p$ will have $ \varphi ^{-1}(\red)\in \cF(\beta;p)$.  
Again, similarly to Proposition~\ref{prop:unbalanced-lower}, we will not be able to take a union bound over all possible red subgraphs $T\in \cF$ and will instead consider only carefully chosen subgraphs of such $T$ that we can enumerate more efficiently.  Given that we aim to find many $C_{rrbb}$, the following definitions will be useful.   

\begin{defi} \label{def:S params}
Let $S \subseteq K_n$ be a graph on $n$  vertices. We define the following parameters: 
\begin{itemize}
    \item  $X_2(S)$ denotes the number of copies of $K_{1,2}$ in $S$;
    \item $\Pi(S)$ denotes the  edges in $K_n$ that complete a triangle with a copy of $K_{1,2}$ that lies in $S$;
    \item $\cX_S$ denotes the family of all copies of $K_{1,2}$ in $K_n$ that form a 4-cycle  with some copy of $K_{1,2}$ in $S$.
\end{itemize}
  
\end{defi} Our next simple lemma gives a lower bound on $X_2(S)$ in terms of the number of edges of a subgraph $S\subseteq K_n$, given that $S$ is not too small.
  
    \begin{lem}
  \label{lem:X2-T'-lower}
  If $S$ is a graph on $n$ vertices with at least $2n$ edges, then $X_2(S) \ge \frac{3e(S)^2}{2n}.$
\end{lem}

\begin{proof}
  By convexity, we have that 
  \[
    X_2(S) = \sum _{v\in V} \binom{d_{S}(v)}{2} \ge n \cdot \binom{\sum_{v \in V} d_{S}(v)/n}{2} = n \cdot \binom{2e(S)/n}{2} \ge \frac{3e(S)^2}{2n},
\]
where the last inequality holds due to our assumption that $e(S) \ge 2n$.
\end{proof}

 Next, for certain subgraphs $S\subseteq K_n$, we show that $|\Pi(S)|$ can be lower bounded by $X_2(S)$. 
 
 \begin{lem} \label{lem:big pi}
   Suppose that $n^{-3/5}\ll p\ll n^{-1/2}$ and $S\subseteq K_n$ is an $n$-vertex graph such that $s\coloneqq e(S)\geq \theta n^3p^3/2$ and $S$ satisfies~\ref{item:K210-count} (upper bounding the number of $K_{2,10}$)  of Lemma~\ref{lem:randomgraph}. Then \[|\Pi(S)|\geq \frac{X_2(S)}{12}\geq \frac{s^2}{8n}.\]
 \end{lem}
 
 \begin{proof}
   For each pair of vertices $\rho\in \binom{[n]}{2}=E(K_n)$, let $d_\rho$ be the number of copies of $K_{1,2}$ in $S$ that form a triangle with $\rho$ and call $\rho$  \emph{heavy} if $d_\rho\geq 10$. Then $\Pi=\Pi(S)\subseteq E(K_n)$ are the pairs $\rho\in E(K_n)$ such that $d_\rho\geq 1$ and let $\Pi_H\subseteq \Pi$ be the heavy pairs. Then we have that 
   \begin{equation} \label{eq:heavy uppper sum}
     \sum_{\rho \in \Pi_H} 
     \frac{d_\rho}{10} \leq \sum_{\rho\in \Pi_H}\left(\frac{d_\rho}{10}\right)^{10} \leq \sum_{\rho\in \Pi_H}\binom{d_\rho}{10}=N_{K_{2,10}}(S) \leq n^{11}p^{18}\ll n^5p^6, \end{equation}
   using that property~\ref{item:K210-count} of Lemma~\ref{lem:randomgraph} holds in $S$ in the penultimate inequality and the fact that $p\ll n^{-1/2}$ in the final inequality.  On the other hand,
   \begin{equation} \label{eq:X2 lower sum}
      \sum_{\rho \in \Pi}d_\rho= X_{2}(S)\geq \frac{3s^2}{2n}\geq \frac{\theta^2 n^5p^6}{4},
   \end{equation}
   by appealing to Lemma~\ref{lem:X2-T'-lower} and our lower bound on $s=e(S)$. Combining~\eqref{eq:heavy uppper sum} and~\eqref{eq:X2 lower sum} then gives that   $\sum_{\rho \in \Pi\setminus \Pi_H}d_\rho\geq 5X_2(S)/6$ and so 
   \[|\Pi(S)|\geq |\Pi\setminus \Pi_H|\geq \frac{1}{10}\sum_{\rho\in \Pi\setminus \Pi_H}d_\rho\geq \frac{X_2(S)}{12}\geq \frac{s^2}{8n}, \]
   using Lemma~\ref{lem:X2-T'-lower}, which completes the proof.
 \end{proof}

 Our next lemma identifies, for each $T\in \cF$, some subgraph $S=S(T)\subseteq T$ for which the collection $\cX_S$ is large and well-spread in $K_n$. Following  the notation of Lemma~\ref{janson}, for a subgraph $S\subseteq K_n$ and $p=p(n)$, we let $\mu(\cX_S)$ be the expected number of copies of $K_{1,2}$ in $\cX_S$ that appear in $G_{n,p}$.
 Since every edge in $\Pi(S)$ gives rise to either $n-3$ or $n-2$ copies of $K_{1,2}$ in $K_n$ that close a $4$-cycle with some $K_{1,2}$ in $S$, we have
 \begin{equation} \label{eq:mu S}
   \mu(\cX_S) = |\cX_S|p^2 \in \left[(n-3)|\Pi(S)| p^2 , (n-2)|\Pi(S)| p^2\right].
 \end{equation}
 We also let
 \[
   \Delta(\cX_S)\coloneqq\sum_{K,K'}p^{e(K\cup K')},
 \]
 where the sum goes over all pairs of copies $K,K'\in \cX_S$ such that $K \cap K'\neq \emptyset$.

\begin{lem}
  \label{lemma:DT_is_big}
Suppose $0<\beta <2^{-100}$ and $n^{-3/5}\ll p\ll n^{-1/2}$ and let $T\in \cF=\cF(\beta;p)$ with $t\coloneqq e(T)$.  Then there exists a subgraph $S=S(T)\subseteq T$ with  $e(S)\geq t/2$ and such that either 
 \begin{enumerate}[label=(\alph*)]
     \item \label{type:a} $\frac{\mu(\cX_S)^2}{\Delta(\cX_S)}\geq 10t\log \left( \frac{2n^2p}{t} \right) $; or,
     \item \label{type:b} $
    10t\log \left( \frac{2n^2p}{t} \right)>\frac{\mu(\cX_S)^2}{\Delta(\cX_S)}\geq  \frac{\mu(\cX_S)}{3}$.
 \end{enumerate} 
\end{lem}

\begin{proof}
  Fix some $T \in \cF$ and denote $t \coloneqq e(T)$. Now for any subgraph $S\subseteq T$, we 
  define
  \begin{align*}
      \Delta_1(\cX_S)&\coloneqq|\{K,K'\in\cX_S: K\cup K' \text{ is a path with } 3 \text{ edges }\}|\cdot p^3 \text{ and } \\
      \Delta_2(\cX_S)&\coloneqq|\{K,K'\in\cX_S: K\cup K' \text{ is a copy of } K_{1,3}\}|\cdot p^3,
  \end{align*}
  and we note that, for every $S$, we have $\Delta(\cX_S)=\Delta_1(\cX_S)+\Delta_2(\cX_S)+\mu(\cX_S)$. Indeed, the sum in the definition of $\Delta(\cX_S)$ ranges over all pairs $K,K'\in \cX_S$ that intersect in at least one edge. If they intersect in two edges, then $K=K'$ and the contribution to $\Delta(\cX_S)$ is counted by $\mu(\cX_S)$ and if they intersect in precisely one edge, then their union is either a path, in which case they are counted by $\Delta_1(\cX_S)$, or a star in which case they are counted by $\Delta_2(\cX_S)$. The following claim is the key step in proving the lemma. 
  
  \begin{clm}\label{claim1statupper}
  There is an $S \subseteq T$ with at least $t/2$ edges such that
  \[
    \frac{\mu(\cX_S)^2}{\Delta_2(\cX_S)} \ge 30t\log \left( \frac{2n^2p}{t} \right).
  \]
\end{clm}
  
  With Claim~\ref{claim1statupper}, the lemma follows quickly. Indeed, fix $S\subseteq T$ as output by the claim and note that 
  \begin{equation} \label{eq:mu split min}
\frac{\mu(\cX_S)^2}{\Delta(\cX_S)}=\frac{\mu(\cX_S)^2}{\Delta_1(\cX_S)+\Delta_2(\cX_S)+\mu(\cX_S)} \geq \frac{1}{3} \min\left\{\frac{\mu(\cX_{S})^2}{\Delta_1(\cX_S)}, \frac{\mu(\cX_{S})^2}{\Delta_2(\cX_S)}, \mu(\cX_S)\right\}.  \end{equation}
Firstly, suppose that the minimum is achieved by the last term.   In this case we have that $\frac{\mu(\cX_S)^2}{\Delta(\cX_S)}\geq \frac{\mu(\cX_S)}{3} $ and the conclusion of the lemma is satisfied, with $S$ satisfying~\ref{type:a} if $\frac{\mu(\cX_S)}{3} \geq 10t\log \left( \frac{2n^2p}{t} \right)$ and~\ref{type:b} otherwise. Likewise, if the minimum is achieved by the middle term, then by Claim~\ref{claim1statupper} we have that $\frac{\mu(\cX_S)^2}{\Delta(\cX_S)}$ satisfies~\ref{type:a}. It remains to consider the case where the minimum is achieved by the first term. For this, note that we have $\Delta_1(\cX_S)\leq 4 |\Pi(S)|^2p^3$. Indeed, it $v_1v_2v_3v_4$ is a path of length three in $S$ labeled so that $K$ is the path $v_1v_2v_3$ and $K'$ is the path $v_2v_3v_4$, see Figure~\ref{fig:path}, then $v_1v_3, v_2v_4\in \Pi(S)$ and thus we can count the number of pairs $K,K'$ whose union is a path with three edges by the number of pairs in $\Pi(S)$ with a labelling of the endpoints of each pair. Using~\eqref{eq:mu S}, we therefore have that 
\[\frac{\mu(\cX_{S})^2}{\Delta_1(\cX_S)}\geq \frac{|\Pi(S)|^2 n^2p^4 }{5 |\Pi(S)|^2p^3}\geq \frac{n^2p}{5}= t \log \left( \frac{2n^2p}{t} \right) \cdot \frac{n^2p}{5t}\cdot \log ^{-1} \left( \frac{2n^2p}{t}\right) \ge 30t \log \left( \frac{2n^2p}{t} \right),\]
using here that $t<\beta n^2p$ (property~\ref{cond:Ta} of Definition~\ref{def:upper-red-conds}) and the fact that $\beta<2^{-100}$. Therefore, we also satisfy part~\ref{type:a} of the lemma when the minimum in~\eqref{eq:mu split min} is achieved by the first term.

 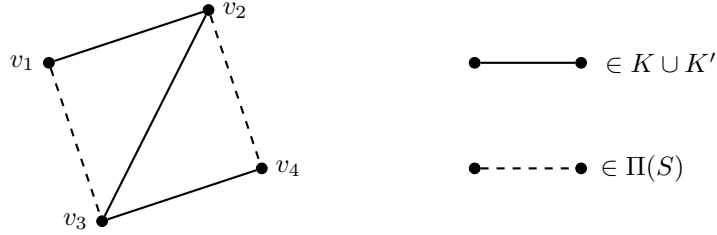
\begin{figure}[h]
\centering
\begin{tikzpicture}[vertex/.style={draw,circle,color=black,fill=black,inner sep=1,minimum width=4pt}]
    
	\pgfmathsetmacro{\scale}{0.7}    
    
    \node[vertex] (v1) at (-2*\scale, 1*\scale) {};
   
    \node[vertex] (v3) at (-1*\scale, -2*\scale) {};
    
    \node[vertex] (v4) at (2*\scale, -1*\scale) {};

    \node[vertex] (v2) at (1*\scale, 2*\scale) {};
    
        \node[vertex] (u1) at (6*\scale, 1*\scale) {};
                \node[vertex] (u2) at (8*\scale, 1*\scale) {};
                    
        \node[vertex] (u3) at (6*\scale, -1*\scale) {};
                \node[vertex] (u4) at (8*\scale, -1*\scale) {};

	\draw[thick] (v1) to (v2);
	\draw[thick] (v2) to (v3);
	\draw[thick] (v4) to (v3);
		\draw[thick] (u1) to (u2);
		\draw[thick,dashed] (v1) to (v3);
		\draw[thick,dashed] (v2) to (v4);
			\draw[thick,dashed] (u3) to (u4);

	\node at (9.5*\scale, 1*\scale) {$\in K\cup K'$};
		\node at (9.15*\scale, -1*\scale) {$\in \Pi(S)$};
	\node at (-2.5*\scale, 1*\scale) {$v_1$};
    \node at (1.5*\scale, 2*\scale) {$v_2$};
     \node at (-1.5*\scale, -2*\scale) {$v_3$};
     \node at (2.5*\scale, -1*\scale) {$v_4$};
    
\end{tikzpicture}

\caption{Two copies of $K_{1,2}$ in $\cX_S$ whose union is a path.}\label{fig:path}
\end{figure}

It remains to prove Claim~\ref{claim1statupper}, which we do now.

\begin{claimproof}[Proof of Claim~\ref{claim1statupper}]

In order to derive the lower bound in Claim~\ref{claim1statupper}, we need to upper bound $\Delta_2(\cX_S)$ and hence we need an upper bound on the count of pairs $K,K'\in \cX_S$ such that $K\cup K'$ forms a copy of $K_{1,3}$. Given such a pair $K$ and $K'$, denote the vertex of degree three in $K\cup K'$ by $u$ and the remaining vertices by $w_1,w_2,w_3$ so that $K$ lies on vertices $u,w_1,w_2$, $K'$ lies on vertices $u,w_2,w_3$ and thus $w_1w_2, w_2w_3 \in \Pi(S)$, see Figure~\ref{fig:K13}.  Hence, we have that 
\begin{equation} \label{eq:Delta 2 Xs}
    \Delta_2(\cX_S)\leq X_2(\Pi(S))np^3,
\end{equation}
where $X_2(\Pi(S))$ is the number of copies of $K_{1,2}$ in  $\Pi(S)$ when considered as a graph on $n$ vertices. Indeed, the number of possible pairs $K,K'$ with $K\cup K'$ a copy of $K_{1,3}$ can be bounded by choosing a copy of $K_{1,2}$ in $\Pi(S)$  and a choice of vertex $u$ (at most $n$ choices). We proceed by splitting our analysis into cases,  depending on whether or not $T$ contains a large subgraph with maximum degree at most $d \coloneqq\sqrt{\frac{tp}{2000\log \left( 2n^2p/t \right)}}$.
  
  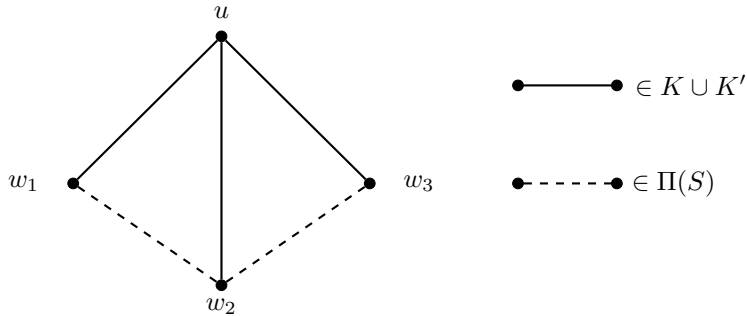
\begin{figure}[h]
\centering

\begin{tikzpicture}[vertex/.style={draw,circle,color=black,fill=black,inner sep=1,minimum width=4pt}]
    
	\pgfmathsetmacro{\scale}{0.65}    
    
    \node[vertex] (w2) at (0, -2) {};

    \node[vertex] (w1) at (-3*\scale, -1*\scale) {};
    \node[vertex] (w3) at (3*\scale, -1*\scale) {};
    
    \node[vertex] (u) at (0, 2*\scale) {};

        \node[vertex] (x1) at (6*\scale, 1*\scale) {};
                \node[vertex] (x2) at (8*\scale, 1*\scale) {};
                    
        \node[vertex] (x3) at (6*\scale, -1*\scale) {};
                \node[vertex] (x4) at (8*\scale, -1*\scale) {};

	\draw[thick] (u) to (w1);
	\draw[thick] (u) to (w2);
	\draw[thick] (u) to (w3);
		\draw[thick] (x1) to (x2);
		\draw[thick,dashed] (w1) to (w2);
		\draw[thick,dashed] (w2) to (w3);
			\draw[thick,dashed] (x3) to (x4);

	\node at (9.5*\scale, 1*\scale) {$\in K\cup K'$};
		\node at (9.15*\scale, -1*\scale) {$\in \Pi(S)$};
	
	\node at (0, - 3.5*\scale) {$w_2$ };
	
		\node at (-4*\scale, -1*\scale) {$w_1$ };
		
		\node at (4*\scale, -1*\scale) {$w_3$ };
		
		\node at (0*\scale, 2.5*\scale) {$u$ };

\end{tikzpicture}

\caption{Two copies of $K_{1,2}$ in $\cX_S$ whose union is a copy of $K_{1,3}$.}\label{fig:K13}
\end{figure}

  \medskip
  \noindent
  \textit{Case 1. $T$ contains a subgraph with at least $t/2$ edges and maximum degree at most $d$.}

  We let $S$ be one such subgraph.  Note that $X_{2}(\Pi(S))\le |\Pi(S)|\Delta(\Pi(S))$ where $\Delta(\Pi(S))$ is the maximum degree of a vertex in $\Pi(S)$ when considered as a graph on $n$ vertices. Since $\Delta(\Pi(S))\le \Delta(S)^2\leq d^2$, appealing to~\eqref{eq:mu S} and~\eqref{eq:Delta 2 Xs}, we have that 
  \begin{align*}
    \frac{\mu(\cX_S)^2}{\Delta_2(\cX_S)} & \ge \frac{|\Pi(S)|^2 n^2p^4}{2|\Pi(S)|d^2np^3} \ge \frac{1000|\Pi(S)|n \log(2n^2p/t)}{t} \geq 30 t \log\left(\frac{2n^2p}{t}\right) ,
  \end{align*}
  as claimed, where we used Lemma~\ref{lem:big pi}, noting that the conditions are satisfied as $S\subseteq T\in \cF$ and $e(S)\geq t/2$.

  \medskip
  \noindent
  \textit{Case 2. Every subgraph of $T$ with maximum degree at most $d$ has fewer than $t/2$ edges.}
 In this case, we define $S = T$ as the subgraph with the desired properties.  First, we claim that
  \begin{equation}
    \label{eq:X2T-strong-LB}
    |\Pi(T)| \ge \frac{td}{24}.
  \end{equation}
  Indeed, let $H\subseteq T$ be a maximal subgraph with respect to inclusion such that $\Delta (H) \le d$. By our assumption, $e(H) < t/2$. By the definition of $H$, for any $e \in E(T)\setminus E(H)$, one of its endpoints has degree at least $d+1$ when added to $H$ and hence $e$ is contained in at least $d$ copies of $K_{1,2}$ with edges in $H$. Summing over all edges in $E(T)\setminus E(H)$ gives that $X_2(T)\geq td/2$ and~\eqref{eq:X2T-strong-LB} follows from Lemma~\ref{lem:big pi}.    We will also show  that
  \begin{equation}
    \label{eq:X4T-upper}
   X_2(\Pi(T)) \le 20 |\Pi(T)| tp.
  \end{equation}
We claim that this suffices to prove Claim~\ref{claim1statupper}.   Indeed, appealing to~\eqref{eq:mu S},~\eqref{eq:Delta 2 Xs} and~\eqref{eq:X2T-strong-LB}, we get that 
  \begin{align*}
    \frac{\mu(\cX_T)^2}{\Delta_2(\cX_T)} &{\ge} \frac{|\Pi(T)|^2 n^2p^4 }{40|\Pi(T)|  tp \cdot np^3}                {\ge} \frac{|\Pi(T)| n}{40t}
        {\ge} \frac{dn}{1000}
                                  = \frac{1}{1000}\cdot \sqrt{\frac{tp}{2000\log \left( 2n^2p/t\right)}} \cdot n \\
    & \ge 30t \log \left( \frac{2n^2p}{t} \right) \cdot \sqrt{\frac{n^2p}{2^{50}t}\cdot \log ^{-3} \left( \frac{2n^2p}{t}\right)} \ge 30t \log \left( \frac{2n^2p}{t} \right) ,
  \end{align*}
 as desired, using that $t<\beta n^2p$ and the fact that $\beta \le 2^{-100}$ in the last inequality here.

  \begin{figure}[h]
\centering

\begin{tikzpicture}[vertex/.style={draw,circle,color=black,fill=black,inner sep=1,minimum width=4pt}]
    
	\pgfmathsetmacro{\scale}{0.65}    
    
    \node[vertex] (w2) at (0, -2) {};

    \node[vertex] (w1) at (-3*\scale, -1*\scale) {};
    \node[vertex] (w3) at (3*\scale, -1*\scale) {};

    \node[vertex] (z) at (3*\scale, -4*\scale) {};

        \node[vertex] (x1) at (6*\scale, -3*\scale) {};
                \node[vertex] (x2) at (8*\scale, -3*\scale) {};
                    
        \node[vertex] (x3) at (6*\scale, -1*\scale) {};
                \node[vertex] (x4) at (8*\scale, -1*\scale) {};

\draw[thick, red] (z) to (w3);
\draw[thick, red] (z) to (w2);
		\draw[thick, red] (x1) to (x2);
		\draw[thick,dashed] (w1) to (w2);
		\draw[thick,dashed] (w2) to (w3);
			\draw[thick,dashed] (x3) to (x4);

	\node at (8.75*\scale, -3*\scale) {$\in T$};
		\node at (9.15*\scale, -1*\scale) {$\in \Pi(T)$};
	
	\node at (0, - 2.5*\scale) {$w$ };
		\node at (4*\scale, - 4*\scale) {$z$ };
	
		\node at (-4*\scale, -1*\scale) {$w_1$ };
		
		\node at (4*\scale, -1*\scale) {$w_2$ };

\end{tikzpicture}

\caption{A copy of  $K_{1,2}$ in $\Pi(T)$.}\label{fig:Pi}
\end{figure}
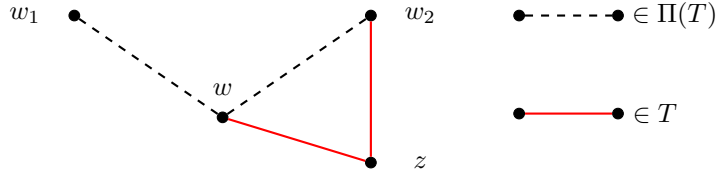

To show that~\eqref{eq:X4T-upper} holds, note first that $X_2(\Pi(T))$ is at most the number of paths $w_1wzw_2$ in $K_n$ such that $w_1w \in \Pi(T)$ and $wz, zw_2 \in T$, see Figure~\ref{fig:Pi}.  Consequently, denoting by $d_{\Pi}(v)$ the number of neighbours of $v$ in $\Pi(T)$, we have 
\begin{equation} \label{eq:X2upper1}
  X_2(\Pi(T))\leq \sum_{e=wz\in T} d_\Pi(w) d_T(z)=\sum_{w\in [n]}\sum_{z\in [n]} d_\Pi(w) d_T(z) \mathds{1}[wz\in T].      
\end{equation}

  We will split this sum further by grouping together vertices depending on their degrees. To this end, for $0 \le \alpha \le \log_2(np)$ and $0\le \beta \le 2\log _2(np)+1$ let
  \[
    A_\alpha \coloneqq \left\{ z\in V(T) : \frac{d_T(z)}{2np} \in \left[ 2^{-\alpha-1},2^{-\alpha} \right] \right\} , \qquad B_\beta \coloneqq \left\{ w\in V(T) : \frac{d_{\Pi}(w)}{4n^2p^2} \in \left[ 2^{-\beta-1},2^{-\beta} \right] \right\},
  \]
  and note that the sets $A_\alpha$ partition the vertices with non-zero degree in $T$ and similarly the sets $B_\beta$ partition the vertices with non-zero degree in $\Pi$, using here that $d_\Pi(w)\leq \Delta(\Pi(S))\leq \Delta(T)^2\leq 4n^2p^2$ for all $w\in V(T)$ due to the fact that $T\in \cF(\beta,p)$ and so $\Delta(T)\leq 2np$ (see Definition~\ref{def:upper-red-conds}). Hence, returning to the upper bound~\eqref{eq:X2upper1}, we have that 
  \begin{equation} \label{eq:X2upper2}
    X_2(\Pi(T)) \le \sum_{\alpha=0}^{\log _2(np)} \sum_{\beta=0}^{2\log _2(np)+1} e_T(A_\alpha,B_\beta) \cdot 8n^3p^3 2^{-\alpha-\beta}.
  \end{equation}
We now turn to bounding  $e_T(A_\alpha,B_\beta)$ for each $0 \le \alpha \le \log_2(np)$ and $0\le \beta \le 2\log _2(np)+1$. For each such $\alpha$ and $\beta$, set
  \[
    a_\alpha \coloneqq \frac{2^{\alpha+1} t}{np}, \qquad b_\beta \coloneqq \frac{2^{\beta} |\Pi(T)|}{n^2p^2},
  \]
and note that  $|A_\alpha|\le a_\alpha$ and $|B_\beta|\le b_\beta$. Indeed, 
$2^{-\alpha}np|A_\alpha|\leq \sum_{z\in V}d_T(z)\leq 2t,$
and similarly for $|B_\beta|$. Moreover
  \[
    a_\alpha = \frac{2^{\alpha+1} t}{np} \ge \frac{t}{np} \ge \theta n^2p^2 \gg \frac{(\log n)^7}{p},\]
    using that $t\geq \theta n^3p^3$ as $T\in \cF$ (see Definition~\ref{def:upper-red-conds}) and the fact that $p\geq n^{-3/5}$. Similarly,  we have that  
    \[b_\beta= \frac{2^\beta |\Pi(T)|}{n^2p^2} \ge \frac{td}{24n^2p^2} \geq \frac{t^{3/2}p^{1/2}}{1200n^2p^2\log(2n^2p/t)}\geq  \frac{\theta^{3/2} n^{5/2}p^{3}}{1200\log(2n^2p/t)} \gg \frac{(\log n)^7}{p},  \]
    where we used~\eqref{eq:X2T-strong-LB} to lower bound $|\Pi(T)|$ as well as our lower bounds on $t $ and $p$. 
  Now as $T\in \cF(\beta;p)$ has property~\ref{item:eAB-concentration} of Lemma~\ref{lem:randomgraph}, see Definition~\ref{def:upper-red-conds},
  \[
    e_T(A_\alpha, B_\beta) \le |A_\alpha|\cdot |B_\beta| \cdot p + \frac{a_\alpha \cdot b_\beta \cdot p}{\log ^3 n} \le |A_\alpha|\cdot |B_\beta| \cdot p + \frac{2^{\alpha+\beta+1} t |\Pi(T)|}{n^3p^2\log ^3 n},
  \]
  for all $\alpha, \beta$ in our ranges of interest. 
  Therefore, plugging these upper bounds into~\eqref{eq:X2upper2}, we get 
  \begin{align*}
    X_2(\Pi(T)) & \le \sum_{\alpha=0}^{\log _2(np)} \sum_{\beta=0}^{2\log _2(np)+1} \left( |A_\alpha|\cdot |B_\beta| \cdot p + \frac{2^{\alpha+\beta+1} t |\Pi(T)|}{n^3p^2\log ^3 n} \right) \cdot 8n^3p^3 2^{-\alpha-\beta} \\
           & \le \frac{16|\Pi(T)|\cdot tp}{\log n} + 8p \cdot \left( \sum _{\alpha=0}^{\log _2(np)} |A_\alpha|\cdot 2^{-\alpha}np \right) \cdot \left( \sum _{\beta=0}^{2\log _2(np)+1} |B_\beta| \cdot 2^{-\beta} n^2 p^2 \right) \\
           \\
           & \le |\Pi(T)|\cdot tp + 8p\cdot \left(\sum_{z\in [n]}d_T(z)\right) \cdot  \left(\sum_{w\in [n]}d_\Pi(w)/2
          \right)\\
           & \le |\Pi(T)|\cdot tp + 8p\cdot 2t \cdot |\Pi(T)| \le 20|\Pi(T)| tp,
  \end{align*}
  establishing~\eqref{eq:X4T-upper} and completing the proof of this case, the claim and the lemma.
\end{claimproof}
\end{proof}

We will split our analysis in the  proof of Proposition~\ref{prop:unbalanced-upper} depending on whether $T\in \cF(\beta;p)$ outputs an $S(T)$ that satisfies~\ref{type:a}  or~\ref{type:b} when Lemma~\ref{lemma:DT_is_big} is applied to $T$.  If~\ref{type:a} holds, then we say that $T$ is of \emph{type}~\ref{type:a} and, likewise, if~\ref{type:b} holds for $S(T)$, we say that $T$ is of \emph{type}~\ref{type:b}.  Our next lemma states that type~\ref{type:b} subgraphs can only occur when both $p$ and $e(T)$ are very close ot their minimal values.

\begin{lem}
  \label{lem:b consequences}
  Suppose $0<\beta \le 2^{-100}$,  $n^{-3/5}\ll p\ll n^{-1/2}$ and  $T\in \cF=\cF(\beta;p)$ is of type~\ref{type:b} with $t\coloneqq e(T)$. Then
  \begin{enumerate}[{label=(\roman*)}]
  \item
    \label{item:p-small}
    $p \le Cn^{-3/5} \log ^{1/5}n$ for some $C = C(\theta)$; 
  \item
    \label{item:eT-small}
    $t \ll n^3p^3\log n$;
  \item
    \label{item:lower log bd}
    $\log\left( \frac{2n^2p}{t}\right)\geq \frac{\log n}{6}$.
  \end{enumerate}
  
\end{lem}
\begin{proof}
  Suppose that $T\in \cF(\beta;p)$ is of type~\ref{type:b}, let $S \coloneqq S(T)\subseteq T$ be the graph output by Lemma~\ref{lemma:DT_is_big} when applied to $T$ and let $s \coloneqq e(S)\geq t/2$.  If~\ref{item:p-small} did not hold, then~\eqref{eq:mu S}, Lemma~\ref{lem:big pi} and the fact that $t\geq \theta n^3p^3$ (see Definition~\ref{def:upper-red-conds}) would imply that, if $C = C(\theta)$ is sufficiently large,
  \[ \frac{\mu(\cX_S)}{3}\geq \frac{|\Pi(S)|np^2}{4} \geq \frac{s^2p^2}{32} \geq \frac{t^2p^2}{128} \geq \frac{\theta tn^3p^5}{128} \ge \frac{C^5\theta t \log n}{128} \ge 10t\log\left(\frac{2n^2p}{t}\right),
  \]
  contradicting the fact that $S$ satisfies part~\ref{type:b} of Lemma~\ref{lemma:DT_is_big}.  
  Similarly, if~\ref{item:eT-small} did not hold, then, for some positive constant $c$, we would have, using that $p\gg n^{-3/5}$,
 \[
   \frac{\mu(\cX_S)}{3} \geq \frac{t^2p^2}{128}\geq \frac{ctn^3p^5\log n}{128}\gg t \log n,
 \]
 again contradicting the fact that $S$ satisfies part~\ref{type:b} of Lemma~\ref{lemma:DT_is_big}.  This means that both~\ref{item:p-small} and~\ref{item:eT-small} must hold.  Finally, the third assertion~\ref{item:lower log bd} follows from the other two.  Indeed, we have that  $p\ll  n^{-7/12}(\log n)^{-1/2}$ from  part~\ref{item:p-small} and so using~\ref{item:eT-small}, we get that
 \[\frac{2n^2p}{t}\geq \frac{2}{np^2\log n}\gg n^{1/6}.\qedhere\]
\end{proof}

Finally, we prove Proposition~\ref{prop:unbalanced-upper}, completing this section.

\begin{proof}[{Proof of Proposition~\ref{prop:unbalanced-upper}}]
  Let $\alpha = 1/6400$, $\beta \coloneqq 2^{-100}$ and let $\zeta \coloneqq \min\{\alpha \theta^2 / 96, \alpha \theta / (18C^5)\}$, where $C = C(\theta)$ is the constant from the statement of Lemma~\ref{lem:b consequences}.  Let $G\sim G_{n,p}$, let $\cF \coloneqq \cF(\beta;p)$ be the collection of graphs from Definition~\ref{def:upper-red-conds} and let $\cH$ be the collection of $n$-vertex graphs that satisfy properties~\ref{item:bound max deg}--\ref{item:eAB-concentration} of Lemma~\ref{lem:randomgraph}.  Further, let $\cF_{\ref{type:a}}, \cF_{\ref{type:b}}\subseteq \cF$ be the graphs of types~\ref{type:a} and~\ref{type:b}, respectively.  Now, for a graph $R \subseteq K_n$, let $A(R)$  be the event that $R = \varphi^{-1}(\red)$ for some $K_3$-free colouring $\varphi \colon E(G) \to \{\red, \blue\}$ of $G$ with fewer than $\zeta n^6p^8$ copies of $C_{rrbb}$.   As $\Pr[G\notin \cH]\ll 1$, due to Lemma~\ref{lem:randomgraph}, it suffices to show that a.a.s.\ the event $\bigcup\{A(R) : e(R) < \beta n^2p\} \cap \{G\in \cH\}$ does not occur.  To this end, we first claim that for any $R\subseteq K_n$ with $e(R)<\beta n^2p$, the event $A(R) \cap \{G\in \cH\}$ is empty unless $R \in \cF$.  To see this, suppose that $G\in \cH$ and the event $A(R)$ happens for some $R$ with $e(R)<\beta n^2 p$. As properties~\ref{item:bound max deg}--\ref{item:K210-count} and~\ref{item:eAB-concentration} of Lemma~\ref{lem:randomgraph} are all monotone decreasing, the graph $R \subseteq G \in \cH$ must satisfy condition~\ref{cond:Tb} of Definition~\ref{def:upper-red-conds}.  Moreover, the upper bound on $e(R)$ in condition~\ref{cond:Ta} of Definition~\ref{def:upper-red-conds} is satisfied by assumption.  As for the lower bound, property~\ref{item:many K3} of Lemma~\ref{lem:randomgraph} supplies a collection of at least $\theta n^3p^3$ edge-disjoint triangles in $G$ and each colours class of every $K_3$-free colouring of $G$ must contain at least one edge from each triangle in this collection.
Therefore, it remains to show that a.a.s.\ no event $A(T)\cap \{G\in \cH\}$ occurs with $T\in \cF=\cF_{\ref{type:a}}\cup \cF_{\ref{type:b}}$. We first deal with the type~\ref{type:a} graphs $T$.

\begin{clm} \label{clm:type a upper bound}
  For every $T\in \cF_{\ref{type:a}}$ with $e(T) = t$, we have $\Pr[A(T)]\leq p^t\left(\frac{t}{2n^2p}\right)^t$. 
\end{clm}
\begin{claimproof}
  Applying Lemma~\ref{lemma:DT_is_big}, we get some subgraph $S \coloneqq S(T)\subseteq T$ with $s\coloneqq e(S)\geq t/2$ and $\frac{\mu(\cX_S)^2}{\Delta(\cX_S)}\geq 10t\log \left( \frac{2n^2p}{t} \right)$.  Now,  let $\cY_T \subseteq \cX_S$ be the family of all copies of $K_{1,2}$ in $K_n\setminus T$  that form a 4-cycle with some copy of $K_{1,2}$ in $S$ and avoid the edges of $T$.  Using the notation of Lemma~\ref{janson}, let $\mu\coloneqq|\cY_T|p^2$ be the expected number of copies of $K_{1,2}$ in $\cY_T$ that appear in $G_{n,p}$ and let $\Delta\coloneqq\sum_{K,K'}p^{e(K\cup K')}$, where the sum goes over all pairs of copies $K,K'\in \cY_T$ such that $K \cap K' \neq \emptyset$. Note that  $\Delta\leq \Delta(\cX_S)$, as $\cY_T\subseteq \cX_S$, and we also have that 
  \[
    \mu=|\cY_T|p^2\geq |\Pi(S)|(n-2\Delta(T))p^2\geq \frac{2}{\sqrt{5}}|\Pi(S)|np^2 \ge \frac{2\mu(\cX_S)}{\sqrt{5}},
  \]
  appealing to~\eqref{eq:mu S} and the fact that $\Delta(T)\leq 2np\ll n$, as $T\in \cF$, here. 

  Now, let $B(T)$ be the event that fewer than $\mu/2$ copies of $K_{1,2}$ from $\cY_T$ appear in $G$.  By Lemma~\ref{janson},
  \[
    \Pr[B(T)]\leq \exp\left(-\frac{\mu^2}{8\Delta}\right)\leq \exp\left(-\frac{\mu(\cX_S)^2}{10\Delta(\cX_S)}\right)\leq \exp\left(-t\log\left( \frac{2n^2p}{t}\right)\right)=\left( \frac{t}{2n^2p}\right)^t.
  \]
  We claim that $A(T) \subseteq B(T)$.  Indeed, suppose that $A(T)$ occurs and fix some colouring $\varphi \colon E(G) \to \{\red, \blue\}$ with $\phi^{-1}(\red)=T$  and fewer than $\zeta n^6 p^8$ copies of $C_{rrbb} $.  Since  every copy $K$ of $K_{1,2}$ in $\cY_T$ that appears in $G$ is coloured blue (as its edges are not in $T$), it gives rise to at least one copy of $C_{rrbb}$, as $K$ forms a copy of $C_4$ with two edges of $S$ (which $\phi$ colours red).  This implies that $B(T)$ occurs as otherwise, by~\eqref{eq:mu S}, Lemma~\ref{lem:big pi} and the fact that $s \ge t/2 \ge \theta n^3p^3/2$, see Definition~\ref{def:upper-red-conds}, we would get
  \[
    \frac{\mu}{2}\geq \frac{\mu(\cX_S)}{4}\geq \frac{|\Pi(S)|np^2}{8}\geq \frac{s^2p^2}{64} \ge\zeta n^6p^8
  \]
  copies of $C_{rrbb}$, a contradiction.  Finally, as the event $A(T)$ occurring implies that $T\subseteq G$ and the events $T \subseteq G$ and $B(T)$ are independent (the copies of $K_{1,2}$ in $\cY_T$ avoid the edges of $T$), we have
  \[
    \Pr[A(T)]\leq \Pr[\{T \subseteq G\}\cap B(T)]\leq 
    \Pr[T\subseteq G]\cdot \Pr[B(T)]\leq p^t\left(\frac{t}{2n^2p}\right)^t,
  \]
  as required. 
\end{claimproof}

Using Claim~\ref{clm:type a upper bound} and appealing to a union bound, we thus have that 
  \begin{align*}
    \Pr\big[ A(T) \cap \{G\in \cH\} \text{ for some $T\in \cF_{\ref{type:a}}$} \big]
    & \le \sum _{T \in \cF_{\ref{type:a}}} \Pr\big[  A(T) \big] \le \sum _{T \in \cF_{\ref{type:a}}} p^{e(T)} \cdot  \left( \frac{e(T)}{2n^2p} \right) ^{e(T)} \\
    & \le \sum^{\beta n^2p}_{t = \theta n^3p^3} \binom{\binom{n}{2}}{t} \cdot p^t \cdot \left( \frac{t}{2n^2p} \right) ^t  \leq \beta n^2p\left( \frac{e}{4} \right) ^{n} \ll 1.
  \end{align*}

  It remains to consider the events $A(T) \cap \{G\in \cH\}$ for type~\ref{type:b} graphs $T\in \cF_{\ref{type:b}}$.  More precisely, we need to show that a.a.s.\ $G$ does not belong to the family $\cH'$ defined by
  \[
    \cH'\coloneqq\left\{H\in \cH: \exists  \phi \colon E(H)\rightarrow\{\red,\blue\} \text{ with } \phi^{-1}(\red)\in \cF_{\ref{type:b}} \text{ and fewer than } \zeta n^6p^8 \text{ copies of } C_{rrbb}\right\}.
  \]
  In order to do this, for each $H\in \cH'$, we will identify some $\ell=\ell(H)\in \NN$ and a pair of increasing sequences ${\bf{J}}(H)=(J_0,J_1,\ldots, J_\ell)$  and ${\bf{R}}(H)=(R_0,R_1,\ldots,R_\ell) $ of  subgraphs of  $H$;  we will refer to this pair as the \emph{stamp} of $H$ and denote it by ${\bf{S}}(H)=({\bf{J}}(H), {\bf{R}}(H))$. Our proof will then provide an upper bound on the probability that ${\bf S}(G)=\bf{S}$, for any given stamp $\bf{S}$, that is strong enough to survive a union bound over all possible stamps ${\bf{S}}$.  Before proceeding, we remark that, by Lemma~\ref{lem:b consequences}, we can assume that $p \le Cn^{-3/5} \log ^{1/5}n$, as otherwise $\cF_{\ref{type:b}}=\emptyset$ and so $\cH'=\emptyset$. 

  Now, for each $H\in \cH'$, we define $\ell(H)$ and construct the sequences  ${\bf{J}}={\bf J}(H)$ and ${\bf R}={\bf R}(H)$ (and hence the stamp ${\bf S}(H)$) by considering the following process: 
\begin{enumerate}
\item
  Fix some $K_3$-free colouring $\phi \colon E(H)\rightarrow\{\red,\blue\}$ with fewer than $\zeta n^6p^8$ copies of $C_{rrbb}$ and $ T\coloneqq\phi^{-1}(\red)\in \cF_{\ref{type:b}}$.
\item
  Choose a collection $\cC$ of $c_0\coloneqq \theta n^3p^3$ edge-disjoint triangles in $H$ (this is possible as $H\in \cH$ and so it satisfies property~\ref{item:many K3} of Lemma~\ref{lem:randomgraph}), fix $J_0$ to be the collection of edges featuring in $\cC$ and $R_0 \coloneqq J_0\cap T$ to be the collection of edges in $J_0$ which are coloured red by $\phi$. 
\end{enumerate}
At this point note that, for any $n$-vertex graph $R$ with $R_0 \subseteq R\subseteq T$, we have that $R\in \cF$. Indeed,  $|R|\geq |R_0|\geq c_0= \theta n^3p^3 $, as there is at least one red edge in each triangle in $\cC$, and the other conditions of Definition~\ref{def:upper-red-conds} follow from the fact that $R\subseteq T\in \cF_{\ref{type:b}}\subseteq \cF$.  We now continue to form our sequences.  We will  maintain that  $J_{i-1}\subseteq J_i$ for all $i\geq 1$ and define $R_i \coloneqq J_i\cap T$ to be the collection of edges in $J_i$ that are coloured red by $\phi$ (as is the case with $R_0\subseteq J_0$).
\begin{enumerate}
 \setcounter{enumi}{2}
\item \label{step: third}
  Suppose that $i \ge 0$ and that $J_i$ and $R_i$ have already been defined.    As $R_0\subseteq R_i\subseteq T$, we have that $R_i\in \cF$.  In particular, Lemma~\ref{lemma:DT_is_big} gives us $S_i \coloneqq S(R_i)$ such that $e(S_i)\geq e(R_i)/2$. Now fix
  \begin{equation} \label{eq:Midef}
    M_i\coloneqq\alpha \min\left\{  |\Pi(S_i)|np^2,e(R_i)\log\left(\frac{2n^2p}{e(R_i)} \right) \right\},
  \end{equation}
  recalling the definition of $\Pi(S_i)$ from Definition~\ref{def:S params}.
   Let $\cX(H,J_i,S_i)\subseteq \cX_{S_i}$ be the set of copies of $K_{1,2}$ in $H\setminus J_i$ that avoid the edges of $J_i$ and form a $4$-cycle with some copy of $K_{1,2}$ in $S_i$.  Moreover, let $\cZ$ be a largest collection of edge-disjoint copies of $K_{1,2}$ in $\cX(H,J_i,S_i)$. If $|\cZ|< M_i$ then terminate the process and fix  $\ell(H)\coloneqq i$. Otherwise, if $|\cZ|\geq M_i$, then choose some $\cZ'\subseteq \cZ$ with $|\cZ'|=M_i$, let $J_{i+1}$ be the graph obtained by adding the edges in copies of $K_{1,2}$ in $\cZ'$ to $J_i$ and let $R_{i+1}=J_{i+1}\cap T$. Repeat step~\ref{step: third} with $i+1$ replacing $i$.  
\end{enumerate}

We begin by collecting some observations about the process with the following claims. 

\begin{clm}
  \label{clm:SiTi-one}
  For every $i \in \{0, \dotsc, \ell-1\}$, we have
  \[
    e(R_{i+1}) - e(R_i) \ge \frac{e(J_{i+1}) - e(J_i)}{3} = \frac{2M_i}{3} \ge 2\zeta n^6p^8.
  \]
  Consequently, $e(R_i) \ge e(J_i) / 3 \ge 2\zeta i n^6p^8/3$ for every $i \in \{0, \dotsc, \ell\}$.
\end{clm}
\begin{claimproof}
  Since we have already shown that $e(R_0) \geq e(J_0)/3= c_0 = \theta n^3p^3$, as each triangle in $\cC$ must contain at least one red edge, the second assertion of the lemma easily follows from the first assertion.  We begin by showing that $M_i\geq 3 \zeta n^6p^8$, which follows from the stronger inequality
  \begin{equation}
    \label{eq:MiRi}
    \frac{M_i}{e(R_i)} \ge \frac{3\zeta n^3p^5}{\theta},
  \end{equation}
  as $e(R_i) \ge e(R_0) \ge \theta n^3p^3$.  To see that~\eqref{eq:MiRi} holds, consider two cases.  If $M_i$ is equal to the first term in~\eqref{eq:Midef}, this follows from Lemma~\ref{lem:big pi} as 
  \[
    \frac{|\Pi(S_i)|np^2}{e(R_i)} \ge \frac{e(S_i)^2p^2}{8e(R_i)} \ge \frac{e(R_i)p^2}{32} \ge \frac{e(R_0)p^2}{32} \geq \frac{\theta n^3 p^5}{32} \geq \frac{3\zeta n^3p^5}{\alpha\theta}.
  \]
  Similarly, if $M_i$ is equal to the second term, then, recalling that we have assumed that $n^3p^5\leq C^5 \log n$,
  \[
    \log\left(\frac{2n^2p}{e(R_i)}\right)\geq \log\left(\frac{2n^2p}{e(T)}\right)\geq  \frac{\log n}{6} \geq \frac{n^3p^5}{6C^5} \geq \frac{3\zeta n^3p^5}{\alpha \theta},
  \]
  where we also used Lemma~\ref{lem:b consequences}~\ref{item:lower log bd} and the fact that $R_i\subseteq T\in \cF_{\ref{type:b}}$. 

  Observe now that at least two thirds among the collection $\cZ'$ of $M_i$ copies of $K_{1,2}$, which are added to $J_i$ to get $J_{i+1}$, must contain an edge of $T$. Indeed, if this was not the case, then more than $M_i/3 \ge \zeta n^6p^8$ such copies would be coloured completely blue.  However, each of those forms a copy of $C_{rrbb}$ with two edges of $S_i\subseteq R_i\subseteq T$, which are all coloured red.  This contradicts the assumption that there are fewer than $\zeta n^6p^8$ copies of $C_{rrbb}$ in $H$.  Therefore, it must be that  $e(R_{i+1})-e(R_i)\geq 2M_i/3= (e(J_{i+1})-e(J_i))/3$.
\end{claimproof}

Since $n^6p^8 \gg n^3p^3$, by our assumption that $p \gg n^{-3/5}$, Claim~\ref{clm:SiTi-one} implies that $e(R_\ell) \ge \ell n^3p^3$.  Consequently, since $R_\ell\subseteq T\in \cF_{\ref{type:b}}$, we must have that $\ell=\ell(H)\leq \log n$.

\begin{clm} \label{clm:Miactual}
  For every $i \in \{0, \dotsc, \ell-1\}$, we have $M_i = \alpha  |\Pi(S_i)|np^2$.
\end{clm}
\begin{claimproof}
  If this was not the case, then, for some $i \in \{0, \dotsc, \ell-1\}$, we would have that
  \[
    M_i = \alpha e(R_i) \log\left(\frac{2n^2p}{e(R_i)}\right) \geq   \frac{ \alpha e(J_i)\log n}{18} \geq   \frac{\alpha e(J_0)\log n}{18} \geq  \frac{\alpha \theta n^3p^3\log n}{18},
  \]
  using Lemma~\ref{lem:b consequences}~\ref{item:lower log bd} and Claim~\ref{clm:SiTi-one}  here.
  But then, appealing again to Claim~\ref{clm:SiTi-one}, we have that
  \[
    e(R_\ell)\geq \frac{e(J_\ell)}{3} \geq \frac{e(J_{i+1})}{3} \ge \frac{2M_i}{3}\geq \frac{\alpha \theta n^3p^3\log n}{27},
  \]
  which is a contradiction, as $R_\ell\subseteq T\in \cF_{\ref{type:b}}$ and hence $e(R_\ell) \ll n^3p^3 \log n$ by Lemma~\ref{lem:b consequences}.
\end{claimproof}

We are finally in a position to bound the probability that ${\bf S}(G) = {\bf S}$ for each possible stamp ${\bf S}=({\bf J}, {\bf R})$.  Recall the definition of $\cH'$ and let
\[
  {\bm \cS}=\left\{{\bf S}(H)=({\bf J}(H), {\bf R}(H)): H\in \cH'\right\}
\]
be the set of stamps obtained by running the above process on all possible graphs $H\in \cH'$.  Further, for $0\leq k \leq \log n$, let ${\bm \cS}^k\coloneqq\{{\bf S}(H): H\in \cH', \ell(H)=k\}\subseteq {\bm \cS}$ be the stamps of length $k$.

\begin{clm} \label{clm:small prob}
  For any $0\leq k\leq \log n$ and all ${\bf S}=((J_0,\ldots,J_k),(R_0,\ldots,R_k))\in {\bm \cS}^k$, we have that
  \[
    \Pr[{\bf S}(G)={\bf{S}}]\leq p^{e(J_k)}\exp(-\zeta n^3p^5e(J_k)).
  \]
\end{clm}
Since ${\bf S}(G)$ is only defined for $G \in \cH'$, the event that ${\bf S}(G)={\bf S}$ implicitly implies that $G\in \cH'$. 

\begin{claimproof}[Proof of Claim~\ref{clm:small prob}]
  Fix some $0\leq k\leq \log n$ and  ${\bf S}=((J_0,\ldots,J_k),(R_0,\ldots,R_k))\in {\bm \cS}^k$ as in the statement of the claim.  Further, let $M\coloneqq M_k$, as in~\eqref{eq:Midef}, where $S\coloneqq S_k=S(R_k)$ is the graph obtained from Lemma~\ref{lemma:DT_is_big} with input $R_k$.  Now, let $Z$ be the largest size of a collection of edge-disjoint copies of $K_{1,2}$ in $G\setminus J_k$ that  form a $4$-cycle with some copy of $K_{1,2}$ in $S$.  We claim that ${\bf S}(G)={\bf{S}}$ implies both $J_k\subseteq G$ and $Z< M$.  Indeed, certainly any $H\in \cH'$ with ${\bf S}(H)={\bf{S}}$ must satisfy $J_k\subseteq  H$.  Further, if $Z\geq M$ and $G \in \cH'$, then  the process  defining ${\bf{S}}(G)$ would not terminate at step $k$, and thus $\ell(G)>k$, precluding ${\bf{S}}(G) = {\bf{S}}$.  Since the events $J_k\subseteq G$ and $Z< M$ are independent, as $Z$ depends only on $G \setminus J_k$, we have
  \[
    \Pr[{\bf S}(G)={\bf{S}}]\leq \Pr[\{J_k\subseteq G\}\cap \{Z<M\}]\leq \Pr[J_k\subseteq G]\Pr[Z<M]=p^{e(J_k)}\Pr[Z<M].
  \]
  It thus remains to bound $\Pr[Z<M]$.

  To this end, we let $\cY\subseteq \cX_{S}$ be the family of copies of $K_{1,2}$ in $K_n\setminus J_k$ that form a $4$-cycle with two edges in $S$.  As in the setting of Lemma~\ref{janson}, let $\mu\coloneqq|\cY|p^2$ be the expected number of copies of $K_{1,2}$ in $\cY$ that appear in $G$ and  $\Delta\coloneqq\sum_{K,K'}p^{e(K\cup K')}$, where the sum goes over all pairs of copies $K,K'\in \cY$ with $K \cap K'\neq \emptyset$.
  Note that $\Delta\leq \Delta(\cX_S)$, as $\cY\subseteq \cX_S$, and we also have that 
  \[
    \mu=|\cY|p^2\geq |\Pi(S)|(n-2\Delta(J_k))p^2\geq \frac{1}{\sqrt{2}}|\Pi(S)|np^2 \ge \frac{\mu(\cX_S)}{\sqrt{2}},
  \]
using~\eqref{eq:mu S} and the inequality $\Delta(J_k)\leq 2np\ll n$, which holds as $J_k\subseteq H$ for some $H\in \cH'\subseteq \cH$. Note also that $\mu(\cX_S)/3\geq |\Pi(S)|np^2/4$.  Therefore, by Lemma~\ref{lemma:DT_is_big} and our choice of $\alpha$,
\[
  D\coloneqq\frac{\mu^2}{800\Delta}\geq \frac{\mu(\cX_S)^2}{1600\Delta(\cX_S)}\geq \min\left\{\frac{1}{6400}|\Pi(S)|np^2,\frac{1}{160}e(R_k)\log\left(\frac{2n^2p}{e(R_k)} \right)\right\}\geq M.
\]
In the notation of Corollary~\ref{cor:maxdisfam}, letting $\cA$ be the graph with vertex set $\Gamma=E(K_n)$ whose edges encode copies of $K_{1,2}$ in $\cY$,  we have that  $Z = \nu(\cA[\Gamma_p])$, the size of the largest matching in $\cA$.  In particular, we may apply Corollary~\ref{cor:maxdisfam} to conclude that $\Pr[Z<M]\leq \Pr[Z<D]\leq \exp(-D)\leq \exp (-M)$; thus, the claim will follow after showing that $M\geq \zeta n^3p^5e(J_k)$. This follows from Claim~\ref{clm:SiTi-one} and inequality~\eqref{eq:MiRi}:
\[
  \frac{M_k}{e(J_k)} \ge \frac{M_k}{3e(R_k)} \ge \frac{\zeta n^3p^5}{\theta} \ge \zeta n^3p^5.\qedhere
\]
\end{claimproof}
  
It remains to perform a union bound over all possible stamps ${\bf S}\in{\bm \cS}$. We have that
\begin{align*}
  \Pr[G\in \cH']=\sum_{{\bf S}\in{\bm \cS}}\Pr[{\bf S}(G)={\bf S}]= \sum^{\log n}_{k=0}\underbrace{\sum_{{\bf S}\in{\bm \cS^k}}\Pr[{\bf S}(G)={\bf S}]}_{\Sigma^k}.
  \end{align*} 
  In order to get a grasp on $\Sigma^k$, consider the following random process that constructs increasing random sequences ${\bf J^*}=(J^*_0,\dotsc,J^*_k)$ and ${\bf{R}^*}=(R_0^*,\dotsc,R_k^*)$ of subgraphs of $K_n$:

  \medskip
  
  \begin{algorithm}[H]
    \SetAlgoNoEnd
    Let $\cC^*$ be a uniformly chosen random family of $c_0 = \theta n^3p^3$ triangles in $K_n$.

    Let $J_0^*$ be the union of $\cC^*$ and let $R_0^* \subseteq J_0^*$ be a~uniformly chosen random subset.
    
    \For{$i = 0, \dotsc, k-1$}{    
      \eIf{$R_i^* \in \cF$}{
        Let $S_i^* \coloneqq S(R_i^*)$, the graph output by Lemma~\ref{lemma:DT_is_big}.

        Let $\cX^* \subseteq \cX_{S_i^*}$ be a uniformly chosen random subset with $\alpha |\Pi(S_i^*)|np^2$ elements.

        Let $L^*$ be the union of all copies of $K_{1,2}$ in $\cX^*$.

        Let $Q^*$ be a uniformly chosen random subset of $L^*$.

        Let $J_{i+1}^* \coloneqq J_i^* \cup L^*$ and $R_{i+1}^* \coloneqq R_i^* \cup Q^*$.
      }{
        Let $J_{i+1}^* \coloneqq J_i^*$ and $R_{i+1}^* \coloneqq R_i^*$.
      }
    }
  \end{algorithm}

  \medskip

  The key observation is that, for any ${\bf S}=((J_0,\ldots,J_k),(R_0,\ldots,R_k))\in {\bm \cS}^k$,
  \begin{equation}
    \Pr\big(({\bf{J}^*},{\bf{R}^*})={\bf S}\big) \ge q^*({\bf S})\coloneqq\left(\dbinom{\binom{n}{3}}{c_0} 2^{3c_0}{ \prod_{i=0}^{k-1} }\left( \dbinom{|\Pi(S_i)|n}{\alpha |\Pi(S_i)|np^2} 2^{2\alpha |\Pi(S_i)|np^2}\right)\right)^{-1},
  \end{equation}
  where, for each $i$, we denote by $S_i=S(R_i)$ the  graph output by Lemma~\ref{lemma:DT_is_big} with input $R_i$.  Since clearly $\sum_{{\bf{S}}\in {\bm \cS}^k} q^*({\bf S})\leq 1$, we have
  \[
    \Sigma^k \le \frac{\sum_{{\bf{S}}\in {\bm \cS}^k}
\Pr[{\bf S}(G) = {\bf S}]}{\sum_{{\bf{S}}\in {\bm \cS}^k} q^*({\bf S})}
\leq  \max \big\{ \Pr[{\bf S}(G) = {\bf S}] \cdot q^*({\bf S})^{-1} : {\bf S} \in {\bm S}^k\big\}.
  \]
  Now, fix some  ${\bf S}=((J_0,\ldots,J_k),(R_0,\ldots,R_k))\in {\bm \cS}^k$ and let $S_i\coloneqq S(R_i)$ and  $M_i\coloneqq\alpha |\Pi(S_i)|np^2$ for each $i \in \{0, \dotsc, k-1\}$. By Claims~\ref{clm:Miactual} and~\ref{clm:small prob}, we have that $e(J_k)=3c_0+\sum_{i=0}^{k-1}2M_i$ and thus
\begin{align*}
  \Pr[{\bf S}(G)={\bf S}] \cdot q^*({\bf{S}})^{-1}
  &\leq
    p^{e(J_k)} \exp(-\zeta n^3p^5 e(J_k)) \cdot \left(\frac{8en^3/6}{\theta n^3p^3}\right)^{c_0} \prod_{i=0}^{k-1}  \left(\frac{4e|\Pi(S_i)|n }{\alpha |\Pi(S_i)|np^2}\right)^{M_i} \\
  & \leq \exp({-\zeta n^3p^5 e(J_k)}) \cdot \left(\frac{4}{\theta}\right)^{c_0} \prod_{i=0}^{k-1} \left( \frac{4e}{\alpha}\right)^{M_i} \\
  &\leq \exp\left(- e(J_k) \cdot \left(\zeta n^3p^5 + \log \theta + \log \alpha \right)\right) \le \exp(-e(J_k)),
\end{align*}
as $n^3p^5\gg 1$. Since $e(J_k)\geq c_0\gg n$, we conclude that $\Sigma^k \le e^{-n}$ and, consequently, $\Pr(G \in \cH') \ll 1$.
\end{proof}

\bibliographystyle{abbrv}
\bibliography{Biblio.bib}

\appendix

\section{Proof of Corollary~\ref{cor:maxdisfam}} \label{app:maxdisfam}

Our proof follows the approach of~\cite[Lemma 8.4.2]{alonspencer}. We denote
\[
  \tilde{\Delta} \coloneqq \Delta - \mu =  \sum_{i\neq j}\mathds{1}[{A_i\cap A_j\neq \emptyset}] \cdot p^{|A_i \cup A_j|},
\]
where the sum is taken over ordered pairs $(i,j)\in [m]^2$. We consider two cases, depending on the ratio $\tilde\Delta/\mu$.

  \medskip
  \noindent
  \textit{Case 1.} $\tilde{\Delta} \le \mu/3$.
  For an index $i\in [m]$, let $\delta_i\coloneqq\sum_{j\neq i}\mathds{1}[A_i\cap A_j\neq \emptyset] \cdot p^{\left| A_j\setminus A_i \right|}$ and call  $i\in [m]$ \emph{good} if $\delta_i \le \frac{4\tilde{\Delta}}{\mu}$. Furthermore,  let $\Lambda \subseteq [m]$ be the set of good indices and let $\mu_g \coloneqq \sum_{i\in \Lambda}p^{|A_i|}$. We will look for a matching of size $D^* \coloneqq \tfrac{\mu^2}{50\Delta}$  in $\cA[\Gamma_p]$ only among the edges with good indices.

We will call a family $I\subseteq \Lambda$ \emph{disjoint} if $A_i\cap A_j=\emptyset$ for all $i\neq j\in I$.  The crucial observation is that, if the event $\nu(\cA[\Gamma_p])\le {D^*}$ occurs, then there must be some (possibly empty) disjoint family of indices $I \subseteq \Lambda$ of size at most ${D^*}$ whose all corresponding edges appear in $\Gamma_p$ that is maximal in the sense that no edge corresponding to the family $\Lambda_I \coloneqq \{ j\in \Lambda : \forall i\in I:A_i\cap A_j= \emptyset \}$ appears in $\Gamma_p$.  Denote the former event (that $A_i \subseteq \Gamma_p$ for all $i \in I$) by $Q_I$ and the latter event (that $A_j \nsubseteq \Gamma_p$ for all $i \in \Lambda_I$) by $M_I$.  Note that $\Pr[Q_I] = \prod_{i \in I} p^{|A_i|}$ and, crucially, that $Q_I$ and $M_I$ are independent.

In order to estimate the probability of $M_I$, observe first that ignoring bad indices does not have a big effect on the expected number of sets appearing in $\Gamma_p$ and we have that $\mu_g\geq 3\mu/4$. Indeed, if this were not the case, then we would have that 
\[
  \tilde\Delta = \sum_{i\in[m]}p^{|A_i|}\delta_i> \frac{4\tilde{\Delta}}{\mu}\sum_{i\in [m]\setminus \Lambda}p^{|A_i|}= \frac{4\tilde{\Delta}}{\mu}\cdot (\mu-\mu_g)> \tilde{\Delta},
\]
a contradiction. We further note that 
\[
  \mu_I\coloneqq\sum_{j\in \Lambda_I}p^{|A_j|}\geq \mu_g-\sum_{i\in I}\sum_{j\in \Lambda}\mathds{1}[A_i\cap A_j\neq \emptyset] \cdot p^{\left| A_j \right|}
  \geq \mu_g-\sum_{i\in I} (1+\delta_i)\geq \mu_g-|I|\left(1+\frac{4\tilde\Delta}{\mu}\right).
\]
In particular, if $|I| \le \tfrac{\mu^2}{16\Delta}$, then
\[
  \mu_I \ge \mu_g - \frac{\mu^2}{16\Delta} \cdot \left(1 + \frac{4\tilde\Delta}{\mu}\right) = \mu_g - \frac{\mu^2}{16\Delta} \cdot \left(\frac{4\Delta}{\mu} - 3\right) \ge \mu_g - \frac\mu4 \ge \frac{\mu}{2}.
\]
We also clearly have that
\[
  \Delta_I\coloneqq\sum_{i,j\in \Lambda_I} \mathds{1}[A_i \cap A_j \neq \emptyset] \cdot \mathbb{E}[I_iI_j]\le \Delta.
\]
Hence, by Lemma~\ref{janson}
\[
\Pr \left[ M_I \right] \le \exp \left( -\frac{\mu_I^2}{2\Delta_I} \right) \le \exp \left( -\frac{\mu^2}{8\Delta} \right) .
\]
This gives the following estimate:
\begin{align*}
  \Pr [\nu(\cA[\Gamma_p])\le {D^*}]
  & \le \sum_{|I| \le D^*} \Pr[Q_I \wedge M_I] =\sum _{k = 0}^{D^*} \sum _{\substack{I\subseteq \Lambda \\|I|=k}} \Pr[M_I] \cdot \Pr[Q_I] \\
& \le \exp \left( -\frac{\mu^2}{8\Delta} \right) \cdot\sum _{k=0}^{D^*} \sum _{\substack{I\subseteq \Lambda \\|I|=k}} \prod _{i\in I} p^{|A_i|} \le \exp \left( -\frac{\mu^2}{8\Delta} \right) \cdot  \sum_{k=0}^{D^*} \frac{1}{k!} \left( \sum_{i\in \Lambda} p^{|A_i|} \right) ^k \\
  & \le  \exp \left( -\frac{\mu^2}{8\Delta} \right) \cdot \sum_{k=0}^{D^*} \frac{\mu^k}{k!}
    \le  \exp \left( -\frac{\mu^2}{8\Delta} \right) \cdot \left( \frac{e\mu}{{D^*}} \right) ^{D^*},
\end{align*}
where the last inequality is the well-known concentration inequality for the Poisson distribution that states that, if $x\le \mu $, then $\Pr[\mathrm{Poisson}(\mu) \leq x] \le \left(\tfrac{e\mu}{x}\right)^xe^{-\mu}$.  Finally, since we assume that $\Delta = \mu + \tilde\Delta \le 4\mu/3$, we conclude that
\[
  \Pr [\nu(\cA[\Gamma_p])\le {D^*}] \le \exp\left(D^* \cdot \left(-\frac{25}{4} +\log \frac{50e\Delta}{\mu} \right)\right) \le \exp(-D^*).
\]

  \medskip
  \noindent
  \textit{Case 2.} $\tilde{\Delta} > \mu/3$.
  In this case, we simply find a subset of indices that satisfies the assumption of Case 1.  For a set $S \subseteq [m]$ of indices, denote
  \[
    \mu_S \coloneqq \sum_{i \in S} p^{|A_i|} \qquad \text{and} \qquad \tilde{\Delta}_S \coloneqq  \sum_{i\neq j\in S}\mathds{1}[A_i\cap A_j\neq \emptyset] \cdot p^{|A_i \cup A_j|}.
  \]
  It is clearly enough to show that there exists a set $S$ with ${\tilde\Delta}_S \le \mu_S/3$ and $D_S^* \coloneqq \tfrac{\mu_S^2}{50(\mu_S+\tilde{\Delta}_S)} \ge D$.  Let $S$ be a random subset of $[m]$ obtained by independently retaining each element with probability $q \coloneqq \tfrac{\mu}{6\tilde\Delta}$. Then
  \[
    \Ex[ \mu_S - 3\tilde{\Delta}_S ] = q\mu -3q^2\tilde{\Delta} = \frac{\mu ^2}{12 \tilde{\Delta}},
  \]
  so there is an $S$ that satisfies both $\mu_S \ge \frac{\mu^2}{12\tilde{\Delta}}$ and $\tilde{\Delta}_{S} \le \mu_S/3$ and thus also
  \[
    D^*_S = \frac{\mu _S^2}{50(\mu _S + \tilde{\Delta}_S)} \ge \frac{3\mu_S}{200} \ge \frac{\mu ^2}{800 \tilde{\Delta}} \ge \frac{\mu ^2}{800\Delta}=D.
  \]
  In particular, arguing as in Case~1, with $\cA$ replaced by $\mathcal{A}_S \coloneqq \{A_i \in \cA : i \in S\}$, we obtain,
\[
  \Pr [\nu (\mathcal{A}[\Gamma_p]) \le D]\le \Pr [\nu (\mathcal{A}_S[\Gamma_p]) \le D] \le \Pr [\nu (\mathcal{A}_S[\Gamma_p]) \le D^*_S] \le \exp (-D^*_S) \le \exp(-D).\]

\section{Derivation of Theorem~\ref{thm:containers}}
\label{app:containers}

Here we derive our container lemma, Theorem~\ref{thm:containers}, which we restate for convenience. 
\containersthm*

Theorem~\ref{thm:containers} is  a slight reformulation of the following theorem of Saxton and Thomason. 

\begin{thmtool}[{\cite[Corollary~3.6]{SaxTho15}}]
  \label{thm:containers-Saxton-Thomason}
  For every $2\le k\in \NN$ and  $\eps>0$, there exists an integer $s\in \NN$ such that the following holds.
  Suppose that a nonempty $k$-uniform hypergraph $\cH$ on vertex set $V$ and $\tau \in (0,1/2)$ satisfy
  \[
    \delta(\cH, \tau) \coloneqq 2^{\binom{k}{2}-1} \sum_{j=2}^k 2^{-\binom{j-1}{2}} \delta_j(\cH, \tau) \le \frac{\eps}{12k!},
  \]
  where
  \[
    \delta_j(\cH, \tau) \coloneqq \frac{\tau^{1-j}}{k e(\cH)} \cdot \sum_{v \in V} \max\{d_\cH(T) : v \in T \subseteq V \text{ and } |T| = j\}.
  \]
  Then there exists a function $C \colon \cP(V)^s \to \cP(V)$ such that, letting
  \[
    \cT \coloneqq \big\{(T_1, \dotsc, T_s) \in \cP(V)^s : |T_i| \le s\tau|V| \text{ for all } i \in [s] \big\},
  \]
  we have:
  \begin{enumerate}[{label=(\alph*)}]
  \item
    \label{item:containers-ST-1}
    For every set $I \subseteq V$ satisfying $e(\cH[I]) \le 24\eps k! k \tau^k e(\cH)$, there exists $T = (T_1, \dotsc, T_s) \in \cT \cap \cP(I)^s$ with $I \subseteq C(T)$.
  \item
    \label{item:containers-ST-2}
    For every $T \in \cT$, the set $C(T)$ induces at most $\eps e(\cH)$ edges in $\cH$.
  \end{enumerate}
\end{thmtool}

\begin{proof}[Derivation of Theorem~\ref{thm:containers} from Theorem~\ref{thm:containers-Saxton-Thomason}]
  Let $\cH$ be a nonempty $k$-uniform hypergraph with vertex set $V$ and let $\eps>0$ and $K\in\NN$. We set $s \coloneqq s_{\ref{thm:containers-Saxton-Thomason}}(k, \eps)$ to be the constant output by Theorem~\ref{thm:containers-Saxton-Thomason} with input $k$ and $\eps$ and let
  \[
    L \coloneqq \left\lceil \frac{12k!2^{\binom{k}{2}}K}{\eps} \right\rceil,
    \qquad
    t \coloneqq 2Ls^2,
    \qquad
    \text{and}
    \qquad
    \delta \coloneqq 24 \eps k! k L^k.
  \]
  Suppose that the maximum degrees of $\cH$ satisfy the assumptions of the theorem for some $\tau\le 1/t$.  Note that, for every $j \in \{2, \dotsc, k\}$,
  \[
    \delta_j(\cH, L \tau) \leq  \frac{(L \tau)^{1-j}}{ke(\cH)} \cdot v(\cH) \Delta_j(\cH) \le \frac{KL^{1-j}}{k}
  \]
  and thus, as $L \ge 2$,
  \[
    \delta(\cH, L\tau) \le 2^{\binom{k}{2}-1} \cdot \sum_{j=2}^{k}\frac{KL^{1-j}}{k} \le \frac{2^{\binom{k}{2}}K}{L} \le \frac{\eps}{12k!}.
  \]
  Consequently, Theorem~\ref{thm:containers-Saxton-Thomason}, invoked with $\tau_{\ref{thm:containers-Saxton-Thomason}} = L\tau$, implies that there exist a function $C \colon \cP(V)^s \to \cP(V)$ that satisfies conditions~\ref{item:containers-ST-1} and~\ref{item:containers-ST-2} of that theorem. Here, we used that $\tau \le 1/t$ implies that $L\tau<1/2$. Now  define $f \colon \cP(V)^t \to \cP(V)$ by letting
  \[
    f(S_1, \dotsc, S_t) \coloneqq C\big(S_1 \cup \dotsb \cup S_{t/s}, \dotsc, S_{(s-1)t/s+1} \cup \dotsb \cup S_t\big).
  \]
If $I \subseteq V$ satisfies
  \[
    e(\cH[I]) \le \delta \tau^k e(\cH) \le 24\eps k!k (L \tau)^k e(\cH),
  \]
  then there are $T_1, \dotsc, T_s \subseteq I$, with $|T_i| \le s L\tau |V|$ for each $i$, such that $I \subseteq C(T_1, \dotsc, T_s)$.  This gives the assertion of part~\ref{item:almostindep} of the theorem, as we may partition each $T_i$ into $t/s$ sets $S_{t(i-1)/s+1}, \dotsc, S_{t(i+1)/s}$, each of size at most $\lceil s/t \cdot sL\tau |V| \rceil \le \tau |V|$. Part~\ref{item:containersparse} of the theorem follows directly from our definition of $f$ and part~\ref{item:containers-ST-2} of Theorem~\ref{thm:containers-Saxton-Thomason}. 
\end{proof}

\end{document}